\documentclass[12pt]{amsart}
\usepackage{amssymb,amsmath,amsthm}
\usepackage[all,cmtip]{xy}
\usepackage{graphicx}
\usepackage{caption,subcaption}
\usepackage{adjustbox}
\usepackage{arydshln}
\usepackage{MnSymbol}
\usepackage{mathtools}
\usepackage{mathrsfs}
\usepackage{hyperref}
\usepackage[msc-links, backrefs]{amsrefs}
\hypersetup{
  colorlinks   = true,	
  urlcolor     = blue,	
  linkcolor    = blue,	
  citecolor   = red	
}
\calclayout
\newcommand{\beq}{\begin{small} \begin{equation}}
\newcommand{\eeq}{\end{equation} \end{small}}
\newcommand{\beqn}{\begin{small} \begin{equation*}}
\newcommand{\eeqn}{\end{equation*} \end{small}}
\newcommand{\gerade}{\ell}
\newtheorem{theorem}{Theorem}[section]
\newtheorem{proposition}[theorem]{Proposition}
\newtheorem{corollary}[theorem]{Corollary}
\newtheorem{lemma}[theorem]{Lemma}
\newtheorem{remark}[theorem]{Remark}
\numberwithin{theorem}{section} \numberwithin{equation}{section}
\newcommand\scalemath[2]{\scalebox{#1}{\mbox{\ensuremath{\displaystyle #2}}}}
\begin{document}
\title[Isogenies of certain K3 surfaces of rank 18]{Isogenies of certain K3 surfaces of rank 18}
\author{Noah Braeger}
\address{Dept.\!~of Mathematics \& Statistics, Utah State University, Logan, UT 84322}
\email{noah.braeger@usu.edu}
\author{Adrian Clingher}
\address{Dept.\!~of Mathematics \& Statistics, University of Missouri - St. Louis, MO 63121}
\email{clinghera@umsl.edu}
%
\author{Andreas Malmendier}
\address{Dept.\!~of Mathematics, University of Connecticut, Storrs, Connecticut 06269\\
\hspace*{0.3cm} Dept.\!~of Mathematics \& Statistics, Utah State University, Logan, UT 84322}
\email{andreas.malmendier@uconn.edu}
%
\author{Shantel Spatig}
\address{Dept.\!~of Mathematics \& Statistics, Utah State University, Logan, UT 84322}
\email{shantel.black@aggiemail.usu.edu}
%
\begin{abstract}
We construct geometric isogenies between three types of two-parameter families of K3 surfaces of Picard rank 18. One is the family of Kummer surfaces associated with Jacobians of genus-two curves admitting an elliptic involution, another is the family of Kummer surfaces associated with the product of two non-isogeneous elliptic curves, and the third is the twisted Legendre pencil.  The isogenies imply the existence of algebraic correspondences between these K3 surfaces and prove that the associated four-dimensional Galois representations are isomorphic. We also apply our result to several subfamilies of Picard rank 19. The result generalizes work of van~\!Geemen and Top in \cite{MR2214473}.
\end{abstract}
\maketitle
In \cite{MR1890996}, Ahlgren, Ono and Penniston computed the $L$-series and zeta function, over the finite field $\mathbb{F}_p$, for certain families of K3 surfaces of Picard rank 19. The K3 families in question are associated with a specific one-parameter elliptic curve family $\mathcal{E}_t$ and include the twisted Legendre pencil $\mathcal{X}_t$, as well as the Kummer surfaces $\operatorname{Kum}(\mathcal{E}_t \times \mathcal{E}_t)$. The Ahlgren-Ono-Penniston results imply that two three-dimensional $\ell$-adic Galois representations, associated with $\mathcal{X}_t$ and $\operatorname{Kum}(\mathcal{E}_t \times \mathcal{E}_t)$, respectively, are isomorphic, where $\ell$ is any prime number different from $p$. In turn, the above isomorphism determines a Galois invariant class in a certain \'etale cohomology group of $\mathcal{X}_t \times \operatorname{Kum}(\mathcal{E}_t \times \mathcal{E}_t)$. The Tate conjecture predicts then the existence of an algebraic cycle on $\mathcal{X}_t \times \operatorname{Kum}(\mathcal{E}_t \times \mathcal{E}_t)$, with coefficients in $\mathbb{Q}_\ell$, that realizes this Galois invariant class. The explicit algebraic correspondence determining the predicted cycle was 
subsequently constructed by van Geemen and Top \cite{MR2214473}. This correspondence is obtained as the graph of a dominant rational map from $\mathcal{X}_t$ to $\operatorname{Kum}(\mathcal{E}_t \times \mathcal{E}_t)$, defined over a finite extension of $\mathbb{Q}(t)$.
\par The present article indicates a generalization of the above work. We extend the above results to certain two-parameter families of K3 surfaces of generic Picard rank 18.  The K3 families in question include: Kummer surfaces associated with Jacobians of genus-two curves admitting an elliptic involution, Kummer surfaces associated with the product of two elliptic curves, and the two-parameter twisted Legendre pencil.
\par To begin with, let us briefly introduce some notation. Let $V$ be a smooth projective variety over $\mathbb{Q}$, with separable closure $\overline{\mathbb{Q}}$ and let $\mathrm{G}_\mathbb{Q} = \operatorname{Gal}{(\overline{\mathbb{Q}}/\mathbb{Q})}$ be the absolute Galois group. Fix a prime number $\ell$ and consider the $\ell$-adic cohomology groups, i.e., the cohomology groups of  the base extension $V_{\overline{\mathbb{Q}}}$ of $V$ to $\overline{\mathbb{Q}}$ with coefficients in the $\ell$-adic integers $\mathbb{Z}_\ell$, and corresponding scalars then extended to the $\ell$-adic numbers $\mathbb{Q}_\ell$. The \'etale cohomology theory determines then $\ell$-adic cohomology groups for the algebraic variety $V$, which carry many of the same properties that the usual singular cohomology groups have. In addition, however, the \'etale cohomology groups are representations of $\mathrm{G}_\mathbb{Q}$.  They also satisfy a form of Poincar\'e duality on non-singular projective varieties, and a  K\"unneth formula holds. In particular, when $V$ is a non-singular algebraic curve of genus $g$, $H^1_{\text{\'et}}$ is a free $\mathbb{Z}_\ell$-module of rank $2g$, dual to the Tate module of the Jacobian variety of $V$. In turn, $H^1_{\text{\'et}}$ is isomorphic to the singular cohomology (of the complex algebraic curve $V$) with $\mathbb{Z}_\ell$-coefficients.  Getting back to the general $V$ case, for $r>0$, a codimension-$r$ subvariety of $V$, defined over $\mathbb{Q}$, determines an element of the cohomology group $H^{2r}_{\text{\'et}}(V_{\overline{\mathbb{Q}}}, \mathbb{Q}_\ell(r))$ fixed by $\mathrm{G}_\mathbb{Q}$. Here, $\mathbb{Q}_\ell(r)$ denotes the $r$th Tate twist. In this context, the Tate conjecture states that the subspace of Galois invariant classes is in fact the $\mathbb{Q}_\ell$-vector space of algebraic cycles on $V$ of codimension~$r$. For more details on this aspect, we refer the reader to the excellent survey in \cite{MR3683625}. 
\smallskip
\par In this article, we consider several special families of K3 surfaces. First, the two-parameter twisted Legendre pencil $\mathcal{X}$ is, by definition, given by the family of K3 surfaces that are minimal resolutions of double covers of $\mathbb{P}^2 = \mathbb{P}(z_1, z_2, z_3)$ branched over a special union of six lines (and hence over a sextic curve). Specifically, consider the family of double sextics, given by
\beq
\label{eqn:X_intro}
  \mathcal{X}: \quad y^2 = z_1 z_2(z_1-z_2) (z_1-z_3) (z _3 - A z_2) (z_3 - B z_2) \,,
\eeq  
where the parameters $A, B$ satisfy $A\not = B$. The branch locus of a general member is the configuration of six lines $\{ \gerade_1, \dots, \gerade_6\}$ in $\mathbb{P}^2$  such that, up to permutation, the triple intersections $\{ p_{123} \} = \gerade_1 \cap \gerade_2 \cap \gerade_3$ and $\{ p_{145} \} =\gerade_1 \cap \gerade_4 \cap  \gerade_5$ consist of two distinct points $p_{123} \not = p_{145}$, and the configuration is generic otherwise. Notice that the configuration $(\mathbb{P}^2; \gerade_1, \dots, \gerade_6)$ is \emph{not} a Kummer plane, i.e., the general member of $\mathcal{X}$ is not a Kummer surface. We assume then:
\beq
 A = \left(\frac{1+ \lambda_2}{1- \lambda_2}\right)^2 \,,  \qquad  B = \left(\frac{1+ \lambda_3}{1- \lambda_3}\right)^2
\eeq
with $\lambda_2, \lambda_3 \in \mathbb{Q}$ with $\lambda_2, \lambda_3 \in \mathbb{P}^1 \backslash \{ 0, 1, \infty\}$ and $\lambda_2 \not= \lambda_3^{\pm 1}$. In the generic case the Picard rank of $ \mathcal{X}$ is 18.  We also remind the reader that the Picard group is generated by (curve) classes defined over $\mathbb{Q}$. Hence, one has an isomorphism of $\mathrm{G}_\mathbb{Q}$-representations:
\beq
 H^2_{\text{\'et}}\Big(\mathcal{X}_{\overline{\mathbb{Q}}}, \mathbb{Q}_\ell \Big) \ \cong \ \mathrm{T}^{(1)}_{\ell} \oplus \mathbb{Q}_\ell (-1)^{\oplus 18} \,,
\eeq 
where $\mathrm{T}^{(1)}_{\ell}$ is a certain four-dimensional $\ell$-adic representation.
\par We also consider Kummer surfaces $\operatorname{Kum}(\mathbf{A})$, of generic Picard rank 18, associated with certain principally polarized abelian surfaces $\mathbf{A}$. If we denote the minus identity involution denoted by $-\mathbb{I}$, then $\operatorname{Kum}(\mathbf{A})$ is, by definition,  the minimal resolution of the quotients $\mathbf{A}/\langle - \mathbb{I}\rangle$. We shall realize the quotients $\mathbf{A}/\langle - \mathbb{I}\rangle$ as double quadric or double sextic surfaces, as well as Jacobian elliptic surfaces. We shall refer to the minimal resolutions of the corresponding quadratic-twist surfaces as twisted Kummer surfaces $\operatorname{Kum}(\mathbf{A})^{(\varepsilon)}$ with twist factor $\varepsilon$.
\par One choice of abelian surface we shall consider is the product $\mathbf{A}=\mathcal{E}_1\times \mathcal{E}_2$, where the elliptic curves $\mathcal{E}_l$, for $l=1,2$, are not mutually isogenous.  Specifically, we consider the (families of) elliptic curves given by the equation
\beq
\label{eqn:EC_intro}
 \mathcal{E}_l : \quad y_l^2  z_l = x_l \big(x_l-z_l\big) \big(x_l- \Lambda_l z_l\big) \,,
\eeq
such that
\beq
\label{eqn:moduli_intro}
 \Lambda_1 \Lambda_2 = \dfrac{(\lambda_2 +\lambda_3)^2 -4 \lambda_2\lambda_3}{(1-\lambda_2)^2 (1-\lambda_3)^2} \,, \qquad
 \Lambda_1 + \Lambda_2 = - \dfrac{2(\lambda_2 +\lambda_3)}{(1-\lambda_2) (1-\lambda_3)} \,.
\eeq 
The Picard rank of  $\operatorname{Kum}(\mathcal{E}_1 \times \mathcal{E}_2)$ is generically 18.  Hence, one has an isomorphism of $\mathrm{G}_\mathbb{Q}$-representations, namely
\beq
 H^2_{\text{\'et}}\Big( \operatorname{Kum}( \mathcal{E}_1 \times \mathcal{E}_2)_{\overline{\mathbb{Q}}}, \mathbb{Q}_\ell \Big) \ \cong \ \mathrm{T}^{(2)}_{\ell} \oplus \mathbb{Q}_\ell (-1)^{\oplus 18} \,,
\eeq 
where $\mathrm{T}^{(2)}_{\ell}$ is a four-dimensional $\ell$-adic representation. Then, at the level of $\mathrm{G}_\mathbb{Q}$-representations, one has the isomorphism:
\beq
 \mathrm{T}^{(2)}_{\ell} \cong H^1_{\text{\'et}}\Big(\mathcal{E}_1, \mathbb{Q}_\ell \Big) \otimes H^1_{\text{\'et}}\Big(\mathcal{E}_2, \mathbb{Q}_\ell \Big) \,.
\eeq
\par Finally, we consider Kummer surfaces $\operatorname{Kum}(\mathbf{A})$ with abelian surface $\mathbf{A}$ realized as the 
Jacobian $\operatorname{Jac}{(\mathcal{C}_0)}$ of a smooth genus-two curve $\mathcal{C}_0$ admitting an elliptic involution. The curves $\mathcal{C}_0$ are
 given explicitly by
\beq
\label{eqn:eqn:G2_intro}
 \mathcal{C}_0: \quad Y^2 = X Z \big(X-Z\big) \, \big( X- \lambda_2\lambda_3 Z\big) \,  \big( X- \lambda_2 Z\big) \,  \big( X- \lambda_3 Z\big) \,.
\eeq
Due to the elliptic involution existing on the genus-two curve $\mathcal{C}_0$, the transcendental lattice of $\mathbf{A}=\operatorname{Jac}{(\mathcal{C}_0)}$ 
generically has rank four. As $\mathrm{G}_\mathbb{Q}$-representations, one has an isomorphism:
\beq
 H^2_{\text{\'et}}\Big( \operatorname{Kum}( \operatorname{Jac}{\mathcal{C}_0 })^{(\varepsilon)}_{\overline{\mathbb{Q}}}, \mathbb{Q}_\ell \Big) 
 \ \cong \ \mathrm{T}^{(3)}_{\ell} \oplus \mathbb{Q}_\ell (-1)^{\oplus 18} \,,
\eeq 
where $\mathrm{T}^{(3)}_{\ell}$ is a second four-dimensional $\ell$-adic representation and $\varepsilon$ is a certain quadratic-twist factor. 
\par The main result of this article can be stated as follows:
\begin{theorem}
Assume that  $\lambda_2, \lambda_3 \in \mathbb{Q}$ satisfy $\lambda_2, \lambda_3  \not \in \{ 0, 1 \}$, $\lambda_2 \not= \lambda_3^{\pm 1}$, and are generic otherwise. Further assume that $\Lambda_1, \Lambda_2$ satisfy 
\beq
\label{eqn:moduli_intro_thm}
 \Lambda_1 \Lambda_2 = \dfrac{(\lambda_2 +\lambda_3)^2 -4 \lambda_2\lambda_3}{(1-\lambda_2)^2 (1-\lambda_3)^2} \,, \qquad
 \Lambda_1 + \Lambda_2 = - \dfrac{2(\lambda_2 +\lambda_3)}{(1-\lambda_2) (1-\lambda_3)} \,,
\eeq 
and $\varepsilon =\lambda_2\lambda_3$. Then, there are explicit algebraic correspondences
\beq
\begin{split}
 \Gamma^{(1,2)} &\ \subset \ \operatorname{Kum}\big( \mathcal{E}_1 \times \mathcal{E}_2 \big) \times \mathcal{X} \,, \\
 \Gamma^{(2,3)} &\ \subset \ \mathcal{X} \times \operatorname{Kum}\big( \operatorname{Jac}{\mathcal{C}_0 }\big)^{(\varepsilon)} \,, \\
 \Gamma^{(3,1)} &\ \subset \ \operatorname{Kum}\big( \operatorname{Jac}{\mathcal{C}_0 }\big)^{(\varepsilon)} \times  \operatorname{Kum}\big( \mathcal{E}_1 \times \mathcal{E}_2 \big)  \,, \\
\end{split} 
\eeq
defined over $\mathbb{Q}$, which induce $\ell$-adic $\mathrm{G}_\mathbb{Q}$-representation isomorphisms:
\beq
 \left\lbrack \Gamma^{(i, j)} \right\rbrack: \quad  \mathrm{T}^{(i)}_{\ell}  \ \overset{\cong}{\longrightarrow} \ \mathrm{T}^{(j)}_{\ell} \,,
\eeq
for $i, j \in \{1, 2, 3 \}$ and $i \not = j$.
\end{theorem}
We note that the result above contains three special subfamilies, associated with the non-generic situation when the elliptic curves $\mathcal{E}_1$ and $\mathcal{E}_2$ are 
isomorphic, up to a quadratic twist. The $\mathrm{G}_\mathbb{Q}$-representations involved are then three-dimensional. The details on this case are included in Theorem~\ref{thm:ladic_2}. One of these specializations is exactly the case considered  in \cite{MR2214473}. To our knowledge, the other two cases have not appeared in the literature. In addition, in Theorem~\ref{thm:ladic_45} we prove analogous results for two other subfamilies of Picard rank 19 and 18 - one where the elliptic curves are two-isogenous to each other and one where one of the elliptic curves is constant with complex multiplication and j-invariant $12^3$. 
\par The article structure is as follows. In Section~\ref{sec:AbelianSurfaces}, we derive normal forms for Jacobians of genus-two curves admitting an elliptic involution, their elliptic quotient-curves and certain $(2,2)$-isogenous abelian surfaces. In Section~\ref{sec:KummerSurfaces}, we construct the Kummer surfaces associated with the Jacobian varieties of genus-two curves admitting an elliptic involution and the product abelian surfaces of two non-isogeneous elliptic curves. These Kummer surfaces are explicitly described as quartic hypersurfaces,  double quadric surfaces, double sextic surfaces, and Jacobian elliptic surfaces. As a side observation, we show that Legendre's glueing construction, i.e., the construction of a genus-two curve with involution from its elliptic-curve quotients over a small field of definition, has an exact analogue for quartic Kummer surfaces. This is 
discussed in Theorem~\ref{prop:KUMC0}. Using the results from Section~\ref{sec:AbelianSurfaces}, we are then able to construct explicit isogenies between the various Kummer surfaces involved; see Theorem~\ref{thm:isogenies_explicit}.  In Section~\ref{sec:LegendrePencil} we prove that the two-parameter twisted Legendre pencil fits into a Kummer sandwich of dominant rational maps, directly relating it to the Kummer surfaces from Section~\ref{sec:KummerSurfaces}. This geometrical aspect is presented in Theorem~\ref{thm:combined}. The rational maps involved are defined only over a finite field extension of the field of definition for the given varieties. The varieties involved have (projective) models defined over the rationals, but one has to choose the right quadratic-twists in order to obtain non-trivial rational correspondences between them. We achieve this goal by arranging for the pullbacks between natural generators of the corresponding $H^{2,0}$-parts in cohomology to be rational. Theorem~\ref{thm:ladic_1} is the main result of this article and proves that the aforementioned four-dimensional $\ell$-adic representations -- as obtained from the corresponding transcendental lattices of families of K3 surfaces of Picard rank 18 -- are isomorphic. Finally, Theorems~\ref{thm:ladic_2} and \ref{thm:ladic_45} discuss subfamilies of generic Picard rank 19, one of which recovers the main result from \cite{MR2214473}.
\enlargethispage{\baselineskip}
\subsection*{Acknowledgments}
The authors would like to thank the referees for their thoughtful comments and effort towards improving the manuscript.  The authors also thank Dr.~Muhammad Arjumand Masood for some helpful discussions about Jacobian elliptic functions during an early stage of this project. N.B. and S.S. would like to acknowledge the support from the Office of Research and Graduate Studies at Utah State University.  A.C. acknowledges support from a UMSL Mid-Career Research Grant. A.M. acknowledges support from the Simons Foundation through grant no.~202367.
\section{Abelian surfaces admitting an elliptic involution}
\label{sec:AbelianSurfaces}
Let $\mathcal{C}$ be an irreducible, smooth, projective curve of genus two, defined over the complex numbers $\mathbb{C}$.  Let $\mathcal{M}$ be the coarse moduli space of smooth curves of genus two. We denote by $[\mathcal{C}]$ the isomorphism class of $\mathcal{C}$, i.e.,  the corresponding point in $\mathcal{M}$.  For a smooth genus-two curve $\mathcal{C}$ given as a sextic $Y^2 = f_6(X,Z)$ in the weighted complex projective space $\mathbb{WP}_{(1,3,1)}=\mathbb{P}(X,Y,Z)$, we send three roots $(\lambda_4, \lambda_5, \lambda_6)$ to $(0, 1, \infty)$ to obtain an isomorphic curve in Rosenhain normal form, given by
\beq
\label{Eq:Rosenhain}
 \mathcal{C}: \quad Y^2 = X\,Z \, \big(X-Z\big) \, \big( X- \lambda_1 Z\big) \,  \big( X- \lambda_2 Z\big) \,  \big( X- \lambda_3 Z\big) \;.
\eeq
We denote the hyperelliptic involution on $\mathcal{C}$ by $\imath_\mathcal{C}$.  The Weierstrass points of $\mathcal{C}$ are the fixed points of $\imath_\mathcal{C}$, and they are $P_i: [X:Y:Z]=[\lambda_i:0:1]$ for $i=1, \dots,5$, and $P_6:  [X:Y:Z]=[1:0:0]$. The tuple $(\lambda_1, \lambda_2, \lambda_3)$ where the roots $\lambda_i$ are all distinct and different from $\{0, 1, \infty\}$ determines a point in the moduli space of genus-two curves with six marked Weierstrass points, denoted by $\mathcal{M}(2)$. The Jacobian variety $\operatorname{Jac}{(\mathcal{C})}$ is the moduli space of degree-zero line bundles on $\mathcal{C}$. It is the connected component of the identity in the Picard group of $\mathcal{C}$, hence an abelian surface. We refer to a genus-two curve as \emph{generic} if it is smooth and its Jacobian $\operatorname{Jac}{(\mathcal{C})}$ has no extra automorphisms.
\par The Siegel three-fold is the quasi-projective variety of dimension three, obtained from the Siegel upper half-plane $\mathbb{H}_2$ of degree two\footnote{By definition $\mathbb{H}_2$ is the set of two-by-two symmetric matrices over $\mathbb{C}$ whose imaginary part is positive definite} divided by the action of the modular transformations $\Gamma_2:= \operatorname{Sp}_4(\mathbb{Z})$, i.e., 
\beq
 \mathcal{A}_2 =  \mathbb{H}_2 / \Gamma_2 \;.
\eeq
Each $\tau= \bigl(\begin{smallmatrix} \tau_{11}& \tau_{12}\\ \tau_{12} & \tau_{22} \end{smallmatrix} \bigr) \in \mathbb{H}_2$ determines a complex abelian surface $\mathbf{A} = \mathbb{C}^2 / \Lambda$ obtained from the lattice $\Lambda = \mathbb{Z}^2 \oplus \tau \, \mathbb{Z}^2$ with the period matrix $(\mathbb{I}_2,\tau) \in \mathrm{Mat}(2, 4;\mathbb{C})$.  We consider two abelian surfaces $\mathbf{A}$  and $\widetilde{\mathbf{A}}$ isomorphic if and only if there is an element $M \in \Gamma_2$ such that $\widetilde{\tau} = M (\tau)$. The \emph{canonical principal polarization} of $\mathbf{A}$ is given by the positive definite hermitian form $\mathbf{h}$ on $\mathbb{C}^2$ such that $\alpha = \operatorname{Im} \mathbf{h}(\Lambda,\Lambda) \subset \mathbb{Z}$ where $\alpha$ is the Riemann form  $\alpha( x_1 + x_2 \tau, y_1  + y_2 \tau)=x_1^t\cdot y_2 - y_1^t\cdot x_2$ on  $\mathbb{Z}^2 \oplus \tau \, \mathbb{Z}^2$.  Such a hermitian form determines the class of a line bundle $\mathcal{L} \to \mathbf{A}$ in the N\'eron-Severi lattice $\mathrm{NS}(\mathbf{A})$.  Alternatively, $\tau$ determines a line bundle $\mathcal{L}$ with first Chern class $ \operatorname{Im}(\mathbf{h}) \in \mathrm{NS}(\mathbf{A}) \subset \bigwedge^2 H^1(\mathbf{A}, \mathbb{Z})$. A general fact from algebra asserts that one can always  choose a basis of $\Lambda$ such that $\alpha$ is given  by the matrix $\bigl(\begin{smallmatrix} 0&D\\ -D&0 \end{smallmatrix} \bigr)$ with $D=\bigl(\begin{smallmatrix}d_1&0\\ 0&d_2 \end{smallmatrix} \bigr)$ where $d_1, d_2 \in \mathbb{N}$, $d_1, d_2 \ge 0 $, and $d_1$ divides $d_2$. For a principal polarization one has $(d_1, d_2)=(1, 1)$.  Since transformations in $\Gamma_2$  preserve the Riemann form $\alpha$, it follows that the Siegel three-fold $\mathcal{A}_2$ is also the set of isomorphism classes of principally polarized abelian surfaces, i.e., abelian surfaces with a polarization of type $(1,1)$. We also define the subgroup $\Gamma_2(2) = \lbrace M \in \Gamma_2 | \, M \equiv \mathbb{I} \mod{2}\rbrace$  such that $\Gamma_2/\Gamma_2(2)\cong S_6$ where $S_6$ is the permutation group of six elements representing the permutations of the Rosenhain roots $(\lambda_1, \lambda_2, \lambda_3, 0, 1, \infty)$. Then, $\mathcal{A}_2(2)$ is the three-dimensional moduli space of principally polarized abelian surfaces with marked level-two structure.
\par For a principally polarized abelian variety $\mathbf{A}$ the line bundle $\mathcal{L}$ defining its  principal polarization is ample and satisfies $h^0(\mathcal{L}) = 1$. There exists an effective divisor $\Theta$ such that $\mathcal{L}= \mathcal{O}_{\mathbf{A}}(\Theta)$, uniquely defined only up to translations. The divisor $\Theta \in \mathrm{NS}(\mathbf{A})$ is called a \emph{theta divisor} associated with the polarization.  It is known that the abelian surface $\mathbf{A}$ is not the product of two elliptic curves if and only if $\Theta$ is an irreducible divisor. Torelli's theorem implies that the map sending a curve $\mathcal{C}$ to its Jacobian  $\mathrm{Jac}(\mathcal{C})$ is injective and defines a morphism $\mathcal{M} \hookrightarrow \mathcal{A}_2$. The complement of its image in $\mathcal{A}_2$ corresponds to abelian surfaces obtained as product of two complex elliptic curves. Moreover, $\mathcal{M} \hookrightarrow \mathcal{A}_2$ lifts to an injective morphism  $\mathcal{M}(2) \hookrightarrow \mathcal{A}_2(2)$ between the coarse moduli spaces of smooth curves of genus two with marked Weierstrass points and principally polarized abelian surfaces with marked level-two structures, respectively.

%
%
\subsection{Isogenies of Jacobian surfaces}
\label{ssec:isogenies}
Translations of the Jacobian $\mathbf{A}=\operatorname{Jac}{(\mathcal{C})}$ by order-two points in $\mathbf{A}[2]$ are isomorphisms of the Jacobian and map the set of two-torsion points to itself. Moreover, a \emph{G\"opel group} is a two-dimensional subspace $G \cong (\mathbb{Z}/2 \mathbb{Z})^2$ of $\mathbf{A}[2]$ such that $\mathbf{A}^\prime=\mathbf{A}/G$ is again a principally polarized abelian surface~\cite{MR2514037}*{Sec.~23}. The corresponding isogeny $\Psi: \mathbf{A} \rightarrow  \mathbf{A}^\prime$ between principally polarized abelian surfaces has as its kernel $G \leqslant \mathbf{A}[2]$ and is called a \emph{$(2,2)$-isogeny}.  
\begin{remark}
In general one defines an isogeny of type  $(n_1, \dots, n_k)$ with $k \le 4$ between abelian surfaces to be an isogeny with kernel isomorphic to $\oplus_{i=1}^k \mathbb{Z}/n_i\mathbb{Z}$. We will also consider compositions of two $(2,2)$-isogenies, leading to isogenies of type $(2, 2, 2, 2)$ (which is multiplication by two up to automorphisms), $(4, 2, 2)$, and $(4, 4)$.  This `calculus' for compositions of isogenies also leads to Hecke operators related to algebraic subvarieties on $\mathcal{A} \times \mathcal{A}$ parametrizing pairs $(\mathbf{A}, \mathbf{A}/G)$ for suitable finite subgroups $G$ such that $\mathbf{A}/G$ is still a principally polarized abelian surface; see work by Andrianov for example \cite{MR884891}. This generalizes the construction of modular curves $T_N= X_0(N) \subset \mathcal{A}_1 \times \mathcal{A}_1$; see also Remark~\ref{rem:Hoyt}.
\end{remark}
\par For the Jacobian of a genus-two curve, every nontrivial two-torsion point is a difference of Weierstrass points $P_i \in \mathcal{C}$ for $1 \le i \le 6$. In fact, the sixteen order-two points of $\mathbf{A}=\operatorname{Jac}{(\mathcal{C})}$ are obtained using the embedding of the curve into the connected component of the identity in the Picard group, i.e., $\mathcal{C} \hookrightarrow \operatorname{Jac}{(\mathcal{C})} \cong \operatorname{Pic}^0(\mathcal{C})$. We obtain the 15 elements $P_{i j}=[ P_i + P_j - 2 \, P_6]  \in \mathbf{A}[2]$ with $1 \le i < j < 6$  and set $P_0=P_{66}= [0]$. For $\{i, j, k, l, m, n\}=\{1, \dots, 6\}$,  the group law on $\mathbf{A}[2]$ is given by the relations
\beq
 \label{group_law}
    P_0 +  P_{ij} =  P_{ij}\,, \quad  P_{ij} +  P_{ij} =  P_{0}\,, \quad 
    P_{ij} + P_{kl} =  P_{mn}, \quad P_{ij} +
    P_{jk} =  P_{ik}\,.
\eeq
The  space $\mathbf{A}[2]$ of two-torsion points admits a symplectic bilinear form, called the \emph{Weil pairing} such that the two-dimensional, maximal isotropic subspace of $\mathbf{A}[2]$ with respect to the Weil pairing are the G\"opel groups.  The Weil pairing is induced by the pairing
\beq
 \langle [ P_i - P_j  ] ,[ P_k - P_l] \rangle =\#\{  P_{i}, P_{j}\}\cap \{ P_{k}, P_{l}\} \mod{2} \,.
\eeq
It is easy to check that there are exactly 15 inequivalent G\"opel groups, and they are of the form
\beq
\label{eqn:G_groups}
 \Big \lbrace P_0, P_{ij}, P_{kl}, P_{mn} \Big \rbrace
\eeq 
such that  $\{i, j, k, l, m, n\}=\{1, \dots, 6\}$.
\par For a general genus-two curve $\mathcal{C}$, one may ask whether the $(2,2)$-isogenous abelian surface $\mathbf{A}^\prime =\mathbf{A}/G$ satisfies $\mathbf{A}^\prime =\operatorname{Jac}{(\mathcal{C}^\prime)}$ for some smooth curve $\mathcal{C}^\prime$ of genus two. The geometric moduli relationship between the two curves of genus two was found by Richelot \cite{MR1578135}; see also \cite{MR970659}. If we choose for $\mathcal{C}$ a sextic equation $Y^2 = f_6(X, Z)$, then from any factorization $f_6 = A\cdot B\cdot C$ into three degree-two polynomials $A, B, C$ one obtains a new genus-two curve $\mathcal{C}^\prime$, given by
\beq
\label{Richelot}
 \mathcal{C}^\prime: \qquad \Delta_{ABC} \cdot Y^2 = [A,B] \, [A,C] \, [B,C] \,,
\eeq
where we have set $[A,B] \cdot Z= B \, \partial_X  A  - A \, \partial_X B$ with $\partial_X$ denoting the derivative with respect to $X$ and $\Delta_{ABC}$ is the determinant of $(A, B, C)$ with respect to the basis $X^2, XZ, Z^2$.  Notice that the Richelot construction is guaranteed to work only for a general curve of genus two. Later we will also consider isogenies from $\mathrm{Jac}(\mathcal{C}_0)$ for genus-two curves $\mathcal{C}_0$ with extra involution to the decomposable principally polarized abelian varieties that are products of two elliptic curves $\mathcal{E}_1 \times \mathcal{E}_2$.
\subsection{Humbert surfaces}
\par Sets of abelian surfaces with the same endomorphism ring form subvarieties within $\mathcal{A}_2$.  The endomorphism ring of principally polarized abelian surface tensored with $\mathbb{Q}$ is either a quartic CM field, an indefinite quaternion algebra, a real quadratic field or in the generic case $\mathbb{Q}$. Irreducible components of the corresponding  subsets in $\mathcal{A}_2$ have dimensions $0, 1, 2$ and are known as CM points, Shimura curves, and Humbert surfaces, respectively.  (Here, we are excluding the endomorphism algebras of abelian surfaces which are isogenous to $\mathcal{E}_1 \times \mathcal{E}_2$.) The Humbert surface $\mathscr{H}_{\Delta}$, with invariant $\Delta$, is the space of principally polarized abelian surfaces admitting a symmetric endomorphism with discriminant $\Delta$. It turns out that $\Delta$ always is a positive integer satisfying $\Delta \equiv 0, 1\mod{4}$ and uniquely determined $\mathscr{H}_{\Delta}$. In fact, $\mathscr{H}_{\Delta}$ is the image inside $\mathcal{A}_2$ under the projection of the rational divisor associated with the equation
\beq
\label{eqn:discriminant}
 a \, \tau_{11} + b \, \tau_{12} + c \, \tau_{22} + d\, (\tau_{12}^2 -\tau_{11} \, \tau_{22}) + e = 0 \;,
\eeq
with integers $a, b, c, d, e$ satisfying $\Delta=b^2-4\,a\,c-4\,d\,e$ and $\tau = \bigl(\begin{smallmatrix} \tau_{11}& \tau_{12}\\ \tau_{12} & \tau_{22} \end{smallmatrix} \bigr) \in \mathbb{H}_2$. The following was proven by Birkenhake and Wilhelm in \cite{MR1953527}:
\begin{theorem}
\label{prop:isogeny_Delta}
For $\delta \in \mathbb{N}$ the Humbert surface $\mathscr{H}_{\delta^2}$ is the locus of principally polarized abelian surfaces $(\mathbf{A}, \mathcal{L}) \in \mathcal{A}_2$ admitting an isogeny of degree $\delta^2$, given by
\beq
\label{eqn:Phi}
 \Phi: \quad \Big( \mathcal{E}_1 \times \mathcal{E}_2, \;  \mathcal{O}_{\mathcal{E}_1} (\delta) \boxtimes  \mathcal{O}_{\mathcal{E}_2} (\delta)  \Big)
 \ \longrightarrow \ \Big( \mathbf{A}, \mathcal{L} \Big) \,,
\eeq 
where $\mathcal{O}_{\mathcal{E}_l} (\delta)$ is a line bundle of degree $\delta$ on an elliptic curve $\mathcal{E}_l$ for $l= 1,2$.
\end{theorem}
The Humbert surface $\mathscr{H}_1$ is irreducible. In order to underscore this point, we shall refer to it as $\mathcal{H}_1$. It is the image, under the period projection, of the rational divisor associated to $\tau_{12}=0$. The Humbert surface $\mathscr{H}_4$ has two irreducible components, one of which is $\mathcal{H}_1$. We shall denote the second component by $\mathcal{H}_4$. This is the image of the divisor corresponding to $\tau_{11} - \tau_{22}=0$. Moreover,  the singular locus of $\mathcal{A}_2$ is given by 
$\mathscr{H}_4 = \mathcal{H}_1 \cup \mathcal{H}_4$.  As analytic spaces, $\mathcal{H}_1$ and $\mathcal{H}_4$ are each isomorphic to the Hilbert modular surface 
\beq
\label{modular_product2}
 \Big( (\mathrm{SL}_2(\mathbb{Z}) \times \mathrm{SL}_2(\mathbb{Z}) ) \rtimes \mathbb{Z}_2 \Big) \backslash \Big( \mathbb{H} \times \mathbb{H} \Big) \,.
\eeq
\par For a detailed introduction to Siegel modular forms relative to $\Gamma_2$, Humbert surfaces, and the Satake compactification of the Siegel modular threefold, we refer the reader to Freitag's book \cite{MR871067}. Igusa proved \cites{MR0229643, MR527830} that the ring of modular forms is generated by the Siegel modular forms $\psi_4$, $\psi_6$, $\chi_{10}$, $\chi_{12}$ and by one more cusp form $\chi_{35}$ of odd weight $35$  whose square is the following polynomial \cite{MR0229643}*{p.~\!849} in the even generators 
\beq
\label{chi_35sqr}
\begin{split}
\chi_{35}^2 & = \frac{1}{2^{12} \, 3^9} \; \chi_{10} \,  \Big(  
2^{24} \, 3^{15} \; \chi_{12}^5 - 2^{13} \, 3^9 \; \psi_4^3 \, \chi_{12}^4 - 2^{13} \, 3^9\; \psi_6^2 \, \chi_{12}^4 + 3^3 \; \psi_4^6 \, \chi_{12}^3 \\
& - 2\cdot 3^3 \; \psi_4^3 \, \psi_6^2 \, \chi_{12}^3 - 2^{14}\, 3^8 \; \psi_4^2 \, \psi_6 \, \chi_{10} \, \chi_{12}^3 -2^{23}\, 3^{12} \, 5^2\, \psi_4 \, \chi_{10}^2 \, \chi_{12}^3  + 3^3 \, \psi_6^4 \, \chi_{12}^3\\
& + 2^{11}\,3^6\,37\,\psi_4^4\,\chi_{10}^2\,\chi_{12}^2+2^{11}\,3^6\,5\cdot 7 \, \psi_4 \, \psi_6^2\, \chi_{10}^2 \, \chi_{12}^2 -2^{23}\, 3^9 \, 5^3 \, \psi_6\, \chi_{10}^3 \, \chi_{12}^2 \\
& - 3^2 \, \psi_4^7 \, \chi_{10}^2 \, \chi_{12} + 2 \cdot 3^2 \, \psi_4^4 \, \psi_6^2 \, \chi_{10}^2 \, \chi_{12} + 2^{11} \, 3^5 \, 5 \cdot 19 \, \psi_4^3 \, \psi_6 \, \chi_{10}^3 \, \chi_{12} \\
&  + 2^{20} \, 3^8 \, 5^3 \, 11 \, \psi_4^2 \, \chi_{10}^4 \, \chi_{12} - 3^2 \, \psi_4 \, \psi_6^4 \, \chi_{10}^2 \, \chi_{12} + 2^{11} \, 3^5 \, 5^2 \, \psi_6^3 \, \chi_{10}^3 \, \chi_{12}  - 2 \, \psi_4^6 \, \psi_6 \, \chi_{10}^3 \\
 & - 2^{12} \, 3^4 \, \psi_4^5 \, \chi_{10}^4 + 2^2 \, \psi_4^3 \, \psi_6^3 \, \chi_{10}^3 + 2^{12} \, 3^4 \, 5^2 \, \psi_4^2 \, \psi_6^2 \, \chi_{10}^4 + 2^{21} \, 3^7 \, 5^4 \, \psi_4 \, \psi_6 \, \chi_{10}^5 \\
 & - 2 \, \psi_6^5 \, \chi_{10}^3 + 2^{32} \, 3^9 \, 5^5 \, \chi_{10}^6 \Big) \;.
\end{split}
\eeq
Hence, the expression $Q:= 2^{12} \, 3^9 \, \chi_{35}^2 /\chi_{10}$ is a polynomial of degree $60$ in the even generators.   Igusa also proved that each Siegel modular form (with trivial character) of odd weight is divisible by the form $\chi_{35}$. The following fact is well-known \cite{MR1438983}:
\begin{proposition}
\label{prop:Q}
The vanishing divisor of the cusp form $\chi_{10}$ is the Humbert surface $\mathcal{H}_1$, i.e., a period point $\tau$ is equivalent to a point with $\tau_{12}=0$ relative to $\Gamma_2$ if and only if $\chi_{10}(\tau)=0$.  The vanishing divisor of $Q$ is the Humbert surface $\mathcal{H}_4$, i.e.,  a period point $\tau$ is equivalent to a point with $\tau_{11}=\tau_{22}$ relative to $\Gamma_2$ if and only if $Q=0$. 
\end{proposition}
It is known that one has $\chi_{10}(\tau)=0$ if and only if the principally polarized abelian surface $\mathbf{A}$ is a product of two elliptic curves  $\mathbf{A} =\mathcal{E}_{\tau_{11}} \times \mathcal{E}_{\tau_{22}}$ with the transcendental lattice  $\mathrm{T}_\mathbf{A} = H \oplus H$; see \cite{MR3712162}.  Here, $H$ denotes the lattice $\mathbb{Z}^2$ with quadratic form $q(\vec{v}) = 2 v_1 v_2$. Moreover, for $Q(\tau)=0$ the transcendental lattice of the corresponding abelian surface $\mathbf{A}$ is given by  $\mathrm{T}_\mathbf{A} = H \oplus \langle 2 \rangle \oplus \langle -2 \rangle$. Here, $\langle m \rangle$ denotes the rank-one lattice $\mathbb{Z} v$ with $q(v)=m$. Similarly, the locus of $\mathcal{A}_2$ corresponding to abelian surfaces $\mathbf{A} $ with transcendental lattice $\mathrm{T}_\mathbf{A} = H \oplus \langle 4 \rangle \oplus \langle -4 \rangle$ 
is a surface $\mathcal{H}_{16}  \subset \mathcal{A}_2 $, which is a special irreducible component of $\mathscr{H}_{16}$. 
\subsection{Components of the Humbert surface \texorpdfstring{$\mathcal{H}_4(2)$}{H4(2)}}
The ring of invariants of binary sextics is generated by the so-called Igusa-Clebsch invariants $(I_2 , I_4 , I_6 , I_{10})$ which were studied in \cite{MR1106431}*{p.~\!319} and also in \cite{MR0141643}*{p.~\!17}.  Igusa \cite{MR0229643}*{p.~\!848} proved that the relations between the Igusa invariants of a binary sextic $Y^2=f_6(X,Z)$ defining a smooth genus-two curve $\mathcal{C}$ and the even Siegel modular forms for the associated principally polarized abelian surface $\mathbf{A} = \operatorname{Jac}{(\mathcal{C})}$ with period matrix $\tau$ are as follows:
\beq
\label{invariants}
\begin{split}
 I_2(f_6) & = -2^3 \cdot 3 \, \dfrac{\chi_{12}(\tau)}{\chi_{10}(\tau)} \;, \\
 I_4(f_6) & = \phantom{-} 2^2 \, \psi_4(\tau) \;,\\
 I_6(f_6) & = -\frac{2^3}3 \, \psi_6(\tau) - 2^5 \,  \dfrac{\psi_4(\tau) \, \chi_{12}(\tau)}{\chi_{10}(\tau)} \;,\\
 I_{10}(f_6) & = -2^{14} \, \chi_{10}(\tau) \not = 0 \;.
\end{split}
\eeq
Here, the Igusa invariant $I_{10}$ is the discriminant of the sextic $f_6(X,Z)$ and we are using the same normalization as in \cites{MR3712162,MR3731039}.
\subsubsection{Sextics with extra involution}
Bolza \cite{MR1505464} described the possible automorphism groups of genus-two curves defined by sextics. In particular,  he proved that a sextic curve $Y^2=f_6(X,Z)$ defining a genus-two curve $\mathcal{C}_0$ admits an elliptic involution if and only if $Q(\tau)=0$. Moreover, he proved that one can always represent this extra involution as $[X:Y:Z] \mapsto [-X:Y:Z]$. It follows  that the smooth sextic curve $\mathcal{C}_0$ with extra involution can be brought into the normal form given by
\beq
\label{eqn:genus2+auto}
  \mathcal{C}_0: \quad Y^2 = X^6 + s_1 X^4 Z^2 + s_2 X^2 Z^4 + Z^6 \,.
\eeq 
One uses Equations~(\ref{invariants}) and Equation~(\ref{chi_35sqr}) to check that $Q= 2^{12} \, 3^9 \, \chi_{35}^2 /\chi_{10}$ always vanishes for the genus-two curve $\mathcal{C}_0$.  
\par Given the elliptic involution $[X:Y:Z] \mapsto [-X:Y:Z]$ on $ \mathcal{C}_0$, its composition with the hyperelliptic involution defines a second elliptic involution. Thus, the elliptic involutions on $ \mathcal{C}_0$ come naturally in pairs.  The two involutions define two elliptic subfields of degree two for the function field of $\mathcal{C}_0$. We introduce the elliptic curves $\mathcal{E}_l$ in $\mathbb{P}^2 = \mathbb{P}(x_l, y_l, z_l)$ for $l=1, 2$, given by
\beq
\label{eqn:elliptic_curves}
 \mathcal{E}_1: \ y_1^2 z_1 = x_1^3 + s_2 x_1^2 z_1 + s_1 x_1 z_1^2 +  z_1^3 \,, \qquad
 \mathcal{E}_2: \ y_2^2 z_2 = x_2^3 + s_1 x_2^2 z_2 + s_2 x_2 z_2^2 +  z_2^3 \,,
\eeq
where $\mathcal{E}_1$ and $\mathcal{E}_2$ have the j-invariants $j_1 = j(\mathcal{E}_1)$ and $j_2 = j(\mathcal{E}_2)$ with
\beq
\label{eqn:j_invs}
 j_1 =\frac{2^8 \left(3 s_1 - s_2^2\right)^3}{4(s_1^3+s_2^3)-(s_1s_2)^2-18 s_1s_2+27}  \,, \qquad
 j_2 = \frac{2^8 \left(3 s_2 - s_1^2\right)^3}{4(s_1^3+s_2^3)-(s_1s_2)^2-18 s_1s_2+27} \,,
\eeq
respectively. Here, we use the standard normalization of the j-invariant  where the square torus with the complex structure $i$ satisfies $j=1728=12^3$.  The degree-two quotient maps $\pi_{\mathcal{E}_l}:  \mathcal{C}_0 \rightarrow \mathcal{E}_l$ associated with the involutions are given by
\beq
 \pi_{\mathcal{E}_1} : \quad   \mathcal{C}_0 \ \longrightarrow \  \mathcal{E}_1\,, \qquad \Big[ X : Y : Z \Big]   \ \mapsto \ \Big[ x_1: y_1 : z_1 \Big]  =  \Big[ XZ^2: Y : X^3 \Big] \,,
\eeq
and
\beq
 \pi_{\mathcal{E}_2} : \quad   \mathcal{C}_0 \ \longrightarrow \  \mathcal{E}_2 \,, \qquad \Big[ X : Y : Z \Big]   \ \mapsto \ \Big[ x_2: y_2 : z_2 \Big] =  \Big[ X^2Z: Y : Z^3 \Big] \,,
\eeq
respectively, for $XZ \not =0$.
\par At this point a diagram of the $(\mathbb{Z}/2\mathbb{Z})^2$-covering might be helpful. Assume that the genus-two curve $\mathcal{C}_0$ is given by $Y^2=(X^2-aZ^2)(X^2-bZ^2)(X^2-cZ^2)$ so that the Weierstrass points are again paired by a change of sign. On $\mathcal{C}_0$ we have the involutions $\jmath_1(X,Y,Z)=(X,Y,-Z)$ and $\jmath_2(X,Y,Z)=(X,-Y,-Z)$, and the corresponding quotients are $\mathcal{E}_1: y_1^2 z_1  = (a^2 x_1 - z_1)(b^2 x_1 - z_1)(c^2 x_1 - z_1)$ and $\mathcal{E}_2: y_2^2 z_2  = (x_2 - a^2 z_2)(x_2 - b^2 z_2)(x_2 - c^2 z_2)$, respectively.  We then obtain the diagram in Figure~\ref{fig:Z2square_covering}.  The five points on $\mathbb{P}(x_1,z_1)$ that determine the cover maps are $\{[0:1], [a^2:1], [b^2:1], [c^2:1], [1:0]\}$ (and similar on $\mathbb{P}(x_2,z_2)$).
\begin{figure}
\beqn
\xymatrix{
& \mathcal{C}_0  \ar [dl] _{\pi_{\mathcal{E}_1}} \ar [dr]^{\pi_{\mathcal{E}_2}} \ar [d] \\
\mathcal{E}_1 =\mathcal{C}_0 / \langle \jmath_1 \rangle \ar [d] & \mathbb{P}(X,Z) \ar[dl] \ar[dr]& \mathcal{E}_2=\mathcal{C}_0/\langle \jmath_2\rangle \ar[d] \\
\mathbb{P}(x_1,z_1)  \ar @/_0.5pc/ @{<->} ^{\cong}_{[x_1:z_1]=[z_2:x_2]} [rr] & & \mathbb{P}(x_2,z_2) \\
}
\eeqn
\caption{$(\mathbb{Z}/2\mathbb{Z})^2$-covering for the genus-two curve $\mathcal{C}_0$}
\label{fig:Z2square_covering}
\end{figure}
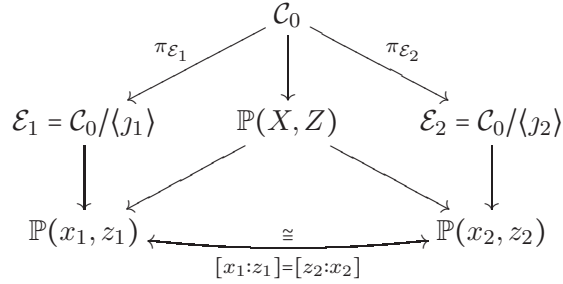
\par Let us state a closing proposition summarizing the main facts in this section:
\begin{proposition}
A point $\mathcal{M}_2 \subset \mathcal{A}_2$ lies in $\mathcal{H}_4$ if and only if the associated genus-two curve can be described as in Equation~(\ref{eqn:genus2+auto}). In this case Figure~\ref{fig:Z2square_covering} gives the corresponding isogeny.
\end{proposition}
\subsubsection{Pringsheim decomposition}
Next we look at the irreducible components of the inverse image of $\mathcal{H}_4$ in $\mathcal{A}_2(2)$ which we denote by $\mathcal{H}_4(2)$. We start with a smooth  genus-two curve $\mathcal{C}$ in Rosenhain form given by Equation~(\ref{Eq:Rosenhain}). We use the aforementioned relations between the Igusa invariants and  the Siegel modular forms to expand the generators of the ring of modular forms in terms of the Rosenhain roots. We obtain
\beq
\label{eqn:symmetry}
\begin{array}{rl}
\multicolumn{2}{l}{- \dfrac{2^{22} \chi_{35}^2(\tau)}{\chi_{10}^4(\tau)}  =  	{\color{black}\big(\lambda_1 - \lambda_2 \lambda_3\big)^2 } 
		{\color{red} \big(\lambda_2 - \lambda_1 \lambda_3\big)^2}  
		{\color{red} \big(\lambda_3 - \lambda_1 \lambda_2\big)^2} }\\[4pt]
\times & 	{\color{magenta}\big(\lambda_1 - \lambda_2 - \lambda_3 + \lambda_2 \lambda_3\big)^2}
		{\color{blue} \big(-\lambda_1 + \lambda_2 - \lambda_3 + \lambda_1 \lambda_3\big)^2} 
		{\color{blue} \big(-\lambda_1 - \lambda_2 + \lambda_3 + \lambda_1 \lambda_2\big)^2} \\[4pt]
\times & 	{\color{magenta}\big(\lambda_1\lambda_2 + \lambda_1\lambda_3  - \lambda_2 \lambda_3 - \lambda_1\big)^2 }
		{\color{blue}\big(\lambda_1\lambda_2 + \lambda_2\lambda_3 - \lambda_1 \lambda_3 - \lambda_2\big)^2} 
		{\color{blue} \big(\lambda_1\lambda_3 + \lambda_2\lambda_3 - \lambda_1 \lambda_2 - \lambda_3\big)^2} \\[4pt]
\times & \left\lbrace \begin{array}{l} 
		{\color{blue}\big(\lambda_1\lambda_2 - \lambda_1\lambda_3  -\lambda_1 + \lambda_3\big)^2} 
		{\color{red}\big(\lambda_1\lambda_3 - \lambda_2\lambda_3 - \lambda_1 + \lambda_3\big)^2} 
		{\color{red}\big(\lambda_1\lambda_2 - \lambda_2\lambda_3 - \lambda_1 + \lambda_2\big)^2}  \\[2pt]
		{\color{blue}\big(\lambda_1\lambda_2 - \lambda_1\lambda_3 + \lambda_1 - \lambda_2\big)^2} 
		{\color{blue}\big(\lambda_1\lambda_3 - \lambda_2\lambda_3 + \lambda_2 - \lambda_3\big)^2} 
		{\color{blue}\big(\lambda_1\lambda_2 - \lambda_2\lambda_3 - \lambda_2 + \lambda_3\big)^2}  \,.
\end{array} \right.
\end{array}
\eeq
The vanishing divisor of Equation~(\ref{eqn:symmetry}) defines 15 components in $\mathcal{A}_2(2)=\mathbb{H}_2 / \Gamma_2(2)$ of discriminant $\Delta =4$. Notice that each line in Equation~(\ref{eqn:symmetry}) is arranged to be invariant under permutations of the three roots $\lambda_1, \lambda_2, \lambda_3$. Pringsheim proved the following statement\footnote{We corrected two minor typos in the statement of the main theorem.} in \cite{MR1509868}:
\begin{proposition}[Pringsheim]
\label{prop:pringsheim}
There are exactly 15 components of $\mathcal{H}_4(2)$ in $\mathcal{A}_2(2)$, given in Table~\ref{fig:Pringsheim}. Each of the component is equivalent to $\tau_{11} = \tau_{22}$  relative to the modular group $\Gamma_2$. Moreover, there is a transposition of the six roots $(\lambda_1, \lambda_2, \lambda_3, 0, 1,\infty)$ in  $\Gamma_2/\Gamma_2(2)\cong S_6$ that permutes each pair of components. 
\end{proposition}
\begin{remark}
In Table~\ref{fig:Pringsheim} for each entry the equation of a rational divisor in $\mathbb{H}_2$ is given whose projection to $\mathbb{H}_2/\Gamma_2(2)$ realizes a component in Proposition~\ref{prop:pringsheim}. The corresponding vanishing divisor in terms of the Rosenhain roots is given as well. The group names (I, II, III, IV) were introduced by Pringsheim, and they are colored to match the corresponding terms in Equation~(\ref{eqn:symmetry}). 
\end{remark}
\begin{table}[ht]
\scalemath{1}{
\begin{tabular}{|l|l|}
\hline
\multicolumn{2}{|c|}{{\color{black}I}} \\
\hdashline
$ 2 \tau_{12} + \tau_{11} \tau_{22} - \tau_{12}^2 =0$ & $\lambda_1=\lambda_2 \lambda_3$ \\
\hline
\multicolumn{2}{|c|}{{\color{magenta}II}} \\
\hdashline
$\tau_{11} + 2 \tau_{12}  =0$ 							& $\lambda_1- \lambda_2 = \lambda_3(1-\lambda_2)$\\
$\tau_{11} + 2 \tau_{12} - (\tau_{11} \tau_{22} - \tau_{12}^2)=0$	& $\lambda_1(1-\lambda_3)=\lambda_2(\lambda_1-\lambda_3)$\\
\hline
\multicolumn{2}{|c|}{{\color{red}III}} \\
\hdashline
$2 \tau_{12} - \tau_{22} =0 $ 							& $\lambda_3 =\lambda_1 \lambda_2$ \\
$2 \tau_{12} - \tau_{22} + (\tau_{11} \tau_{22} - \tau_{12}^2) =0$	& $\lambda_2 =\lambda_1 \lambda_3$ \\
$\tau_{11} - \tau_{22}=0$								& $\lambda_1 - \lambda_2 = \lambda_2 (\lambda_1 - \lambda_3)$\\
$\tau_{11} -\tau_{22} + (\tau_{11} \tau_{22} - \tau_{12}^2) =0$	& $\lambda_1 - \lambda_3 = \lambda_3 (\lambda_1 - \lambda_2)$\\
\hline
\multicolumn{2}{|c|}{{\color{blue}IV}} \\
\hdashline
$ 2 \tau_{12}=1$									& $ \lambda_2-\lambda_3 = - \lambda_1(1-\lambda_2)$\\
$ 2 \tau_{12} + \tau_{22} =1$							& $ \lambda_3 (1-\lambda_2) = -\lambda_1 (\lambda_2 -\lambda_3)$\\
$\tau_{11} + 2 \tau_{12}=1$							& $\lambda_2-\lambda_3 = \lambda_1 (1-\lambda_3)$\\
$\tau_{11} + 2 \tau_{12} + \tau_{22} - (\tau_{11} \tau_{22}-\tau_{12}^2)=1$ & $\lambda_1(\lambda_2-\lambda_3)=\lambda_2(1-\lambda_3)$\\
$4 \tau_{12} + 3 (\tau_{11} \tau_{22} - \tau_{12}^2)=1	$		& $\lambda_2 -\lambda_3=\lambda_2(\lambda_1-\lambda_3)$\\
$4\tau_{12} + \tau_{22} + 3 (\tau_{11} \tau_{22} - \tau_{12}^2)=1$& $\lambda_1 -\lambda_3=\lambda_1(\lambda_2-\lambda_3)$\\ 
$\tau_{11} + 4\tau_{12} + 3 (\tau_{11} \tau_{22} - \tau_{12}^2)=1$& $\lambda_2-\lambda_3=-\lambda_3(\lambda_1-\lambda_2)$\\
$\tau_{11} + 4\tau_{12} +  \tau_{22} + 2 (\tau_{11} \tau_{22} - \tau_{12}^2)=1$& $\lambda_1-\lambda_2=-\lambda_1(\lambda_2-\lambda_3)$\\
\hline
\end{tabular}}
\smallskip
\caption{Pringsheim components of $\mathcal{H}_4$}
\label{fig:Pringsheim}
\end{table}
\subsubsection{Normal forms with level-two structure}
We will use component (I) in Proposition~\ref{prop:pringsheim} to construct a smooth genus-two curve in Rosenhain normal form admitting an elliptic involution.  We consider the smooth genus-two curve $\mathcal{C}_0$, given by
\beq
\label{Eq:Rosenhain_special}
  \mathcal{C}_0: \quad Y^2 = X Z \, \big(X-Z\big) \, \big( X- \lambda_2 \lambda_3 Z\big) \,  \big( X- \lambda_2  Z\big) \,  \big( X - \lambda_3  Z\big) \,,
\eeq
with a discriminant given by
\beq
  \lambda_2^6  \lambda_3^6 (\lambda_2-1)^4  (\lambda_3-1)^4  (\lambda_2\lambda_3-1)^4   (\lambda_2 - \lambda_3)^2 \,.
\eeq
Equation~(\ref{Eq:Rosenhain_special}) makes all Weierstrass points rational and moves them into ‘standard’ position, which is equivalent to choosing one of the 15 components in Proposition~\ref{prop:pringsheim}. In Section~\ref{sssec:Legendre} we will also consider another normal form for $\mathcal{C}_0$ that is based on the sextic with elliptic involution in Equation~(\ref{eqn:genus2+auto}).
\par Let us  denote by $\Lambda_l \in \mathbb{P}^1 \backslash \lbrace 0, 1, \infty \rbrace$ the modular parameter for the elliptic curves $\mathcal{E}_l$ in $\mathbb{P}^2 = \mathbb{P}(x_l, y_l, z_l)$ for $l=1, 2$ defined by the \emph{Legendre normal form}
\beq
\label{eqn:EC}
 \mathcal{E}_l: \quad y_l^2 z_l = x_l \Big( x_l -  z_l\Big) \Big( x_l - \Lambda_l z_l\Big)  \,,
\eeq
with the hyperelliptic involution given by $\imath_{\mathcal{E}_l}: [x_l : y_l : z_l] \mapsto [x_l : -y_l : z_l]$.  We have the following:
\begin{lemma}
\label{lem:ECR}
The function field of the smooth genus-two curve $\mathcal{C}_0$ contains the subfields given by the function fields of the elliptic curves $\mathcal{E}_l$ for $l=1, 2$ if
\beq
\label{eqn:EC_12_j_invariants}
 \Lambda_1 \Lambda_2 = \dfrac{(\lambda_2 +\lambda_3)^2 -4 \lambda_2\lambda_3}{(1-\lambda_2)^2 (1-\lambda_3)^2} \,, \qquad
 \Lambda_1 + \Lambda_2 = - \dfrac{2(\lambda_2 +\lambda_3)}{(1-\lambda_2) (1-\lambda_3)} \,.
\eeq
\end{lemma}
\begin{proof}
The genus-two curve in Bolza normal form in Equation~(\ref{eqn:genus2+auto}) is isomorphic to Rosenhain curve in Equation~(\ref{Eq:Rosenhain_special}) if we set
\beq
 \lambda_2 = \frac{(\mu^2 \nu+1)(\mu-\nu)}{(\mu^2 \nu-1)(\mu+\nu)}\,, \qquad  \lambda_3 =  \frac{(\mu^2 \nu-1)(\mu-\nu)}{(\mu^2\nu+1)(\mu+\nu)} \,,
\eeq 
and
\beq
 s_1 = \mu^2 + \nu^2 + \frac{1}{\mu^2\nu^2} \,, \qquad s_2 =   \frac{1}{\mu^2} + \frac{1}{\nu^2} + \mu^2 \nu^2 \,,
\eeq
for $\mu, \nu \in \mathbb{C}^\times$ such that $(\mu^2-\nu^2)(\mu^4 \nu^2-1)\not =0$. This statement is proved by computing the Igusa-Clebsch invariants, using the normalization in \cites{MR3712162,MR3731039}.  Denoting the Igusa-Clebsch invariants of the genus-two curve in Equation~(\ref{Eq:Rosenhain_special}) and Equation~(\ref{eqn:genus2+auto}) by $[ I_2 : I_4 : I_6 : I_{10} ] \in \mathbb{WP}_{(2,4,6,10)}$ and $[ I'_2 : I'_4 : I'_6 : I'_{10} ]$, respectively, one checks that
\beq
[ I_2 : I_4 : I_6 : I_{10} ] = [ s^2 I'_2 \ : \ s^4I'_4 \ : \ s^6I'_6 \ : \ s^{10}I'_{10} ] = [ I'_2 : I'_4 : I'_6 : I'_{10} ] 
\eeq
with $s \in \mathbb{C}^\times$. The remainder of the proof follows from computing the j-invariants for the curves in Equation~(\ref{eqn:EC}) and comparing them with Equations~(\ref{eqn:j_invs}).  The solution of Equations~(\ref{eqn:EC_12_j_invariants}), up to interchanging $\Lambda_1 \leftrightarrow \Lambda_2$, is then given by
\beq
 \Lambda_1 =  \frac{-\mu^2 + \nu^2}{\mu^4\nu^4-\mu^2} \,, \qquad
 \Lambda_2 =  \frac{-\mu^4 \nu^2 + \mu^2 \nu^4}{\mu^2 \nu^4-1}\,,
\eeq
and $j(\mathcal{E}_1)=j_1$ and $j(\mathcal{E}_2)=j_2$ in Equations~(\ref{eqn:j_invs}).
\end{proof}
\par We have the following:
\begin{proposition}
\label{prop:quotient_maps}
Assume that $\lambda_2, \lambda_3 \in \mathbb{P}^1 \backslash \{ 0, 1, \infty\}$ satisfy $\lambda_2 \not= \lambda_3^{\pm1}$, and the moduli of the curves $\mathcal{C}_0$ and $\mathcal{E}_l$ in Equations~(\ref{Eq:Rosenhain_special}) and~(\ref{eqn:EC}) satisfy Equations~(\ref{eqn:EC_12_j_invariants}). Quotient maps $\pi_{\mathcal{E}_l}:  \mathcal{C}_0 \rightarrow  \mathcal{E}_l$ for $l= 1,2$ are given by
\beq
\label{eqn:pi_l}
 \pi_{\mathcal{E}_l} : \quad   \mathcal{C}_0 \ \longrightarrow  \  \mathcal{E}_l \,, \qquad \Big[ X : Y : Z \Big]   \mapsto \Big[ x_l: y_l : z_l \Big]  \; \; \text{for $l= 1,2$,}
\eeq
with
\beq
\Big[ x_l: y_l : z_l \Big] \ = \ \Big[ r \big(X-\lambda_2 Z \big)\big(X-\lambda_3 Z \big) \, XZ:  \ \big(X - (-1)^l q Z\big) \, Y  :\ r^3 X^2Z^2 \Big] \,,
\eeq
for $XZ \not =0$, and $[ x_l: y_l : z_l] = [1:0:0]$ otherwise. Here, $q, r$ are square roots of $q^2=\lambda_2\lambda_3$ and $r^2 = (1-\lambda_2)(1-\lambda_3)$, respectively. The elliptic involutions $\jmath_l$ on $\mathcal{C}_0$, given by
\beq
\label{eqn:involutions}
\jmath_l: \quad \Big[ X : Y : Z \Big]  \mapsto \Big[ \lambda_2\lambda_3 Z : \ (-1)^{l+1} \lambda_2\lambda_3 q Y : \ X \Big] \,,
\eeq
satisfy $\pi_{\mathcal{E}_l} \circ \jmath_l = \pi_{\mathcal{E}_l}$ for $l=1, 2$.
\end{proposition}
\begin{proof}
The genus-two curve $\mathcal{C}_0$ in Equation~(\ref{Eq:Rosenhain_special}) is assumed to be smooth. Thus, one must have $\lambda_2, \lambda_3 \in \mathbb{P}^1 \backslash \{ 0, 1, \infty\}$ and $\lambda_2 \not= \lambda_3^{\pm1}$.  The remainder of the proof follows by explicit computation. We note that a choice of square root for $r$ is changed by composing $\pi_{\mathcal{E}_l}$ with the elliptic involution on $\mathcal{E}_l$. The choice of square root for $q$ is changed by interchanging $\pi_{\mathcal{E}_1} \leftrightarrow \pi_{\mathcal{E}_2}$.
\end{proof}
In the context of hyperelliptic period integrals, the construction of the rational double covers is known as the (classical) \emph{Jacobi reduction} of a genus-two curve with elliptic involution \cite{MR1500768}. We note that  we are using component (I) in Proposition~\ref{prop:pringsheim} to construct smooth normal forms for the genus-two curve in Rosenhain normal form and its elliptic subfields representing a specific divisor of $\mathcal{H}_4(2) \subset \mathcal{A}_2(2)$:
\begin{remark}
\label{rem:field_extensions}
Using a covering space of component (I) in Proposition~\ref{prop:pringsheim}, given by the set of tuples $(k_1, k_2)$ with $\lambda_l= k_l^2$ for $l=2,3$ and $\lambda_1 = (k_2 k_3)^2$, we obtain for the modular parameters of the elliptic-curve quotients $\mathcal{E}_l$  in Lemma~\ref{lem:ECR} the following algebraic solutions of Equation~(\ref{eqn:EC_12_j_invariants}):
\beq
\label{eqn:moduli}
 \Lambda_1 =  - \dfrac{(k_2 - k_3)^2}{(1- k_2^2) (1-k_3^2)} \,, \qquad
 \Lambda_2 =  - \dfrac{(k_2 + k_3)^2}{(1- k_2^2) (1-k_3^2)}  \,.
\eeq 
A change of square roots $(k_2, k_3) \mapsto (\pm k_2, \pm k_3)$ leaves Equations~(\ref{eqn:moduli}) invariant, whereas the change $(k_2, k_3) \mapsto (\pm k_2, \mp k_3)$ interchanges $\Lambda_1$ and $\Lambda_2$.  This means that the elliptic curves $\mathcal{E}_l$ with level-two structure can be constructed over $\mathbb{Q}(k_2, k_3)$ as a finite field extension of  $\mathbb{Q}(\Lambda_1, \Lambda_2)$.
\end{remark}
\subsubsection{Legendre's gluing construction}
\label{sssec:Legendre}
 In this section we construct an explicit model for $\mathcal{C}_0$ over $\mathbb{Q}(\Lambda_1, \Lambda_2)$ for elliptic curves in Legendre normal form~(\ref{eqn:EC}) with modular parameters $\Lambda_1, \Lambda_2$. This is often referred to as \emph{Legendre's gluing method}.
\par The Weierstrass points $P_i$ of the curve $\mathcal{C}_0$ in Equation~(\ref{Eq:Rosenhain_special}), given by 
\beq
\begin{split}
 P_1 =  [\lambda_2 \lambda_3: 0 :1] \,, \quad  P_2 =  & \, [\lambda_2 : 0 :1] \,,\quad P_3 =  [\lambda_3: 0 :1] \,, \\
 P_4 =  [0: 0 :1] \,,  \quad P_5 =& \, [1: 0 :1] \,,  \quad P_6 =  [1: 0 :0] \,,
 \end{split}
 \eeq
are the fixed points of the hyperelliptic involution $\imath_{\mathcal{C}_0}$. The two elliptic involutions $\jmath_l$ in Equation~(\ref{eqn:involutions}) for $l=1,2$ (each) pairwise interchange the Weierstrass points, i.e., 
 \beq
  \jmath_l: \quad P_2 \leftrightarrow P_3 \,, \quad   P_1 \leftrightarrow P_5 \,, \quad  P_4 \leftrightarrow P_6 \,.
\eeq  
The points $Q_{l,\pm}$, given by
\beq
\begin{split}
 Q_{1,\pm} &=  [ -q: \pm i q^{\frac{3}{2}} (k_2 k_3+1) (k_2+k_3) : 1 ] \,, \\
 Q_{2,\pm} &=  [\phantom{-} q:  \pm \phantom{i} q^{\frac{3}{2}} (k_2 k_3-1) (k_2-k_3): 1 ] \,,
 \end{split}
\eeq  
are the ramification points of $\pi_{\mathcal{E}_l}$ in Equation~(\ref{eqn:pi_l}) for $l= 1$ and $l=2$, respectively.  The involution $\jmath_1$ fixes the points $Q_{1,+}$ and $Q_{1,-}$, and interchanges the points $Q_{2,+} \leftrightarrow Q_{2,-}$. An analogous statements holds for the involution $\jmath_2$.  The hyperelliptic involution of the curve $\mathcal{C}_0$ interchanges the elements of the two pairs simultaneously, i.e.,  $\imath_{\mathcal{C}_0}: Q_{l,+} \leftrightarrow Q_{l,-}$ for $l=1, 2$. 
\par The pairs $\{ P_2, P_3 \}$, $\{ P_1, P_5 \}$, $\{Q_{2,+}, Q_{2,-}\}$, and $\{ P_5, P_6 \}$ are mapped by $\pi_{\mathcal{E}_1}$ to the two-torsion points $p_0: [x_1 : y_1: z_1] = [0:0:1]$,  $p_1: [1:0:1]$, $p_{\Lambda_1}: [\Lambda_1:0:1]$, and the identity $p_\infty: [1:0:0]$ in $\mathcal{E}_1[2]$. Similarly, the pairs $\{ P_2, P_3 \}$, $\{ P_1, P_5 \}$, $\{Q_{1,+}, Q_{1,-}\}$, and $\{ P_5, P_6 \}$ are mapped by $\pi_{\mathcal{E}_2}$ to the two-torsion points $p_0: [x_2 : y_2: z_2] = [0:0:1]$,  $p_1: [1:0:1]$, $p_{\Lambda_2}: [\Lambda_2:0:1]$, and the identity $p_\infty:[1:0:0]$ in $\mathcal{E}_2[2]$. 
\par We associate with the branch locus of $\pi_{\mathcal{E}_1}$ and $\pi_{\mathcal{E}_2}$ the effective divisor of degree two $B_1 = [ \pi_{\mathcal{E}_1}( Q_{1,+}) + \pi_{\mathcal{E}_1}( Q_{1,-}) ]$ on $\mathcal{E}_1$ and  $B_2  = [  \pi_{\mathcal{E}_2}( Q_{2,+}) + \pi_{\mathcal{E}_2}( Q_{2,-}) ]$  on $\mathcal{E}_2$, respectively. One checks that for $\pi_{\mathcal{E}_1}( Q_{1,\pm} )$ one has $[x_1:z_1]=[\Lambda_2:1]$, and for $\pi_{\mathcal{E}_2}( Q_{2,\pm} )$ one has $[x_2:z_2]=[\Lambda_1:1]$.  Because of $\operatorname{Pic}^0(\mathcal{E}_l) \cong \mathcal{E}_l$ the line bundle $\mathcal{O}_{\mathcal{E}_l}(B_l)$ associated with the branch locus $B_l$ is equivalent to the line bundle $\mathcal{O}_{\mathcal{E}_l}(2 p_\alpha)$ if and only if
\beq
 \pi_{\mathcal{E}_l}\big( Q_{l,+} \big) \oplus \pi_{\mathcal{E}_l}\big( Q_{l,-} \big)  = 2 p_\alpha = p_\alpha \oplus p_\alpha \,,
\eeq 
where $\oplus$ refers to the addition with respect to the elliptic curve group law on $\mathcal{E}_l$. Using the properties $\pi_{\mathcal{E}_l}$ in Equation~(\ref{eqn:pi_l}) one finds that
\beq
  \pi_{\mathcal{E}_l}\big( Q_{l,+} \big) \oplus \pi_{\mathcal{E}_l}\big( Q_{l,-} \big)  = 0 \,,
\eeq
and $p_\alpha \in \mathcal{E}_l[2]$ is a two-torsion point.  Thus, there are four line bundles $\mathcal{L}_l \to \mathcal{E}_l$ such that $\mathcal{L}_l^{\otimes 2} \cong \mathcal{O}_{\mathcal{E}_l}(B_l)$, namely $\mathcal{L}_l  \cong   \mathcal{O}_{\mathcal{E}_l}(p_\alpha)$. It follows $h^0( \mathcal{E}_l, \mathcal{L}_l)=1$ by the Riemann-Roch theorem. Thus, the preimage of the vanishing locus of a section of $\mathcal{L}_l \to \mathcal{E}_l$ determines a unique double cover $p_l : \mathcal{C}_0 \rightarrow  \mathcal{E}_l$. The composition on $\mathcal{C}_0$ of the elliptic involution with the hyperelliptic involution defines a second elliptic involution, and a second elliptic-curve quotient is obtained. Thus, the data $(\mathcal{E}_l, \mathcal{L}_l, B_l)$ for either $l=1$ or $l=2$ determines the curve $\mathcal{C}_0$ uniquely. As we have seen, this data is equivalent to an (unordered) pair $\{ \Lambda_1, \Lambda_2\}$ of modular parameters $\Lambda_1, \Lambda_2 \in  \mathbb{P}^1 \backslash \{0, 1, \infty\}$ with $\Lambda_1 \not = \Lambda_2$. Following Serre’s explanation \cite{Serre85}*{Sec.~27} of Legendre's gluing method we obtain:
\begin{proposition}
\label{prop:serre}
Assume that $\Lambda_1, \Lambda_2 \in \mathbb{P}^1 \backslash \{ 0, 1, \infty\}$ satisfy $\Lambda_1 \not = \Lambda_2$, and the moduli of $\mathcal{C}_0$ and $\mathcal{E}_l$ in Equations~(\ref{Eq:Rosenhain_special}) and~(\ref{eqn:EC}) satisfy Equations~(\ref{eqn:EC_12_j_invariants}).  The smooth genus-two curve, given by
\beq
\label{eqn:LegendreSerreCurve}
  \widetilde{\mathcal{C}}_0: \quad Y^2 = \Big( X^2 -Z^2 \Big)  \left( X^2 - \frac{\Lambda_1}{\Lambda_2} Z^2 \right)   \left( X^2 - \frac{1-\Lambda_1}{1-\Lambda_2} Z^2 \right)  \,,
\eeq 
is isomorphic to $\mathcal{C}_0$ over a finite field extension of $\mathbb{Q}(\lambda_2, \lambda_3)$. The curve is also isomorphic to each of the curves obtained by replacing $\{ \Lambda_1, \Lambda_2\}$ in Equation~(\ref{eqn:LegendreSerreCurve}) by one of the following pairs:
\beq
\label{eqn:5values}
 \left\lbrace \frac{1}{1-\Lambda_1}, \frac{1}{1-\Lambda_2} \right\rbrace  , \;  
 \left\lbrace \frac{\Lambda_1-1}{\Lambda_1}, \frac{\Lambda_2-1}{\Lambda_2} \right\rbrace , \;   
 \left\lbrace \frac{1}{\Lambda_1}, \frac{1}{\Lambda_2} \right\rbrace , \;  
 \left\lbrace \frac{\Lambda_1}{\Lambda_1-1}, \frac{\Lambda_2}{\Lambda_2-1} \right\rbrace , \;  
 \Big\lbrace 1-\Lambda_1, 1-\Lambda_2 \Big\rbrace.
\eeq
\end{proposition}
\begin{proof}
Using Remark~\ref{rem:field_extensions} and denoting the Igusa-Clebsch invariants of any two pairs genus-two curves in the statement of Proposition~\ref{prop:serre} by $[ I_2 : I_4 : I_6 : I_{10} ] \in \mathbb{WP}_{(2,4,6,10)}$ and $[ I'_2 : I'_4 : I'_6 : I'_{10} ]$, respectively, one checks that
\beq
[ I_2 : I_4 : I_6 : I_{10} ] = [ s^2 I'_2 \ : \ s^4I'_4 \ : \ s^6I'_6 \ : \ s^{10}I'_{10} ] = [ I'_2 : I'_4 : I'_6 : I'_{10} ] 
\eeq
with $s \in \mathbb{Q}(\lambda_2, \lambda_3, \sqrt{\lambda_2\lambda_3})$.
\end{proof}
The anharmonic group is the group acting on the $\Lambda$-line of an elliptic curve in Legendre form (which is $\mathcal{A}_1(2)$) that is generated by the transformations $\Lambda \mapsto 1-\Lambda, \Lambda/(\Lambda-1)$.  The six pairs, given by $\{\Lambda_1, \Lambda_2\}$ and Equation~(\ref{eqn:5values}), form an orbit under the diagonal action of the anharmonic group. We have the following:
\begin{corollary}
\label{cor:2isogeny}
Assume that $\Lambda_1, \Lambda_2 \in \mathbb{P}^1 \backslash \{ 0, 1, \infty\}$ satisfy
\beq
 \Lambda_1 \ \not \in \ \left \lbrace \Lambda_2,  \, \frac{1}{1-\Lambda_2}, \,  \frac{\Lambda_2-1}{\Lambda_2}, \, \frac{1}{\Lambda_2}, \, \frac{\Lambda_2}{\Lambda_2-1}, \, 1- \Lambda_2 \right \rbrace\,.
\eeq
Using $\{\Lambda_1, \Lambda_2\}$ and the pairs
 \beq
\label{eqn:5values_2}
 \left\lbrace \Lambda_1, \frac{1}{1-\Lambda_2} \right\rbrace  , \;  
 \left\lbrace \Lambda_1, \frac{\Lambda_2-1}{\Lambda_2} \right\rbrace , \;   
 \left\lbrace \Lambda_1, \frac{1}{\Lambda_2} \right\rbrace , \;  
 \left\lbrace \Lambda_1, \frac{\Lambda_2}{\Lambda_2-1} \right\rbrace , \;  
 \Big\lbrace \Lambda_1, 1-\Lambda_2 \Big\rbrace,
\eeq
one obtains six smooth genus-two curves $\widetilde{\mathcal{C}}_0$ admitting an elliptic involution whose quotients $(\mathcal{E}_1, \mathcal{E}_2)$ are pairwise isomorphic and are given by Equations~(\ref{eqn:EC}). Generically, the six curves are non-isomorphic and satisfy $\operatorname{Jac}{( \widetilde{\mathcal{C}}_0)} \in \mathcal{H}_4$.
\end{corollary}
\begin{proof}
The proof follows from an explicit computation similar to the one in the proof of the previous proposition.
\end{proof}
\subsubsection{Relation to G\"opel groups}
Since the Picard group $\operatorname{Pic}^0(\mathcal{C}_0) \cong \operatorname{Jac}{(\mathcal{C}_0)}$ consists of elements of the form $[P + Q - 2 P_6 ]$ for $P, Q \in \mathcal{C}_0$, an isogeny $\Psi$ of abelian surfaces is defined by setting
\beq
\label{eqn:Psi}
 \Psi : \quad \operatorname{Jac}{(\mathcal{C}_0)}  \longrightarrow   \mathcal{E}_1 \times \mathcal{E}_2 \,, \quad  [P + Q - 2 P_6 ]  \mapsto  \Big( \pi_{\mathcal{E}_1}(P) \oplus  \pi_{\mathcal{E}_1}(Q) , \;   \pi_{\mathcal{E}_2}(P) \oplus  \pi_{\mathcal{E}_2}(Q) \Big)\,.
\eeq 
It follows from the construction of the line bundles $\mathcal{L}_l \to \mathcal{E}_l$ in Section~\ref{sssec:Legendre} that the line bundle $\mathcal{L} \to \mathbf{A}=\operatorname{Jac}{(\mathcal{C}_0)}$ associated with the principal polarization of $\mathbf{A}$ satisfies $\mathcal{L} \cong \Psi^* (\mathcal{L}_1 \boxtimes \mathcal{L}_2)$.  The elliptic involutions $\jmath_l$ on $\mathcal{C}_0$ extend to involutions on the Jacobian $\mathbf{A}=\operatorname{Jac}{(\mathcal{C}_0)}$ that coincide on $\mathbf{A}[2]$. We denote this induced involution on $\mathbf{A}[2]$ by $\jmath: \mathbf{A}[2] \to \mathbf{A}[2]$. 
\par Given the marking $(P_1, \dots, P_6) = (\lambda_1=\lambda_2\lambda_3, \lambda_2, \lambda_3, 0, 1,\infty)$ of the Weierstrass points for the curve $\mathcal{C}_0$ as before, we have the following:
\begin{proposition}
\label{prop:special_G_group}
In the situation above, the kernel of $\Psi:  \mathbf{A}=\operatorname{Jac}{(\mathcal{C}_0)} \rightarrow  \mathcal{E}_1 \times \mathcal{E}_2$ is given by the G\"opel group
\beq
 \ker \Psi \ = \ \Big\lbrace P_0, P_{15}, P_{23}, P_{46} \Big\rbrace \ \cong \ (\mathbb{Z}/2 \mathbb{Z})^2 \ \subset \ \mathbf{A}[2] \,.
\eeq 
In particular, $\Psi$ is a $(2,2)$-isogeny and $\mathbf{A}^\prime =\mathbf{A}/ \ker \Psi \cong  \mathcal{E}_1 \times \mathcal{E}_2$ with $\mathrm{T}_{\mathbf{A}^\prime} = H \oplus H$.
\end{proposition}
\begin{proof}
The induced actions of the two elliptic involutions $\jmath_l$ on $\operatorname{Jac}{(\mathcal{C}_0)}$ fix the divisors $P_{15}, P_{23}, P_{46}$, and $[Q_{l,+} + Q_{l,-} - 2 P_6 ]$. It follows from the explicit formulas for $\pi_{\mathcal{E}_l}$ in Proposition~\ref{prop:quotient_maps} that $[P + Q - 2 P_6 ] \in  \ker \Psi$ if and only if $P, Q \in \mathbf{A}[2]$ and $\jmath_l (P)=Q$ for $l=1, 2$. For $\mathbf{A}'$ we have $\chi_{10}=0$ whence $\mathbf{A}'$ is a product of two elliptic curves with the transcendental lattice  $\mathrm{T}_{\mathbf{A}'} = H \oplus H$; see \cite{MR3712162}. 
\end{proof}
The remaining G\"opel groups fall within two cases: the first case consists of six G\"opel groups such that the involution $\jmath$ fixes one non-trivial element and interchanges the other two. These groups are given by
\beq
\label{eqn:Goepel_invariant1}
\begin{split}
 \Big\lbrace P_0, P_{12}, P_{35}, P_{46} \Big\rbrace\,, \quad   \Big\lbrace P_0, P_{13}, P_{25}, P_{46} \Big\rbrace\,, \quad   \Big\lbrace P_0, P_{14}, P_{23}, P_{56} \Big\rbrace\,, \\
  \Big\lbrace P_0, P_{16}, P_{23}, P_{45} \Big\rbrace\,, \quad   \Big\lbrace P_0, P_{15}, P_{24}, P_{36} \Big\rbrace\,, \quad   \Big\lbrace P_0, P_{15}, P_{26}, P_{34} \Big\rbrace\,.
\end{split}
\eeq 
By construction, each of them contains one non-trivial element from the special G\"opel group $\lbrace P_0, P_{15}, P_{23}, P_{46} \rbrace$ that yields the quotient $\mathcal{E}_1 \times \mathcal{E}_2$.
We have the following:
\begin{proposition}
\label{prop:H4}
For each G\"opel group $G$ in Equation~(\ref{eqn:Goepel_invariant1}) the abelian surface $\mathbf{A}^\prime =\mathbf{A}/G$ satisfies $\mathbf{A}^\prime \cong \operatorname{Jac}{(\mathcal{C}_0^\prime)}$ for some smooth genus-two curve $\mathcal{C}_0^\prime$ admitting an elliptic involution. In particular, we have $\operatorname{Jac}{(\mathcal{C}_0^\prime)} \in \mathcal{H}_4$ and $\mathrm{T}_{\mathbf{A}^\prime} = H \oplus \langle 2 \rangle \oplus \langle -2 \rangle$. The six genus-two curves correspond to the curves in Corollary~\ref{cor:2isogeny} under $(2,2)$-isogeny.
\end{proposition}
\begin{proof}
The proof follows by constructing the Richelot curve in Equation~(\ref{Richelot}) and checking that $Q=0$ in Proposition~\ref{prop:Q} for each G\"opel subgroup $G_n$ with $n=1, \dots, 6,$ in Equation~(\ref{eqn:Goepel_invariant1}). For $\mathbf{A}'$ we have $Q=0$ whence the transcendental lattice of $\mathbf{A}'$ is $\mathrm{T}_{\mathbf{A}'} = H \oplus \langle 2 \rangle \oplus \langle -2 \rangle$. 
\end{proof}
\par For the elliptic curve $\mathcal{E}_l$ in Equation~(\ref{eqn:EC}), there are three order-two points that can be used to construct a two-isogenous elliptic curve. Using the order-two point $ [x_l: y_l: z_l]=[0:  \Lambda_l: 1]$, one easily obtains an elliptic curves $\mathcal{E}^\prime_l$ in $\mathbb{P}^2 = \mathbb{P}(X_l, Y_l, Z_l)$ for $l=1, 2$, given by
\beq
\label{eqn:EC_dual}
 \mathcal{E}^\prime_l : \quad Y_l^2  Z_l  = X_l  \Big(X_l^2 +2 \big(1 - 2 \Lambda_l \big) X_l Z_l + Z_l^2 \Big) \,, 
\eeq
together with an isogenies $\chi_{\mathcal{E}_l}: \mathcal{E}_l \rightarrow  \mathcal{E}_{l}^\prime$, and its dual two-isogenies $\chi_{\mathcal{E}_l }^\prime$. Assuming Equations~(\ref{eqn:EC_12_j_invariants}), the two-isogenous elliptic curves can be brought into Legendre normal form, but only over the finite field extension $\mathbb{Q}(k_1, k_2)$. In fact, the elliptic curve $\mathcal{E}^\prime_{l}$ is isomorphic over $\mathbb{Q}(k_1, k_2)$ to the quadratic-twist curve
\beq
  \mathcal{E}^\prime_{l} : \quad  Y_l^2  Z_l = \delta^\prime X_l \, (X_1 - Z_1) \, ( X_1 - \Lambda_1^\prime Z_1) \,,
\eeq
where we have set
\beq
\label{eqn:moduli_dual}
 \Lambda^\prime_1 = \frac{(k_2-1)^2 (k_3+1)^2}{(k_2+1)^2 (k_3-1)^2} \,, \qquad  \Lambda^\prime_2 = \frac{(k_2+1)^2 (k_3+1)^2}{(k_2-1)^2 (k_3-1)^2} \,, \qquad \delta^\prime = - (1-k_2^2)(1-k_3^2) \,.
\eeq
A change of square roots $(k_2, k_3) \mapsto (\pm k_2, \pm k_3)$ leaves Equations~(\ref{eqn:moduli_dual}) invariant or acts as the modular transformation $ \Lambda^\prime_l \mapsto 1/ \Lambda^\prime_l$, whereas 
the change $(k_2, k_3) \mapsto (\mp k_2, \pm k_3)$ in Equations~(\ref{eqn:moduli_dual}) interchanges $\Lambda^\prime_1$ and $\Lambda^\prime_2$ (up to a modular transformation).  Applying Proposition~\ref{prop:serre} one obtains the common double cover
\beq
\label{eqn:LegendreSerreCurve_dual}
  \widetilde{\mathcal{C}}^\prime_0: \quad Y^2 = \Big( X^2 -Z^2 \Big)  \left( X^2 - \frac{\Lambda^\prime_1}{\Lambda^\prime_2} Z^2 \right)   \left( X^2 - \frac{1-\Lambda^\prime_1}{1-\Lambda^\prime_2} Z^2 \right)  \,.
\eeq 
We have the following:
\begin{proposition}
\label{prop:Richelot-genus-two}
Assume that $\lambda_2, \lambda_3 \in \mathbb{P}^1 \backslash \{ 0, 1, \infty\}$ satisfy $\lambda_2 \not= \lambda_3^{\pm1}$ and $\lambda_2 \not =-1$, and  $\mathcal{C}_0$ is given by Equation~(\ref{Eq:Rosenhain_special}) with $\mathbf{A}=\operatorname{Jac}{(\mathcal{C}_0)}$. The genus-two curve 
\beq
\label{Eq:Rosenhain_special_dual}
  \mathcal{C}_0^\prime: \quad Y^2  =  X Z \big(X -Z\big) \big(X - \mu_2\mu_3 Z\big)  \big(X^2 - (\mu_2+\mu_3)X Z+  \mu_2\mu_3 Z^2\big)  \,,
\eeq
is smooth, admits an elliptic involution, and satisfies $\mathbf{A}^\prime \cong \mathbf{A}/G$ for $\mathbf{A}^\prime = \operatorname{Jac}{(\mathcal{C}^\prime_0)}$  and $G=\lbrace P_0, P_{12}, P_{35}, P_{46} \rbrace$. Here, $\mu_2, \mu_3$  are given by
\beq
\label{eqn:Rroots_2isog}
\begin{split}
\mu_2 & = \frac{\lambda_2\lambda_3+\lambda_2-\lambda_3-1+ 2 d}{\lambda_2\lambda_3+\lambda_2-\lambda_3-1- 2 d} \,,\\
\mu_3 & = \frac{\big(\lambda_2\lambda_3+\lambda_2-\lambda_3-1+ 2 d\big) \big((1+\lambda_2)^2 k_3 + (1-\lambda_2)d  -2 (1+\lambda_3) \lambda_2\big)}
{\big(\lambda_2\lambda_3+\lambda_2-\lambda_3-1- 2 d\big) \big((1+\lambda_2)^2 k_3 - (1-\lambda_2)d  -2 (1+\lambda_3) \lambda_2\big)} \,,
\end{split}  
\eeq
with $k_3^2 = \lambda_3$ and $d^2=(\lambda_2-\lambda_3)(\lambda_2\lambda_3-1)$. In particular, the curve $\mathcal{C}_0^\prime$ is isomorphic to $\widetilde{\mathcal{C}}^\prime_0$ in Equation~(\ref{eqn:LegendreSerreCurve_dual}) over a finite field extension of $\mathbb{Q}(\lambda_2, \lambda_3)$.
\end{proposition}
\begin{remark}
\label{rem:all_solutions}
In Equations~(\ref{eqn:Rroots_2isog}) the sign of $k_3$ and $d$ can be chosen independently. A change $(k_3, d) \mapsto (\mp k_3,\mp d)$ acts on the Rosenhain roots as $(\mu_2, \mu_3) \mapsto (1/\mu_2, 1/\mu_3)$. The sign change $(k_3, d) \mapsto (\pm k_3,\mp d)$ acts on the Rosenhain roots as $(\mu_2, \mu_3) \mapsto (\mu_2, \mu_2^2/\mu_3)$ or $(\mu_2, \mu_3) \mapsto (1/\mu_2, \mu_3/\mu_2^2)$. Isomorphic curves are also obtained by interchanging $\lambda_2 \leftrightarrow \lambda_3$, $k_2 \leftrightarrow k_3$, and changing $d \mapsto i d$ in Equations~(\ref{eqn:Rroots_2isog}) (while assuming $\lambda_3 \not =-1$ instead of $\lambda_2 \not =-1$).
\end{remark}
\begin{proof}
The result is obtained by applying the results of \cites{clingher2019isogenies, clingher2020isogenies} in the case $\lambda_1=\lambda_2\lambda_3$. In fact, \cite{clingher2020isogenies}*{Prop.~2.1} yields the curve
\beq
\label{Eq:Rosenhain_special_dual_pf}
  Y^2  =  \Big(X-2 (1+\lambda_2^2)\lambda_3 Z\Big)  \Big( c_2 X^2 + c_1 XZ + c_0 Z^2 \Big)
   \Big(\big(X+ (1+\lambda_2^2)\lambda_3 Z\big)^2 - 36 \lambda_2^2 \lambda_3^2 Z^2\Big) Z \,,
\eeq
with
\beqn
\begin{split}
 c_2 & = 1 + \lambda_3 - 2 k_3 \,,\\
 c_1 & = 2  \lambda_3  (\lambda_2^2 - 6 \lambda_2+1)(\lambda_3+1) + 8 (1+\lambda_2^2)\lambda_3 k_3 \,,\\
 c_0 & = \lambda_3^2 (\lambda_2^4+24 \lambda_2^3 -34 \lambda_2^2+24 \lambda_2+1)(\lambda_3+1)
 + 2  \lambda_3^2  (\lambda_2^2-5)(5 \lambda_2^2-1) k_3\,.
\end{split}  
\eeqn
If in Equation~(\ref{Eq:Rosenhain_special_dual_pf}) one substitutes
\beqn
 X = 3 (1 +\lambda_2)^2 \lambda_3 X' - (1 + 6 \lambda_2 + \lambda_2^2) \lambda_3 Z' \,, \qquad Y = 3^{\frac{5}{2}} (1+\lambda_2)^5(1-k_3) Y' \,, \quad Z = Z'\,,
\eeqn
one obtains the curve 
\beq
\label{eqn:curve_test}
 (Y')^2  =  X' Z' \big(X' -Z'\big) \big(X' - \lambda^\prime_2\lambda^\prime_3 Z'\big)  \big((X')^2 - (\lambda^\prime_2+\lambda^\prime_3)X' Z'+  \lambda^\prime_2\lambda^\prime_3 (Z')^2\big)  \,,
\eeq
where $\lambda^\prime_2, \lambda^\prime_3$  satisfy 
\beq
\begin{split}
 \lambda^\prime_2 \lambda^\prime_3  = \frac{4 \lambda_2}{(1+\lambda_2)^2} \,,\qquad
 \lambda^\prime_2 + \lambda^\prime_3  = 2 \left( 1 - \frac{(1-\lambda_2)^2(1+\lambda_3)}{(1+\lambda_2)^2 (1+\lambda_3 - 2 k_3)}\right) \,,
\end{split}  
\eeq
with $\lambda_3 = k_3^2$. Introducing new moduli $\mu_2, \mu_3$, given by
\beq
 \left(\lambda^\prime_2,  \lambda^\prime_3 \right) =\left( \frac{1-\mu_2}{1-\mu_3} \,, \ \frac{(1-\mu_2) \mu_3}{(1-\mu_3) \mu_2} \right) 
 \ \Leftrightarrow  \
 \left( \mu_2, \mu_3 \right) = \left( \frac{1-\lambda^\prime_2}{1-\lambda^\prime_3} \,, \ \frac{(1-\lambda^\prime_2) \lambda^\prime_3}{(1-\lambda^\prime_3) \lambda^\prime_2} \right) \,,
\eeq
one obtains an isomorphic curve over $\mathbb{Q}(\mu_2, \mu_3)$ given by Equation~(\ref{Eq:Rosenhain_special_dual}). Because of $\lambda_2 \not= \lambda_3^{\pm1}$ we have $d \not =0$. Since we also assumed $\lambda_2 \not =- 1$, we have $\mu_2, \mu_3  \in \mathbb{P}^1 \backslash \{0, 1, \infty\}$ and $\mu_2 \not= \mu_3^{\pm1}$ and the curve is smooth.
\par Denoting the Igusa-Clebsch invariants of the genus-two curves in Equation~(\ref{eqn:LegendreSerreCurve_dual}) and Equation~(\ref{eqn:curve_test}) by $[ I_2 : I_4 : I_6 : I_{10} ] \in \mathbb{WP}_{(2,4,6,10)}$ and $[ I'_2 : I'_4 : I'_6 : I'_{10} ]$, respectively, one checks that
\beq
[ I_2 : I_4 : I_6 : I_{10} ] = [ s^2 I'_2 \ : \ s^4I'_4 \ : \ s^6I'_6 \ : \ s^{10}I'_{10} ] = [ I'_2 : I'_4 : I'_6 : I'_{10} ] 
\eeq
with $s \in \mathbb{Q}(k_2, k_3)$.
\end{proof}
The last case of G\"opel groups to consider consists of four pairs that are pairwise interchanged by the action of the involution $\jmath$. These groups are given by
\beq
\label{eqn:Goepel_invariant2}
\begin{array}{cccc}
\Big\lbrace P_0, P_{12}, P_{34}, P_{56} \Big\rbrace,&\Big\lbrace P_0, P_{12}, P_{36}, P_{45} \Big\rbrace,&\Big\lbrace P_0, P_{13}, P_{24}, P_{56} \Big\rbrace,&\Big\lbrace P_0, P_{13}, P_{26}, P_{45}\Big\rbrace,  
\\
\jmath: \qquad \rotatebox{90}{$\leftrightarrow$} &  \rotatebox{90}{$\leftrightarrow$}  &  \rotatebox{90}{$\leftrightarrow$}  &  \rotatebox{90}{$\leftrightarrow$}  \\
\Big\lbrace P_0, P_{14}, P_{26}, P_{35} \Big\rbrace,&\Big\lbrace P_0, P_{16}, P_{24}, P_{35} \Big\rbrace,&\Big\lbrace P_0, P_{14}, P_{25}, P_{36} \Big\rbrace,&\Big\lbrace P_0, P_{16}, P_{25}, P_{34}\Big\rbrace.
\end{array}
\eeq
We have the following:
\begin{proposition}
\label{prop:H16}
For each pair of G\"opel subgroups $G_\pm$ in Equation~(\ref{eqn:Goepel_invariant2}) which are pairwise interchanged by the action of the involution $\jmath$, the abelian surfaces $\mathbf{A}_\pm^\prime =\mathbf{A}/G_\pm$ satisfy $\mathbf{A}_\pm^\prime  \cong \operatorname{Jac}{(\mathcal{C}_{0,\pm}^\prime)}$ for some smooth genus-two curves $\mathcal{C}_{0, \pm}^\prime$ with $\mathcal{C}^\prime_{0, +} \cong \mathcal{C}^\prime_{0, -}$ and $\operatorname{Jac}{(\mathcal{C}_{0, \pm}^\prime)} \in \mathcal{H}_{16}$. In particular, we have $\mathrm{T}_{\mathbf{A}_\pm^\prime} = H \oplus \langle 4 \rangle \oplus \langle -4 \rangle$.
\end{proposition}
\begin{proof}
The proof follows by constructing the Richelot curves $\mathcal{C}^\prime_{\pm}$ in Equation~(\ref{Richelot}) for each pair of G\"opel groups in Equation~(\ref{eqn:Goepel_invariant2}). Denoting the Igusa-Clebsch invariants of the genus-two curves $\mathcal{C}^\prime_{0, \pm}$ by $[ I_2^\pm : I_4^\pm : I_6^\pm : I_{10}^\pm ] \in \mathbb{WP}_{(2,4,6,10)}$, one checks that
\beq
[ I_2^+ : I_4^+ : I_6^+ : I_{10}^+ ] = [ I_2^- : I_4^- : I_6^- : I_{10}^- ] \,.
\eeq
One then checks by a direct computation that $Q\not= 0$ for each $\mathcal{C}^\prime_{0, \pm}$ whence it follows $\operatorname{Jac}{(\mathcal{C}^\prime_{0, \pm})} \not \in \mathcal{H}_4$. Using the $(2,2)$-isogeny $\Phi$ from Theorem~\ref{prop:isogeny_Delta}, we obtain the following compositions of isogenies
\beq
 \mathcal{E}_1 \times \mathcal{E}_2  
 \ \overset{\Phi}{\longrightarrow} \ 
 \mathbf{A}  =\operatorname{Jac}{(\mathcal{C}_0)}
 \ \overset{\Psi}{\longrightarrow} \ 
 \mathbf{A}_{n,\pm}^\prime  =\operatorname{Jac}{(\mathcal{C}_{0, \pm}^\prime)} = \mathbf{A}/G_{\pm} \,.
\eeq
Since $\Psi$ is obtained from the projection onto the quotient by a G\"opel group, it is a $(2,2)$-isogeny as well. Thus, $\Psi \circ \Phi$ is a $(4,4)$-isogeny.  Using Theorem~\ref{prop:isogeny_Delta} once more it follows that $\operatorname{Jac}{(\mathcal{C}_{0, \pm}^\prime)} \in \mathcal{H}_{16}$.
\end{proof}
We have the following:
\begin{corollary}
\label{cor:special_G_group}
For a genus-two curve $\mathcal{C}_0$ admitting an elliptic involution, there is exactly one G\"opel group $G \leqslant \mathbf{A}[2]$ of $\mathbf{A}=\operatorname{Jac}{(\mathcal{C}_0)}$ such that $\mathbf{A}/G \cong \mathcal{E}_1 \times \mathcal{E}_2$ for two elliptic curves $\mathcal{E}_1, \mathcal{E}_2$. In particular, $G$ is the unique G\"opel group of $\mathbf{A}=\operatorname{Jac}{(\mathcal{C}_0)}$ which is fixed under the action of the involution $\jmath$ on $\mathbf{A}[2]$. 
\end{corollary}
\begin{remark}
The special G\"opel group from Corollary~\ref{cor:special_G_group} has one non-trivial element in common with each G\"opel group in Proposition~\ref{prop:H4} whose associated $(2,2)$-isogenous Richelot curve has a Jacobian in $\mathcal{H}_4$. In contrast, the intersection with each G\"opel group in Proposition~\ref{prop:H16} is trivial. 
\end{remark}
\begin{remark}
The results above determine the $(2,2)$-isogenies for a point $\operatorname{Jac}{(\mathcal{C}_0)}$ in $\mathcal{H}_4$. There is a similar result for a general point $\mathcal{E}_1 \times \mathcal{E}_2$ in $\mathcal{H}_1$. Obviously, $3 \cdot 3 = 9$ such isogenies have images that are products of elliptic curves, so again in $\mathcal{H}_1$. The remaining six are the ones in Corollary~\ref{cor:2isogeny} and give points in $\mathcal{H}_4$.
\end{remark}
\section{Kummer surfaces and isogenies}
\label{sec:KummerSurfaces}
On a principally polarized abelian surface $(\mathbf{A}, \mathcal{L})$ with the minus identity involution denoted by $-\mathbb{I}$, one can always choose a theta divisor $\mathcal{L} = \mathcal{O}_\mathbf{A}(\Theta)$ to satisfy $(-\mathbb{I})^* \Theta=\Theta$, that is, to be a \emph{symmetric theta divisor}. For an irreducible principal polarization the abelian surface $\mathbf{A}$ then maps to the complete linear system $|2\Theta|$, and the rational map $\varphi_{\mathcal{L}^2}: \mathbf{A} \rightarrow  \mathbb{P}^3$ associated with the line bundle $\mathcal{L}^2$ factors via an embedding through the projection $\mathbf{A} \rightarrow  \mathbf{A}/\langle -\mathbb{I} \rangle$; see \cite{MR2062673}. In this way, one identifies $\mathbf{A}/\langle -\mathbb{I} \rangle$ with its image  $\varphi_{\mathcal{L}^2}(\mathbf{A})$  in $\mathbb{P}^3$, which is a singular quartic projective surface $\mathcal{K}_\mathbf{A}$ with sixteen ordinary double points, called a \emph{nodal surface}. The map from $\mathbf{A}$ onto its image $\mathcal{K}_\mathbf{A} \subset \mathbb{P}^3$ is two-to-one, except on sixteen points of order two where it is injective. Its minimum resolution is the Kummer surface $\operatorname{Kum}(\mathbf{A})$ associated with the principally polarized abelian surface $(\mathbf{A}, \mathcal{L})$.  Due to the $16_6$ configuration on $\mathcal{K}_\mathbf{A}$, there are six hyperplanes containing the double point $\varphi_{\mathcal{L}^2}(0)$ and touching $\mathcal{K}_\mathbf{A}$ along a double conic. The linear projection with center $\varphi_{\mathcal{L}^2}(0)$ maps these six hyperplanes onto the six lines $\gerade_1, \dots, \gerade_6$ in a projective plane $\mathbb{P}^2$. The configuration $(\mathbb{P}^2; \gerade_1, \dots, \gerade_6)$ is called the \emph{Kummer plane} associated with $(\mathbf{A}, \mathcal{L})$. The rich geometry of six-line configurations in the Kummer plane, as well as their strong connection with theta functions have been the subject of multiple studies \cites{MR1097176,MR0419459,MR0296076,MR3366121,MR1182682,MR3712162,MR2369941} over the last century and a half.  
\par In this section we will construct the Kummer surfaces associated with the product of two elliptic curves and the Jacobian of a genus-two curve as a double quadric, a double sextic, and a quartic projective surface. We will also construct several Jacobian elliptic fibrations on these surfaces.  As a reminder, an elliptic surface is a (relatively) minimal complex surface $\mathcal{X}$, together with a \emph{Jacobian elliptic fibration}, that is, a holomorphic map $\pi_\mathcal{X}: \mathcal{X} \to \mathbb{P}^1$ to $\mathbb{P}^1$ such that the general fiber is a smooth curve of genus one together with a distinguished section $\Sigma: \mathbb{P}^1 \rightarrow  \mathcal{X}$ that marks a smooth point in each fiber. To each Jacobian elliptic fibration $\pi_\mathcal{X}: \mathcal{X} \rightarrow  \mathbb{P}^1$ there is an associated Weierstrass model obtained by contracting all components of reducible fibers not meeting $\Sigma$. By a slight abuse of notation we will denote the elliptic surface and its associated Weierstrass model  by the same symbol. The complete list of possible singular fibers has been given by Kodaira~\cite{MR0165541}.  It encompasses two infinite families $(I_n, I_n^*, n \ge0)$ and six exceptional cases $(II, III, IV, II^*, III^*, IV^*)$. The Weierstrass model of a smooth K3 surface can always be written in the form
\beq
\label{eqn:WEQ}
 y^2 z = 4 x^3 -g_2(v) \, x z^2 - g_3(v) z^3 \,,
\eeq 
where $v$ is a suitable coordinate on the base curve $\mathbb{P}^1_v$, and $g_2$ and $g_3$ are polynomials of degree $8$ and $12$ in homogeneous coordinates on $\mathbb{P}^1_v$, respectively. The section is given by the point at infinity in each smooth fiber. We denote the Mordell-Weil group of sections on the Jacobian elliptic surface $\pi_\mathcal{X}: \mathcal{X} \to \mathbb{P}^1$ by $\operatorname{MW}(\mathcal{X},\pi_\mathcal{X})$. If a Jacobian elliptic fibration admits in addition a two-torsion section $T \in \operatorname{MW}(\mathcal{X},\pi_\mathcal{X})$, we use a change of coordinates to write Equation~\eqref{eqn:WEQ} in the form
\beq
\label{eqn:WEQ2}
 \mathcal{X}: \quad y^2 z = x^3 + p_1(v) \, x^2z + p_2(v) \, x z^2\,,
\eeq 
where $p_1$ and $p_2$ are polynomials of degree $4$ and $8$ in homogeneous coordinates on the base curve $\mathbb{P}^1_v$, respectively.  A natural holomorphic two-form on the K3 surface, obtained as the minimal resolution of the Weierstrass model, is the pullback of the holomorphic two-form $dv \wedge dx/y$ in the (affine coordinate) chart $z=1$ in Equation~(\ref{eqn:WEQ2}).
We make the following:
\begin{remark}
Oguiso classified the Jacobian elliptic fibrations on the Kummer surfaces associated with $\mathcal{E}_1 \times \mathcal{E}_2$ where the elliptic curves $\mathcal{E}_i$ for $i=1,2$ are not mutually isogenous in \cite{MR1013073}. Kuwata and Shioda furthered Oguiso's work in \cite{MR2409557} where they computed elliptic parameters and Weierstrass equations for all fibrations, and analyzed the reducible fibers and Mordell-Weil lattices.  Similarly, all inequivalent Jacobian elliptic fibrations on the Kummer surface associated with $\operatorname{Jac}{(\mathcal{C})}$ for a general genus-two curve $\mathcal{C}$ were determined explicitly by Kumar in \cite{MR3263663}. In particular, Kumar computed elliptic parameters and  Weierstrass equations for all twenty five different fibrations that appear, and analyzed the reducible fibers and  Mordell-Weil lattices. 
\end{remark}
\subsection{Kummer surfaces associated with two elliptic curves}
In this section we will concerned with the Kummer surface $\operatorname{Kum}(\mathcal{E}_1\times \mathcal{E}_2)$, i.e., the minimal resolution of quotient surface of the product abelian surface  $\mathcal{E}_1\times \mathcal{E}_2$ by the inversion automorphism. Unless stated otherwise we will assume that the two elliptic curves are not mutually isogenous. 
\par Let us define a double quadric surface closely related to  $\mathcal{E}_1 \times  \mathcal{E}_2$ for the elliptic curves given by Equation~(\ref{eqn:EC}). In general, a  \emph{double quadric surface} is obtained as the double cover branched along a locus of bi-degree $(4,4)$ in $\mathbb{P}^1 \times \mathbb{P}^1$; see \cite{MR1922094}. It is well known that the minimal resolution of any double quadric is a K3 surface. In our situation, we identify $\mathbb{P}^1 = \mathbb{P}(x_l, z_l)$ for $l=1, 2$ and define the special double quadric surface $\mathcal{Z}$ using the equation
\beq
\label{eqn:Kummer44}
 \mathcal{Z}: \quad y_{1,2}^2 = x_1 z_1 \big(x_1-z_1\big) \big(x_1- \Lambda_1 z_1\big) \,  x_2 z_2 \big(x_2-z_2\big) \big(x_2- \Lambda_2 z_2\big) \,.
\eeq
There is a natural projection map $\pi: \mathcal{E}_1 \times  \mathcal{E}_2 \dasharrow \mathcal{Z}$ given by $y_{1,2} = z_1 z_2 y_1y_2$, invariant under the action of $\imath_{\mathcal{E}_1} \times \imath_{\mathcal{E}_2}$.  It follows that the minimal resolution of $\mathcal{Z}$ is the Kummer surface $\operatorname{Kum}(\mathcal{E}_1\times \mathcal{E}_2)$. 
\par A holomorphic two-form on $\mathcal{Z}$ is given by $\omega_\mathcal{Z} = dx_1 \wedge dx_2/y_{1,2}$ in the coordinate chart $z_1=z_2=1$ in Equation~(\ref{eqn:Kummer44}). The regular two-from on $\mathcal{E}_1 \times \mathcal{E}_2$, given by the outer tensor product of the elliptic-curve holomorphic one-forms, i.e., 
\beq
\label{eqn:two-form}
 \frac{dx_1}{y_1} \wedge \frac{dx_2}{y_2} := \frac{dx_1}{y_1} \boxtimes \frac{dx_2}{y_2}  \,,
\eeq
descends to the quotient $(\mathcal{E}_1 \times \mathcal{E}_2)/\langle - \mathbb{I}\rangle$ and coincides with $\omega_\mathcal{Z}$. Here, we use the notation $\eta_1 \boxtimes \eta_2 = (\pi_1^*\eta_1) \otimes  (\pi_2^*\eta_2)$ where $\pi_l$ is the canonical projection onto $\mathcal{E}_l$ for $l=1, 2$. Let $\epsilon:  \widehat{\mathcal{Z}} \to \mathcal{Z}$ be the minimal resolution. We obtain a holomorphic two-form $\omega_{\operatorname{Kum}(\mathcal{E}_1\times \mathcal{E}_2)}$ on $\operatorname{Kum}(\mathcal{E}_1\times \mathcal{E}_2)$ by pullback, i.e.,
\beq
\label{eqn:two-form_E1E2}
 \omega_{\operatorname{Kum}(\mathcal{E}_1\times \mathcal{E}_2)} = \epsilon^*\omega_\mathcal{Z} \,, \quad \text{with} \
  \epsilon \colon \; \operatorname{Kum}(\mathcal{E}_1\times \mathcal{E}_2) = \widehat{\mathcal{Z}} \  \to \ \mathcal{Z} \,.
\eeq
We also introduce the quadratic twist of $\mathcal{Z}$, given by
\beq
\label{eqn:Kummer44_b}
 \mathcal{Z}^{(\varepsilon)}: \quad  \hat{y}_{1,2}^2 =  \varepsilon  x_1 z_1 \big(x_1-z_1\big) \big(x_1- \Lambda_1 z_1\big) \,  x_2 z_2 \big(x_2-z_2\big) \big(x_2- \Lambda_2 z_2\big) \,,
\eeq
such that for general $\varepsilon \in \mathbb{Q}$ the surfaces  $\mathcal{Z}$ and $\mathcal{Z}^{(\varepsilon)}$ are isomorphic only over $\mathbb{Q}(\sqrt{\varepsilon})$. We refer to the minimal resolution of  $\mathcal{Z}^{(\varepsilon)}$ as twisted Kummer surface and denote it by $\operatorname{Kum}(\mathcal{E}_1\times \mathcal{E}_2)^{(\varepsilon)}$. Holomorphic two-forms on the twisted double quadric surface and the minimal resolution are then obtained analogously to Equation~(\ref{eqn:two-form_E1E2}). We make the following:
\begin{remark}
Projecting to either $\mathbb{P}^1$ in Equation~(\ref{eqn:Kummer44}) defines a Jacobian elliptic fibration on the Kummer surface $\operatorname{Kum}(\mathcal{E}_1\times \mathcal{E}_2)$ denoted $\mathcal{J}_4$ in \cite{MR2409557}, i.e., a fibration with the singular fibers $4 I_0^*$ and a Mordell-Weil group of sections $(\mathbb{Z}/2\mathbb{Z})^2$.
\end{remark}
\begin{remark}
\label{rem:trans_KumE1E2}
If the elliptic curves $\mathcal{E}_l$ for $l=1,2$ are not mutually isogenous, then the Kummer surface $\operatorname{Kum}(\mathcal{E}_1\times \mathcal{E}_2)$ has Picard rank 18. It follows from \cite{MR728142} that the transcendental lattice is given by
\beq
  \operatorname{T}_{\operatorname{Kum}(\mathcal{E}_1\times \mathcal{E}_2)} \cong \operatorname{T}_{\mathcal{E}_1 \times \mathcal{E}_2}(2) = H(2) \oplus H(2) \,.
\eeq
\end{remark}
\par Moreover, we will consider  $\operatorname{Kum}(\mathcal{E}^\prime_1\times \mathcal{E}^\prime_2)$, i.e., the minimal resolution of the quotient surface of the product abelian surface  $\mathcal{E}^\prime_1\times \mathcal{E}^\prime_2$ by the inversion automorphism -- associated with the two-isogenous elliptic curves $\mathcal{E}^\prime_l$ in Equation~(\ref{eqn:EC_dual}).  The associated double quadric surface is given by
\beq
\label{eqn:Kummer44_dual}
 \mathcal{Z}^\prime: \quad Y_{1,2}^2  = X_1  Z_1  \Big(X_1^2  + 2 (1-2\Lambda_1)  X_1 Z_1 +   Z_1^2 \Big) \,
 X_2  Z_2  \Big(X_2^2  + 2 (1-2\Lambda_2)  X_2   Z_2 +   Z_2^2 \Big) \,.
\eeq
On $\operatorname{Kum}(\mathcal{E}^\prime_1\times \mathcal{E}^\prime_2)$ a holomorphic two-form $\omega_{\operatorname{Kum}(\mathcal{E}^\prime_1\times \mathcal{E}^\prime_2)}$ is obtained as the pullback of the two-form $\omega_{\mathcal{Z}^\prime} = dX_1 \wedge dX_2/Y_{1,2}$ in the chart $Z_1=Z_2=1$ on $\mathcal{Z}^\prime$ in Equation~(\ref{eqn:Kummer44_dual}). The  Cartesian product of the dual two-isogenies $\chi_{\mathcal{E}_1}^\prime \times \chi_{\mathcal{E}_2}^\prime$ of elliptic curves induces a rational cover $\mathcal{Z}^\prime \rightarrow \mathcal{Z}$. We have the following:
\begin{lemma}
\label{lem:map_phi}
The map $\chi: \mathcal{Z}^\prime \dasharrow \mathcal{Z}^{(2^4)}$, given by
\beq
  \chi: \quad \Big(X_1, Z_1, X_2, Z_2, \;  Y_{1,2} \Big) \ \ \mapsto \ \ \Big( x_1, z_1, x_2, z_2, \; \hat{y}_{1,2} \Big) 
\eeq
with 
\beq
\label{eqn:RatMapKumE1E2}
\begin{split}
 x_1 & = 4 \Lambda_1 X_1 Z_1 \,, \qquad    z_1 =  (X_1+Z_1)^2 \,,\\
 x_2 & = 4 \Lambda_2 X_2 Z_2 \,, \qquad  z_2 =  (X_2+Z_2)^2 \,,\\
 \hat{y}_{1,2} & =  16 \,  \Lambda_1 \Lambda_2 (X_1^2 -Z_1^2) (X_2^2 - Z_2^2) Y_{1,2} \,,
\end{split}
\eeq
is a rational covering map defined over $\mathbb{Q}(\Lambda_1, \Lambda_2)$. For the holomorphic two-forms $\omega_{\mathcal{Z}^{(2^4)}} = dx_1 \wedge dx_2/\hat{y}_{1,2}$ in the chart $z_1=z_2=1$ on $\mathcal{Z}^{(2^4)}$ and $\omega_{\mathcal{Z}^\prime} = dX_1 \wedge dX_2/Y_{1,2}$ in the chart $Z_1=Z_2=1$ on $\mathcal{Z}^\prime$, it follows $\omega_{\mathcal{Z}^\prime} = \chi^* \omega_{\mathcal{Z}^{(2^4)}} $.
\end{lemma}
\begin{proof}
The proof follows by an explicit computation using Equations~(\ref{eqn:RatMapKumE1E2}).
\end{proof}
\subsubsection{Certain elliptic fibrations}
On the double quadric surface $\mathcal{Z}$, a non-isotrivial elliptic fibration, with fibers over $\mathbb{P}^1=\mathbb{P}(u_1,u_2)$ embedded into $\mathbb{P}^2 = \mathbb{P}(X, Y, Z)$, is given by the Weierstrass model
\beq
\label{eqn:Y0}
\begin{split}
  Y^2 Z &= X \, \left( X + \frac{1}{2} u_1 u_2 \big(u_1 - u_2 \big) \big( (\rho_1 -\rho_2) u_1 - (\rho_1 + \rho_2) u_2 \big) Z\right) \\
 & \times   \left( X + \frac{1}{2} u_1 u_2  \big( (\rho_1 - \rho_2) u_1^2 -2  (\rho_1+ 1) u_1 u_2 + (\rho_1+\rho_2) u_2^2 \big) Z \right)\,.
\end{split} 
\eeq
The discriminant function of the elliptic fibration is
\beqn
   2^{-4} u_1^8 u_2^8 (u_1-u_2)^2  \big( (\rho_1 - \rho) u_1 - (\rho_1+\rho_2) u_2 \big)^2  \big( (\rho_1 - \rho) u_1^2 - 2 (\rho_1+1) u_1 u_2 + (\rho_1 + \rho) u_2^2 \big)^2.
\eeqn 
We have the immediate:
\begin{lemma}
Equation~(\ref{eqn:Y0}) defines a Jacobian elliptic fibration with the singular fibers $2 I_2^* + 4 I_2$ and the Mordell-Weil group of sections $(\mathbb{Z}/2\mathbb{Z})^2$.
\end{lemma}
In fact,  Equation~(\ref{eqn:Y0}) defines the Jacobian elliptic fibration on $\operatorname{Kum}(\mathcal{E}_1\times \mathcal{E}_2)$ denoted $\mathcal{J}_6$ in \cite{MR2409557}.  We have the following:
\begin{proposition}
\label{prop:equivalence00}
For $\Lambda_1, \Lambda_2  \in \mathbb{P}^1 \backslash \{ 0, 1, \infty\}$ and parameters $\rho_1, \rho_2$ given by
\beqn
 \rho_1 = \Lambda_1 + \Lambda_2 - 2 \Lambda_1 \Lambda_2 -1\,, \qquad \rho_2 =  1 - \Lambda_1 - \Lambda_2\,,
\eeqn
the double quadric surface $\mathcal{Z}$~(\ref{eqn:Kummer44}) and the Jacobian elliptic surface~(\ref{eqn:Y0}) are birational equivalent over $\mathbb{Q}(\Lambda_1, \Lambda_2)$. 
\end{proposition}
\begin{proof}
In the chart $Z=u_2=1$ in Equation~(\ref{eqn:Y0}) and $z_1=z_2=1$ in Equation~(\ref{eqn:Kummer44}), a birational map is given by
\beq
\label{eqn:transfoY_Z}
\begin{split}
 u_1 = \frac{x_1 x_2}{(x_1-1)(x_2-1)} \,, \qquad
 Y= \frac{x_1^2 x_2 (x_1 +x_2-1)  (x_1x_2-\Lambda_2(x_1+x_2)+\Lambda_2) y_{1,2}}{(x_1-1)^4(x_2-1)^5} \,,\\
 X=  \frac{x_1 x_2(x_1+x_2-1) (x_1-\Lambda_1)(x_1x_2-\Lambda_2(x_1+x_2)+\Lambda_2)}{(x_1-1)^3(x_2-1)^3} \,.
\end{split}
\eeq
\end{proof}
The holomorphic two-forms are related as follows:
\begin{corollary}
\label{lem:2forms000}
The birational equivalence of Proposition~\ref{prop:equivalence00} identifies the holomorphic two-form $du_1 \wedge dX/Y$ in the chart $Z=u_2=1$ in Equation~(\ref{eqn:Y0}) with $\omega_\mathcal{Z} = dx_1 \wedge dx_2/y_{1,2}$ in the chart $z_1=z_2=1$ in Equation~(\ref{eqn:Kummer44}).
\end{corollary}
\begin{proof}
The proof follows by an explicit computation using the transformation provided in the proof of Proposition~\ref{prop:equivalence00}.
\end{proof}
\par On the double quadric surface $\mathcal{Z}^\prime$, another elliptic fibration, with fibers over $\mathbb{P}^1=\mathbb{P}(v_1,v_2)$ embedded into $\mathbb{P}^2 = \mathbb{P}(X, Y, Z)$, is given by the Weierstrass model
\beq
\label{eqn:Yvee0}
Y^2 Z = X^3  - 2 v_2 \big(2 v_1 - v_2 \big) \Big( v_1^2 + 2 \rho_1 v_1 v_2 + \rho_2^2 v_2^2 \Big) X^2 Z 
+ v_2^4   \Big( v_1^2 + 2 \rho_1 v_1 v_2 + \rho_2^2 v_2^2 \Big)^2 X Z^2 \,.
\eeq
The discriminant function of the fibration is $16 v_1 v_2^{10} (v_1 - v_2)  ( v_1^2 +  2 \rho_1 v_1v_2 + \rho^2_2 v_2^2 )^6$. We have the immediate:
\begin{lemma}
\label{lem:fib_Yvee0}
Equation~(\ref{eqn:Yvee0}) defines a Jacobian elliptic fibration with the singular fibers $I_4^* + 2 I_1 + 2 I_0^*$ and the Mordell-Weil group of sections $\mathbb{Z}/2\mathbb{Z}$.
\end{lemma}
Equation~(\ref{eqn:Yvee0}) defines the Jacobian elliptic fibration on $\operatorname{Kum}(\mathcal{E}^\prime_1\times \mathcal{E}^\prime_2)$ denoted $\mathcal{J}_7$ in \cite{MR2409557}.  However, to construct a birational equivalence between the Jacobian elliptic fibration and $\mathcal{Z}^\prime$, one has to use a finite field extension $\mathbb{Q}(\kappa_1, \kappa_2)$ with  $\Lambda_1=1/(1 -\kappa_1^2)$ and $\Lambda_2= 1/(1-\kappa_2^2)$. We have the following:
\begin{proposition}
\label{prop:equivalence000}
For $\Lambda_1, \Lambda_2 \in \mathbb{P}^1 \backslash \{ 0, 1, \infty\}$ let parameters $\rho_1, \rho_2$ be given by
\beqn
 \rho_1 = \Lambda_1 + \Lambda_2 - 2\Lambda_1 \Lambda_2 -1\,, \qquad \rho_2 = 1-\Lambda_1 - \Lambda_2\,.
\eeqn
The double quadric surface $\mathcal{Z}'$~(\ref{eqn:Kummer44_dual}) and the Jacobian elliptic surface~(\ref{eqn:Yvee0}) are birational equivalent over $\mathbb{Q}(\kappa_1, \kappa_2)$ with $\Lambda_l=1/(1 -\kappa_l^2)$ for $l= 1,2$.
\end{proposition}
\begin{proof}
In the charts $v_2=1, Z=1$ in Equation~(\ref{eqn:Yvee}) and $Z_1=Z_2=1$ in Equation~(\ref{eqn:Kummer44_dual}), a birational equivalence is given by
\beq
\label{eqn:transfoX3vee_Zvee}
\begin{split}
 \scalemath{0.8}{
	v_1  =  \frac{(\kappa_1^2-1)(\kappa_1 \kappa_2^2-1)X_1X_2-\kappa_1(\kappa_1+1)(\kappa_2^2-1)X_1+ \kappa_1(\kappa_1-1)(\kappa_2^2-1)X_2^2 +(\kappa_1^2-1)(\kappa_1\kappa_2^2+1)X_2 }{(1-\kappa_1^2)(1-\kappa_2^2) X_2 ((\kappa_1+1) X_1 +\kappa_1 -1)} \,,}\\
 \scalemath{0.8}{
 	X  =  \frac{((\kappa_1+1)(\kappa_2-1) X_1 - (\kappa_1-1)(\kappa_2+1) X_2 )  ((\kappa_1+1)(\kappa_2+1) X_1 - (\kappa_1-1)(\kappa_2-1) X_2 )}{(1-\kappa_1^2)^2(1-\kappa_2^2)^2} \,}\\
 \scalemath{0.8}{
 	\times \frac{((\kappa_2+1)X_2+(\kappa_2-1))((\kappa_2-1)X_2+(\kappa_2+1))\kappa_1^2 X_1}{((\kappa_1+1) X_1 +\kappa_1-1)^2X_2^3}\,,}\\
\scalemath{0.8}{
 	Y = - \frac{ ((\kappa_1+1)(\kappa_2-1) X_1 - (\kappa_1-1)(\kappa_2+1) X_2 )^2  ((\kappa_1+1)(\kappa_2+1) X_1 - (\kappa_1-1)(\kappa_2-1) X_2 )^2 }{(1-\kappa_1^2)^3(1-\kappa_2^2)^3 }   \,}\\
 \scalemath{0.8}{
 	\times \frac{((\kappa_2+1)X_2+(\kappa_2-1))((\kappa_2-1)X_2+(\kappa_2+1))\kappa_1^3 Y_{1,2}}{((\kappa_1+1) X_1 +\kappa_1-1)^4X_2^5}  \,.}
\end{split}
\eeq
\end{proof}
The holomorphic two-forms are related as follows:
\begin{corollary}
\label{lem:2forms00}
The birational equivalence of Proposition~\ref{prop:equivalence000} identifies the holomorphic two-form $dv_1 \wedge dX/Y$ in the chart $Z=v_2=1$ in Equation~(\ref{eqn:Yvee0}) with $\omega_{\mathcal{Z}^\prime} = dX_1 \wedge dX_2/Y_{1,2}$ in the chart $Z_1=Z_2=1$ in Equation~(\ref{eqn:Kummer44_dual}).
\end{corollary}
\begin{proof}
The proof follows by an explicit computation using the transformation provided in the proof of Proposition~\ref{prop:equivalence000}.
\end{proof}
\subsection{Jacobian Kummer surfaces}
The Kummer surface $\operatorname{Kum}(\operatorname{Jac}{\mathcal{C}})$ is the minimal resolution of the quotient surface of the Jacobian $\operatorname{Jac}{(\mathcal{C})}$ by the inversion automorphism where $\mathcal{C}$ is a generic, smooth genus-two curve. 
\par Since the Jacobian $\operatorname{Jac}{(\mathcal{C})}$ is birational to the symmetric product $\operatorname{Sym}^2(\mathcal{C})$, let us consider the function field  of $\operatorname{Sym}^2(\mathcal{C}) = (\mathcal{C}\times\mathcal{C})/\langle \sigma \rangle$ where $\sigma$ interchanges two copies of a generic genus-two curve $\mathcal{C}$. In particular, for a smooth genus-two curve $\mathcal{C}$ in Rosenhain form~(\ref{Eq:Rosenhain}), the function field of the variety $\operatorname{Sym}^2(\mathcal{C})/\langle \imath_\mathcal{C} \times  \imath_\mathcal{C} \rangle$ -- where $\imath_\mathcal{C} \times  \imath_\mathcal{C}$ is the involution on $\operatorname{Sym}^2(\mathcal{C})$ induced by the hyperelliptic involution $\imath_\mathcal{C}$ on $\mathcal{C}$ -- is generated by $z_1=Z^{(1)}Z^{(2)}$, $z_2=X^{(1)}Z^{(2)}+X^{(2)}Z^{(1)}$, $z_3=X^{(1)}X^{(2)}$, and $\tilde{z}_4=Y^{(1)}Y^{(2)}$, subject to the relation
\beq
\label{kummer_middle}
  \mathcal{W}: \quad \tilde{z}_4^2 = z_1 z_3 \big(  z_1  - z_2 +  z_3 \big)  \prod_{i=1}^3 \big( \lambda_i^2 \, z_1  -  \lambda_i \, z_2 +  z_3 \big) \;.
\eeq
Equation~(\ref{kummer_middle}) is a special case of a \emph{double sextic surface}, i.e., the double cover of $\mathbb{P}^2 = \mathbb{P}(z_1, z_2, z_3)$ branched on the union of six lines (and hence over a sextic curve). For a configuration of six lines in general position, the minimal resolution of a double sextic surface is always a K3 surface; see \cite{MR1922094}. Equation~(\ref{kummer_middle}) is known as \emph{Shioda sextic} associated with $\operatorname{Jac}{(\mathcal{C})}$; see \cite{MR2296439}.  It follows that the minimal resolution of $\mathcal{W}$ is the Kummer surface $\operatorname{Kum}(\operatorname{Jac}{\mathcal{C}})$. 
\par In Equation~(\ref{kummer_middle}) the double cover is branched along the six lines $\gerade_i$ with $1 \le i \le 6$, given by
\beq
\label{eqn:6lines}
 \lambda_i^2 \, z_1  -  \lambda_i \, z_2 +  z_3 =0 \; \text{with $1\le i \le 3$}, \quad  z_1=0,\quad  z_1  - z_2 +  z_3=0,  \quad  z_3=0 \,.
\eeq
The six lines are tangent to the common conic $K_2= z_2^2 - 4 \, z_1 z_3=0$. Conversely, it is easy to see that any six lines tangent to a common conic can always be brought into the form of Equations~(\ref{eqn:6lines}) using a projective transformation.  Humbert proved that the Kummer plane $(\mathbb{P}^2; \gerade_1, \dots, \gerade_6)$ inherits essential information of the principally polarized abelian surface $\mathbf{A}=\operatorname{Jac}{(\mathcal{C})}$ itself; see \cite{Humbert1901}. A picture is provided in Figure~\ref{fig:6Lines}.
\begin{figure}[ht]
\scalemath{0.85}{
$$
  \begin{xy}
   <0cm,0cm>;<1.5cm,0cm>:
    (2,-.3)*++!D\hbox{$\gerade_1$},
    (1.1,0.25)*++!D\hbox{$\gerade_2$},
    (1.1,1.3)*++!D\hbox{$\gerade_3$},
    (2,1.8)*++!D\hbox{$\gerade_4$},
    (2.8,1.3)*++!D\hbox{$\gerade_5$},
    (2.95,0.2)*++!D\hbox{$\gerade_6$},
    (2,0.7)*++!D\hbox{$K_2=0$},
    (.5,0.18);(3.5,0.18)**@{-},
    (0.5,2);(1.6,0)**@{-},
    (3.1,2);(2.6,0)**@{-},
    (0.7,.5);(2,2.5)**@{-},
    (0.5,1.82);(3.5,1.82)**@{-},
    (2,2.7);(3.45,0)**@{-},
    (2,1)*\xycircle(.8,.8){},
  \end{xy}
  $$}
\caption{Double cover branched along reducible sextic}
\label{fig:6Lines}
\end{figure}
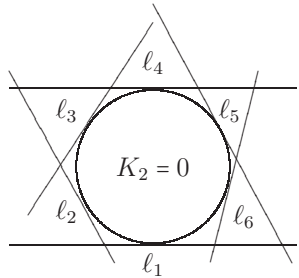
\begin{remark}
\label{rem:config}
In \cite{MR1953527} the geometry of the Kummer plane was determined over several Humbert surfaces $\mathcal{H}_\Delta$. In particular, the following statement were proven~\cite{MR1953527}*{Cor.~7.2} and~\cite{MR1953527}*{Cor.~7.3}:
\begin{enumerate}
\item[$\Delta=\phantom{1}4$:] if $(\mathbf{A}, \mathcal{L}) \in \mathcal{H}_4$ if and only if (numbering the six lines on its Kummer plane $(\mathbb{P}^2; \gerade_1, \dots, \gerade_6)$ suitably) the three points $\gerade_1 \cap \gerade_2$, $\gerade_3 \cap \gerade_4$, $\gerade_5 \cap \gerade_6$ are collinear.
\item[$\Delta=16$:] if $(\mathbf{A}, \mathcal{L}) \in \mathcal{H}_{16}$ then its Kummer plane $(\mathbb{P}^2; \gerade_1, \dots, \gerade_6)$ admits a cubic passing smoothly through three of the 15 points $\gerade_m \cap \gerade_n$ and touching the singular lines $\gerade_n$ in the remaining intersection points with even multiplicity. Conversely, if $(\mathbb{P}^2; \gerade_1, \dots, \gerade_6)$ admits such a curve, then $(\mathbf{A}, \mathcal{L}) \in \mathcal{H}_\Delta$ with $\Delta \in \{ 4, 8, 12, 16, 20\}$.
\end{enumerate} 
For the transcendental lattice it follows $\mathrm{T}_{\operatorname{Kum}(\mathbf{A})} \cong \mathrm{T}_{\mathbf{A}}(2)$ \cite{MR728142}*{Thm.~10}. In particular we have
\beq
\begin{split}
 \mathrm{T}_{\operatorname{Kum}(\mathbf{A})}  & = H(2) \oplus \langle 4 \rangle \oplus \langle -4 \rangle \ \text{for} \ (\mathbf{A}, \mathcal{L}) \in \mathcal{H}_4\,, \\
 \mathrm{T}_{\operatorname{Kum}(\mathbf{A})}  & =  H(2) \oplus \langle 8 \rangle \oplus \langle -8 \rangle \ \text{for} \ (\mathbf{A}, \mathcal{L}) \in \mathcal{H}_{16}\,.
\end{split} 
\eeq
Generally, for $\mathbf{A} \in \mathcal{H}_{\delta^2}$ it follows from \cite{MR728142} that $\operatorname{Kum}(\mathbf{A})$ has Picard rank 18, and the transcendental lattice is given by
\beq
  \operatorname{T}_{\operatorname{Kum}(\mathbf{A})} \cong \operatorname{T}_\mathbf{A}(2) = H(2) \oplus \langle 2\delta \rangle \oplus   \langle - 2\delta \rangle\,.
\eeq
\end{remark}
\subsubsection{Closely related normal forms}
The Shioda sextic is closely related to two quartic equations, known as the \emph{Baker quartic} and \emph{Cassels-Flynn quartic}. The Baker quartic is obtained directly from the Shioda sextic $\mathcal{W}$ in Equation~(\ref{kummer_middle}). Using parameters
\beq
\label{eqn:Ls}
\begin{array}{lclclcl}
 L_4 & = & 1 + \lambda_1 + \lambda_2 + \lambda_3 \,, && L_3 & = &  \lambda_1 + \lambda_2 + \lambda_3 +  \lambda_1 \lambda_2 + \lambda_2 \lambda_3 + \lambda_1 \lambda_3 \,,\\
 L_1 & = & \lambda_1 \lambda_2 \lambda _3 \,, && L_2 & = &   \lambda_1 \lambda_2 + \lambda_2 \lambda_3 + \lambda_1 \lambda_3 +  \lambda_1 \lambda_2 \lambda _3 \,,
\end{array}
\eeq
and a variable transformation given by
\beq
\label{eqn:Baker_transfo}
\begin{array}{lcl}
 \multicolumn{3}{c}{\mathbf{W} =  K_2 \, z_1 \,, \qquad \mathbf{X} =  K_2 \, z_2 \,, \qquad \mathbf{Y} =  K_2 \, z_3 \,,} \\
 \mathbf{Z} = 2 \tilde{z}_4 - \Big( L_1 z_1^2 z_2 - 2 L_2 z_1^2 z_3 + L_3 z_1 z_2 z_3 - 2 L_4 z_1 z_3^2 + z_2 z_3^2 \Big)  \,.
\end{array}
\eeq
with $K_2$ given in Equation~(\ref{eqn:Ks}). Equation~(\ref{kummer_middle}) is equivalent to the quartic surface $\mathcal{U}$ in $\mathbb{P}^3 = \mathbb{P}( \mathbf{W}, \mathbf{X}, \mathbf{Y}, \mathbf{Z})$ given by the vanishing locus of the \emph{Baker determinant}, i.e.,
\beq
\label{eqn:Baker_det}
\mathcal{U}: \quad  \left| \begin{array}{cccc} 
  0 				& L_1 \mathbf{W}				& -\mathbf{Z}					& \mathbf{Y} \\
  L_1 \mathbf{W}	& 2 L_2 \mathbf{W}	+ 2 \mathbf{Z} 	& L_3 \mathbf{W} - \mathbf{Y}		& \mathbf{X} \\
  -\mathbf{Z}		&  L_3 \mathbf{W} - \mathbf{Y}		& 2 L_4 \mathbf{W}	- 2 \mathbf{X}	& \mathbf{W}\\
  \mathbf{Y}		& \mathbf{X}					& \mathbf{W}					& 0
 \end{array} \right| = 0 \,.
\eeq  
The Baker determinant was first derived in \cite{MR1554977}. 
\par In addition to $K_2$ introduced above, we define homogeneous polynomials $K_l=K_l(z_1,z_2,z_3)$ of degree $4-l$ for $l=0,1$ with
\beq
\label{eqn:Ks}
\begin{split}
K_2 =& \; z_2^2 -4 z_1z_3 \,, \qquad K_1 = -2 z_2 z_3^2 - 2 L_1 z_1^2 z_2 + 4 L_2 z_1^2z_3 - 2 L_3 z_1 z_2 z_3 + 4 L_4 z_1 z_3^2 \,,\\
K_0 =& \;L_1^2 z_1^4 - 2 L_1 L_3 z_1^3 z_3 + (2L_1 - 4 L_2 L_4 + L_3^2) z_1^2 z_3^2 + 4 L_1 L_4 z_1^2 z_2 z_3 \\
&+ 3 (-2 L_1 z_2^2 z_3 + 2 L_2 z_2 z_3^2 - L_3 z_3^3) z_1+ z_3^4 \,.
\end{split}
\eeq
A variable transformation in Equation~(\ref{kummer_middle}), given by
\beq
\label{eqn:transfo_1}
 \tilde{z}_4 = \frac{1}{4} \left( K_2 z_4  + 2 K_1 \right) \,,
\eeq
or, equivalently, the variable transformation  in Equation~(\ref{eqn:Baker_det}), given by
\beq
\label{eqn:transfo_2}
 \mathbf{W} =  K_2 \, z_1 \,, \qquad \mathbf{X} =  K_2 \, z_2 \,, \qquad \mathbf{Y} =  K_2 \, z_3 \,, \qquad \mathbf{Z} = K_2 z_4 + K_1 \,,
\eeq 
yields the quartic  projective surface $\mathcal{K}_\mathbf{A}$ in $\mathbb{P}^3=\mathbb{P}(z_1,z_2,z_3,z_4)$ given by 
\beq
\label{kummer}
\mathcal{K}_\mathbf{A}: \quad K_2(z_1,z_2,z_3) \; z_4^2 \; + \; K_1(z_1,z_2,z_3)\; z_4 \; + \; K_0(z_1,z_2,z_3) = 0 \;.
\eeq
The quartic appeared in Cassels and Flynn \cite{MR1406090}*{Sec.~3} and is called the \emph{Cassels-Flynn quartic}.  We have the following:
\begin{proposition}
\label{prop:Kummers}
Assume that $\lambda_1, \lambda_2, \lambda_3$ are the Rosenhain roots of a smooth genus-two curve $\mathcal{C}$.  The surfaces in Equations~(\ref{kummer_middle}), ~(\ref{eqn:Baker_det}), and~(\ref{kummer}) are birational equivalent over $\mathbb{Q}(\lambda_1, \lambda_2, \lambda_3)$.  The minimal resolution is isomorphic to the Kummer surface $\operatorname{Kum}(\mathbf{A})$ associated with the Jacobian $\mathbf{A}=\operatorname{Jac}{(\mathcal{C})}$.
\end{proposition}
\begin{proof}
It is known that under birational equivalence  the involution minus $-\mathbb{I}$ on $\operatorname{Jac}{(\mathcal{C})}$ restricts to the product involution $\imath_\mathcal{C} \times  \imath_\mathcal{C}$ on $\operatorname{Sym}^2(\mathcal{C})$; see  \cite{ClingherMalmendier2019}. Hence, the Shioda sextic in Equation~(\ref{kummer_middle}) is birational to $\operatorname{Jac}{(\mathcal{C})}/\langle - \mathbb{I} \rangle$ associated with the Jacobian $\mathbf{A}=\operatorname{Jac}{(\mathcal{C})}$ of a genus-two curve~$\mathcal{C}$ in Rosenhain normal form~(\ref{Eq:Rosenhain}).  One substitutes Equations~(\ref{eqn:Baker_transfo}) into Equation~(\ref{eqn:Baker_det}) and checks that the strict transform is given by Equation~(\ref{eqn:Kummer44}). Moreover, the variable transformation in Equation~(\ref{eqn:Baker_det}) can be inverted to obtain a rational inverse map.
\end{proof}
\begin{remark}
The sixteen nodes $\mathsf{p}_{i j}$ of the Kummer surface $\mathcal{K}_\mathbf{A}$ in Equation~(\ref{kummer}) are given in Table~\ref{tab:nodes}. They are the images of the order-two points $P_{i j} \in \mathbf{A}[2]$.
\end{remark}
\begin{table}
\scalebox{0.85}{
\begin{tabular}{c|l}
nodes on $\mathcal{K}_\mathbf{A}$  & $[z_1: z_2 : z_3 : z_4]$\\
\hline
$\mathsf{p}_0$ 		& $[0:0:0:1]$\\
$\mathsf{p}_{16}$	& $[0: 1 : \lambda_1 : \lambda_1^2]$\\
$\mathsf{p}_{26}$	& $[0: 1 : \lambda_2 : \lambda_2^2]$\\
$\mathsf{p}_{36}$	& $[0: 1 : \lambda_3 : \lambda_3^2]$\\
$\mathsf{p}_{46}$	& $[0: 1 : 0 : 0]$\\
$\mathsf{p}_{56}$	& $[0: 1 : 1 : 1]$\\
$\mathsf{p}_{14}$	& $[1: \lambda_1 : 0 : \lambda_2\lambda_3]$\\
$\mathsf{p}_{24}$	& $[1: \lambda_2 : 0 : \lambda_1\lambda_3]$\\
$\mathsf{p}_{34}$	& $[1: \lambda_3 : 0 : \lambda_1\lambda_2]$\\
$\mathsf{p}_{45}$	& $[1: 1 : 0 :  \lambda_1\lambda_2 \lambda_3]$\\
$\mathsf{p}_{15}$	& $[1: \lambda_1+1 : \lambda_1 : \lambda_1(\lambda_2+\lambda_3)]$\\
$\mathsf{p}_{25}$	& $[1: \lambda_2+1 : \lambda_2 : \lambda_2(\lambda_1+\lambda_3)]$\\
$\mathsf{p}_{35}$	& $[1: \lambda_3+1 : \lambda_3 : \lambda_3(\lambda_1+\lambda_2)]$\\
$\mathsf{p}_{13}$	& $[1: \lambda_1+\lambda_3 : \lambda_1\lambda_3 : (\lambda_2+1) \, \lambda_1  \lambda_3]$\\
$\mathsf{p}_{23}$	& $[1: \lambda_2+\lambda_3 : \lambda_2\lambda_3 : (\lambda_1+1) \, \lambda_1  \lambda_2]$\\
$\mathsf{p}_{12}$	& $[1: \lambda_1+\lambda_2 : \lambda_1\lambda_2 : (\lambda_3+1) \, \lambda_1  \lambda_2]$
\end{tabular}}
\smallskip
\caption{Nodes on the generic Jacobian Kummer surface $\mathcal{K}_\mathbf{A}$}
\label{tab:nodes}
\end{table}
\par A holomorphic two-form on $\mathcal{W}$ is given by the two-form $\omega_\mathcal{W} = dz_2 \wedge dz_3/\tilde{z}_4$ in the coordinate chart $z_1=1$ in Equation~(\ref{kummer_middle}). A regular two-form on $\mathcal{C} \times \mathcal{C}$, given as the antisymmetric linear combination of the outer tensor products of the two holomorphic one-forms on each copy of $\mathcal{C}$, i.e.,
\beq
\label{eqn:two-form_g2}
\frac{dX^{(1)}}{Y^{(1)}} \boxtimes  \frac{X^{(2)} dX^{(2)}}{Y^{(2)}} -  \frac{X^{(1)} dX^{(1)}}{Y^{(1)}} \boxtimes \frac{dX^{(2)}}{Y^{(2)}} \,,
\eeq
descends to $\operatorname{Sym}^2(\mathcal{C})/\langle \imath_\mathcal{C} \times  \imath_\mathcal{C} \rangle$ and coincides with $\omega_\mathcal{W}$. If $\epsilon$ is the minimal resolution, we obtain a holomorphic two-form $\omega_{\operatorname{Kum}(\operatorname{Jac}{\mathcal{C}})}$ on $\operatorname{Kum}(\operatorname{Jac}{\mathcal{C}})$ by pullback, i.e.,
\beq
\label{eqn:two-form_JacC}
 \omega_{\operatorname{Kum}(\operatorname{Jac}{\mathcal{C}})} = \epsilon^*\omega_\mathcal{W} \,, \quad \text{with} \
  \epsilon \colon \;  \operatorname{Kum}(\operatorname{Jac}{\mathcal{C}}) =\widehat{ \mathcal{W}} \ \dasharrow \ \mathcal{W} \,.
\eeq
We also introduce the twisted Shioda sextic $\mathcal{W}^{(\varepsilon)}$, given by
\beq
\label{kummer_middle_twist}
  \mathcal{W}^{(\varepsilon)}: \quad  \hat{z}_4^2 = \varepsilon  z_1 z_3 \big(  z_1  - z_2 +  z_3 \big)
   \prod_{i=1}^3 \big( \lambda_i^2 \, z_1  -  \lambda_i \, z_2 +  z_3 \big) \;,
\eeq
such that for general $\varepsilon \in \mathbb{Q}$ the surfaces $\mathcal{W}^{(\varepsilon)}$ and $\mathcal{W}$ are isomorphic only over $\mathbb{Q}(\sqrt{\varepsilon})$.  Further, we introduce another genus-two curve, namely
\beq
\label{eqn:Ctilde}
 \widetilde{\mathcal{C}}: \quad Y^2 =  X Z \big(X-\lambda_0 Z\big) \, \big( X- \lambda_0 \lambda_1 Z\big) \,  \big( X- \lambda_0\lambda_2 Z\big) \,  \big( X- \lambda_0 \lambda_3 Z\big) \,,
\eeq
and its associated (twisted) Shioda sextic given by
\beq
\label{kummer_middle_twist_extended}
  \widetilde{\mathcal{W}}^{(\varepsilon)}: \quad  \hat{y}_4^2 =  \varepsilon y_1 y_3 \big(   \lambda_0^2 y_1  -  \lambda_0 y_2 +  y_3 \big) 
   \prod_{i=1}^3 \big( (\lambda_0 \lambda_i)^2  y_1  -  \lambda_0\lambda_i  y_2 +  y_3 \big) \,.
\eeq
We refer to the minimal resolutions of  $\mathcal{W}^{(\varepsilon)}$ and $\widetilde{\mathcal{W}}^{(\varepsilon)}$ as twisted Kummer surfaces and denote them as  $\operatorname{Kum}(\operatorname{Jac}{\mathcal{C}})^{(\varepsilon)}$ and $\operatorname{Kum}(\operatorname{Jac}{\widetilde{\mathcal{C}}})^{(\varepsilon)}$, respectively. Holomorphic two-forms on the various twisted Shioda sextic surfaces and their minimal resolutions are then obtained analogously to Equation~(\ref{eqn:two-form_JacC}). We have the following:
\begin{proposition}
\label{prop:twistedShiodaSextic}
Assume $\lambda_1, \lambda_2, \lambda_3$ are the Rosenhain roots of a smooth genus-two curve $\mathcal{C}$ and $\lambda_0 \not =0$.  The surfaces $\mathcal{W}^{(\varepsilon)}$ and $\widetilde{\mathcal{W}}^{(\varepsilon/\lambda_0^4)}$ are birational equivalent over $\mathbb{Q}(\lambda_0)$ such that the holomorphic two-form $dy_2 \wedge dy_3/\hat{y}_4$ in the chart $y_1=1$ in Equation~(\ref{kummer_middle_twist_extended}) equals $dz_2 \wedge dz_3/\hat{z}_4$ in the chart $z_1=1$ in Equation~(\ref{kummer_middle}). In particular, there is an isomorphism
\beq
\imath : \quad \operatorname{Kum}(\operatorname{Jac}{\mathcal{C}})^{(\varepsilon)} \ \overset{ \cong}{\longrightarrow} \ \operatorname{Kum}(\operatorname{Jac}{\widetilde{\mathcal{C}}})^{(\varepsilon/ \lambda_0^4)}\,,
\eeq
such that $\omega_{\operatorname{Kum}(\operatorname{Jac}{\mathcal{C}})^{(\varepsilon)}} = \imath^*  \omega_{\operatorname{Kum}(\operatorname{Jac}{\widetilde{\mathcal{C}}})^{(\varepsilon/\lambda_0^4)}} $.
\end{proposition}
\begin{proof}
For the isomorphism $\imath: (z_1, z_2, z_3, \hat{z}_4) \mapsto (y_1, y_2, y_3, \hat{y}_4) = (z_1, \lambda_0 z_2, \lambda_0^2 z_3, \lambda_0^3 \hat{z}_4)$ the pullback
$\imath^*(dy_2 \wedge dy_3/\hat{y}_4)$ equals $dz_2 \wedge dz_3/\hat{z}_4$.
\end{proof}
\subsubsection{Certain elliptic fibrations}
The pencil of lines, given by $t z_1 - z_3=0$ in $\mathbb{P}^2=\mathbb{P}(z_1, z_2, z_3)$, induces an elliptic fibration on the Shioda sextic $\mathcal{W}$. The elliptic fibration, with fibers over $\mathbb{P}^1=\mathbb{P}(u_1,u_2)$ embedded into $\mathbb{P}^2 = \mathbb{P}(X, Y, Z)$, is given by the Weierstrass model
\beq
\label{eqn:W0}
\begin{split}
 Y^2 Z & = X \Big(X + (\lambda_1 - \lambda_3)(\lambda_2-1) u_1 u_2 \big(u_1-\lambda_2 u_2\big)  \big(u_1-\lambda_1 \lambda_3 u_2\big) Z \Big) \\
 & \times  \Big(X + (\lambda_1 - \lambda_2)(\lambda_3-1) u_1 u_2 \big(u_1-\lambda_3 u_2\big)  \big(u_1-\lambda_1 \lambda_2 u_2\big) Z \Big) \,.
\end{split} 
\eeq
The discriminant function of the elliptic fibration is given by
\beqn
\begin{split}
 & (\lambda_1-1)^2  (\lambda_2-1)^2  (\lambda_3-1)^2  (\lambda_1-\lambda_2)^2  (\lambda_1-\lambda_3)^2  (\lambda_2-\lambda_3)^2 u_1^6 u_2^6  (u_1 - \lambda_1 u_2)^2  \\
 & \times  (u_1 - \lambda_2 u_2)^2  (u_1 - \lambda_3 u_2)^2 (u_1 - \lambda_1 \lambda_2 u_2)^2  
 (u_1 - \lambda_1 \lambda_3 u_2)^2  (u_1 - \lambda_2 \lambda_3 u_2)^2 \,.
\end{split} 
\eeqn
The fibration in Equation~(\ref{eqn:W0}) was denoted fibration (1) in \cite{MR3263663}. In \cite{MR3263663} the following lemma was proven:
\begin{lemma}
\label{lem:fib_W0}
Equation~(\ref{eqn:W0}) defines a Jacobian elliptic fibration with the singular fibers $6 I_2  + 2 I_0^*$ and a Mordell-Weil group of sections $\mathbb{Z}/2\mathbb{Z} \oplus \langle 1 \rangle$.
\end{lemma}
In the statement above the symbol $\langle m \rangle$ stands for a rank 1 lattice $\mathbb{Z}x$ satisfying $\langle x, x \rangle = m$ with respect to the height pairing. We also have the following:
\begin{proposition}
\label{prop:equivalence0}
Assume that $\lambda_1, \lambda_2, \lambda_3$ are the Rosenhain roots of a smooth genus-two curve. The surfaces in Equation~(\ref{kummer_middle}) and~(\ref{eqn:W0}) are birational equivalent over $\mathbb{Q}(\lambda_1, \lambda_2, \lambda_3)$. 
\end{proposition}
\begin{proof}
In the charts $u_2=1, Z=1$ in Equation~(\ref{eqn:W0}) and $z_1=1$ in Equation~(\ref{kummer_middle}), a birational equivalence is given
\beq
\label{eqn:transfo_Fib1}
\begin{split} 
z_2  =  \frac{(\lambda_1-1)(\lambda_2-1)(\lambda_3-1)u_1 (u_1-\lambda_1)(u_1-\lambda_2)(u_1-\lambda_3)}{X - \lambda_1(\lambda_2-1)(\lambda_3-1)u_1(u_1-\lambda_2)(u_1-\lambda_3) } + u_1 + 1\,,\\
\tilde{z}_4  = \frac{(\lambda_1-1)(\lambda_2-1)(\lambda_3-1)u_1 (u_1-\lambda_1)(u_1-\lambda_2)(u_1-\lambda_3)Y}{\big(X - \lambda_1(\lambda_2-1)(\lambda_3-1)u_1(u_1-\lambda_2)(u_1-\lambda_3) \big)^2} \,, \qquad z_3=u_1\,.
\end{split}
\eeq
\end{proof}
The holomorphic two-forms are related as follows:
\begin{corollary}
\label{lem:2forms0}
The birational equivalence of Proposition~\ref{prop:equivalence0} identifies the holomorphic two-form $dz_2 \wedge dz_3/\tilde{z}_4$ in the chart $z_1=1$ in Equation~(\ref{kummer_middle}) with $du_1 \wedge dX/Y$ in the chart $Z=u_2=1$ in Equation~(\ref{eqn:W0}).
\end{corollary}
\begin{proof}
The proof follows by an explicit computation using the transformation provided in the proof of Proposition~\ref{prop:equivalence0}.
\end{proof}
\par On the Shioda sextic $\mathcal{W}$, another elliptic fibration, with fibers over $\mathbb{P}^1=\mathbb{P}(w_1,w_2)$ embedded into $\mathbb{P}^2 = \mathbb{P}(x, y, z)$, is given by the Weierstrass model
\beq
\label{eqn:W}
\begin{split}
 y^2 z = x^3 + w_1 \Big(w_1 + (\lambda_1+\lambda_2 \lambda_3) w_2\Big)  \Big(w_1 + (\lambda_2+\lambda_1 \lambda_3) w_2\Big)  \Big(w_1 + (\lambda_3+\lambda_1 \lambda_2) w_2\Big)  x^2 z \\
+ \lambda_1 \lambda_2 \lambda_3  w_2^2
\Big(w_1 + (\lambda_1+\lambda_2 \lambda_3) w_2\Big)^2  
\Big(w_1 + (\lambda_2+\lambda_1 \lambda_3) w_2\Big)^2  
\Big(w_1 + (\lambda_3+\lambda_1 \lambda_2) w_2\Big)^2 x z^2.
\end{split} 
\eeq
The discriminant function of the elliptic fibration is given by
\beq
\label{eqn:discriminant_function}
\begin{split}
 &   \lambda_1^2 \lambda_2^2 \lambda_3^2 w_2^4 \Big(w_1^2 - 4  \lambda_1 \lambda_2 \lambda_3 w_2^2 \Big)  \Big(w_1 + (\lambda_1+\lambda_2 \lambda_3) w_2\Big)^6 \\
&\times   \Big(w_1 + (\lambda_2+\lambda_1 \lambda_3) w_2\Big)^6  
\Big(w_1 + (\lambda_3+\lambda_1 \lambda_2) w_2\Big)^6 \,.
\end{split}
\eeq
We have the immediate:
\begin{lemma}
\label{lem:fib_W}
Equation~(\ref{eqn:W}) defines a Jacobian elliptic fibration with the singular fibers $I_4 + 2 I_1 + 3 I_0^*$ and a Mordell-Weil group of sections $\mathbb{Z}/2\mathbb{Z}$.
\end{lemma}
The fibration in Equation~(\ref{eqn:W}) was denoted fibration (4) in \cite{MR3263663}. We have the following:
\begin{proposition}
\label{prop:equivalence}
Assume that $\lambda_1, \lambda_2, \lambda_3$ are the Rosenhain roots of a smooth genus-two curve.  The surfaces in Equation~(\ref{kummer_middle}) and~(\ref{eqn:W}) are birational equivalent over $\mathbb{Q}(\lambda_1, \lambda_2, \lambda_3)$. 
\end{proposition}
\begin{proof}
In the charts $w_2=1, z=1$ in Equation~(\ref{eqn:W}) and $z_1=1$ in Equation~(\ref{kummer_middle}), a birational equivalence is given by
\beqn
\begin{split}
 \scalemath{0.8}{w_1 = \frac{\big(\lambda_1(\lambda_2-\lambda_3)-\lambda_2\lambda_3(\lambda_1+1)\big)z_2z_3+\lambda_1\lambda_2\lambda_3(1+\lambda_1+\lambda_2-\lambda_3)z_2 + \lambda_3(1+\lambda_1+\lambda_2-\lambda_3)z_3^2}{\lambda_3 z_2z_3 - \lambda_1 \lambda_2 z_2 - z_3^2 + \big((\lambda_2- \lambda_3)(\lambda_1+1)+ \lambda_1 -\lambda_2\lambda_3\big)  z_3 + \lambda_1 \lambda_2\lambda_3}\,\phantom{,}}\\
  \scalemath{0.8}{- \frac{(\lambda_1^2\lambda_2+\lambda_1(1+\lambda_2)(\lambda_2-\lambda_3^2)-\lambda_2\lambda_3^2)z_3 +\lambda_1\lambda_2\big(\lambda_2\lambda_3 + \lambda_1(\lambda_2\lambda_3-\lambda_2+\lambda_3\big)}{\lambda_3 z_2z_3 - \lambda_1 \lambda_2 z_2 - z_3^2 + \big((\lambda_2- \lambda_3)(\lambda_1+1)+ \lambda_1 -\lambda_2\lambda_3\big)  z_3 + \lambda_1 \lambda_2\lambda_3}\,,}\\
 \scalemath{0.8}{x = \frac{(\lambda_1-\lambda_3)^2(\lambda_2-\lambda_3)^2(1-\lambda_3)^2 z_3 (z_2-z_3-1)(z_3-\lambda_1) (z_3-\lambda_2) (z_3-\lambda_1\lambda_2) (\lambda_1z_2-z_3 -\lambda_1^2)(\lambda_2z_2-z_3 -\lambda_2^2)}{\big(\lambda_3 z_2z_3 - \lambda_1 \lambda_2 z_2 - z_3^2 + \big((\lambda_2- \lambda_3)(\lambda_1+1)+ \lambda_1 -\lambda_2\lambda_3\big)  z_3 + \lambda_1 \lambda_2\lambda_3\big)^3}\,,}\\
  \scalemath{0.8}{y = \frac{(\lambda_3-\lambda_1)^3(\lambda_2-\lambda_3)^3(\lambda_3-1)^3  (z_2-z_3-1)(z_3-\lambda_1)^2 (z_3-\lambda_2)^2 (z_3-\lambda_1\lambda_2)^2 (\lambda_1z_2-z_3 -\lambda_1^2)(\lambda_2z_2-z_3 -\lambda_2^2) \tilde{z}_4}{\big(\lambda_3 z_2z_3 - \lambda_1 \lambda_2 z_2 - z_3^2 + \big((\lambda_2- \lambda_3)(\lambda_1+1)+ \lambda_1 -\lambda_2\lambda_3\big)  z_3 + \lambda_1 \lambda_2\lambda_3\big)^5}\,.}
\end{split}
\eeqn
\end{proof}
The holomorphic two-forms are related as follows:
\begin{corollary}
\label{lem:2forms}
The birational equivalence of Proposition~\ref{prop:equivalence} identifies the holomorphic two-form $dw_1 \wedge dx/y$ in the chart $z=w_2=1$ in Equation~(\ref{eqn:W}) with $dz_2 \wedge dz_3/\tilde{z}_4$ in the chart $z_1=1$ in Equation~(\ref{kummer_middle}).
\end{corollary}
\begin{proof}
The proof follows by an explicit computation using the transformation provided in the proof of Proposition~\ref{prop:equivalence}.
\end{proof}
\subsubsection{Specialization to Picard rank 18}
We consider the special case of the Kummer surface associated with the Jacobian of the genus-two curve $\mathcal{C}_0$ in Equation~(\ref{Eq:Rosenhain_special}), i.e., a smooth genus-two curve admitting an elliptic involution. We will denote the Shioda sextic in Equation~(\ref{kummer_middle}) for $\lambda_1 = \lambda_2 \lambda_3$ by $\mathcal{W}_0$, and the corresponding vanishing locus for the Baker determinant by $\mathcal{U}_0$. Similarly, we use the symbols $\widetilde{\mathcal{C}}_0$, $\mathcal{W}^{(\varepsilon)}_0$, $\widetilde{\mathcal{W}}^{(\varepsilon)}_0$ to refer to the corresponding specialization for $\lambda_1 = \lambda_2 \lambda_3$.
\begin{remark}
The induced involution $\jmath$ on the Shioda sextic $\mathcal{W}_0$ in Equation~(\ref{kummer_middle}) for $\lambda_1 = \lambda_2 \lambda_3$ is given by
\beq 
\jmath: \quad  \Big( z_1, \, z_2, \, z_3, \, \tilde{z}_4 \Big) \ \ \mapsto \ \  \Big( z_3, \, \lambda_2 \lambda_3 z_2, \, (\lambda_2 \lambda_3)^2  z_3, \, (\lambda_2 \lambda_3)^3 \tilde{z}_4 \Big) \,.
\eeq
\end{remark}
In terms of the elliptic fibration already discussed above, the restriction to the curve $\mathcal{C}_0$ in Equation~(\ref{Eq:Rosenhain_special}) is characterized as follows:
\begin{lemma}
\label{lem:fibs_special}
Assume that $\lambda_1, \lambda_2, \lambda_3$ are the Rosenhain roots of a smooth genus-two curve $\mathcal{C}$.  We have the following:
\begin{enumerate}
\item the singular fibers in the fibration of Lemma~\ref{lem:fib_W0} are $I_4 +  4 I_2 + 2I_0^*$ if and only if $\lambda_1= \lambda_2 \lambda_3$, $\lambda_2= \lambda_1 \lambda_3$, or $\lambda_3= \lambda_1 \lambda_2$,
\item the singular fibers in the fibration of Lemma~\ref{lem:fib_W} are $I_4 +  I_1 + I_1^* + 2I_0^*$ if and only if $\lambda_1= \lambda_2 \lambda_3$, $\lambda_2= \lambda_1 \lambda_3$, or $\lambda_3= \lambda_1 \lambda_2$.
\end{enumerate}
\end{lemma}
\begin{proof}
We prove the second statement, the first one is analogous. We apply the resultant to pairs of factors in the discriminant function~(\ref{eqn:discriminant_function}) to obtain criteria for coalescing fibers. The resultant (with respect to $[w_1:w_2]$) of the factors in the reduced discriminant corresponding to fibers of Kodaira-type $I_0^*$, i.e.,
\beq
 \Big(w_1 + (\lambda_1+\lambda_2 \lambda_3) w_2\Big)
\Big(w_1 + (\lambda_2+\lambda_1 \lambda_3) w_2\Big) 
\Big(w_1 + (\lambda_3+\lambda_1 \lambda_2) w_2\Big) \,,
\eeq
and the factors in the discriminant corresponding to the fibers of Kodaira-type $I_1$, i.e., $w_1^2 - 4  \lambda_1 \lambda_2 \lambda_3 w_2^2$, yields 
\beq
\label{eqn:resultant}
 \Big(\lambda_1 - \lambda_2 \lambda_3\Big)^2   \Big(\lambda_2 - \lambda_1 \lambda_3\Big)^2  \Big(\lambda_3 - \lambda_1 \lambda_2\Big)^2  \;. \qedhere
\eeq 
\end{proof}
The elliptic fibrations in Lemma~\ref{lem:fibs_special} and the quadratic twist are related as follows:
\begin{corollary}
\label{cor:lem:2forms00_fib1}
Assume that $\lambda_2, \lambda_3 \in \mathbb{P}^1 \backslash \{ 0, 1, \infty\}$ satisfy $\lambda_2 \not= \lambda_3^{\pm1}$.  The Weierstrass model, given by
\beq
\label{eqn:W0twist}
\begin{split}
 Y^2 Z & = \varepsilon X \Big(X + \lambda_2(\lambda_3-1)^2 u_1 u_2 \big(u_1-\lambda_3 u_2\big)  \big(u_1-\lambda_2^2 \lambda_3 u_2\big) Z \Big)\\
 & \times  \,  \Big(X + \lambda_3(\lambda_2-1)^2 u_1 u_2 \big(u_1-\lambda_2 u_2\big)  \big(u_1-\lambda_2 \lambda_3^2 u_2\big) Z \Big) \,,
\end{split} 
\eeq
defines a Jacobian elliptic fibration with the singular fibers $I_4 +  4 I_2  + 2 I_0^*$ and a Mordell-Weil group of sections $\mathbb{Z}/2\mathbb{Z} + \langle 1 \rangle$. The elliptic surface is birational equivalent to $\mathcal{W}^{(\varepsilon)}_0\!$ over $\mathbb{Q}(\lambda_2\lambda_3, \lambda_2 + \lambda_3)$  and Corollary~\ref{lem:2forms0} holds. 
\end{corollary}
\begin{corollary}
\label{cor:lem:2forms00_fib2}
Assume that $\lambda_2, \lambda_3 \in \mathbb{P}^1 \backslash \{ 0, 1, \infty\}$ satisfy $\lambda_2 \not= \lambda_3^{\pm1}$.  The Weierstrass model, given by
\beq
\label{eqn:W_twist}
\begin{split}
 y^2 z = x^3 + \varepsilon w_1 \Big(w_1 + 2 \lambda_2 \lambda_3 w_2\Big)  \Big(w_1 + \lambda_2 (1 +\lambda_3^2) w_2\Big)  \Big(w_1 + \lambda_3 (1+ \lambda_2^2) w_2\Big)  x^2 z \\
 + \varepsilon^2  (\lambda_2 \lambda_3)^2  w_2^2
\Big(w_1 + 2 \lambda_2 \lambda_3 w_2\Big)^2  
\Big(w_1 + \lambda_2 (1 + \lambda_3^2) w_2\Big)^2  
\Big(w_1 + \lambda_3 (1+ \lambda_2^2) w_2\Big)^2 x z^2 \,,
\end{split} 
\eeq
defines a Jacobian elliptic fibration with the singular fibers $I_4 + I_1 + I_1^* + 2 I_0^*$ and a Mordell-Weil group of sections $\mathbb{Z}/2\mathbb{Z}$. The elliptic surface is birational equivalent to $\mathcal{W}^{(\varepsilon)}_0\!$ over  $\mathbb{Q}(\lambda_2,\lambda_3)$ and Corollary~\ref{lem:2forms} holds. 
\end{corollary}
\subsection{Quartic surfaces}
Let us briefly recall the construction of the G\"opel-Hudson quartic for a Kummer surface associated with a principally polarized abelian surface $(\mathbf{A}, \mathcal{L})$ based on results in \cite{MR2062673}. Considering the rational map $\varphi_{\mathcal{L}^2}: \mathbf{A} \rightarrow  \mathbb{P}^3$, its image $\varphi_{\mathcal{L}^2}(\mathbf{A})$  is a quartic surface in $\mathbb{P}^3$ which, using the projective coordinates $[\mathbf{w}:\mathbf{x}:\mathbf{y}:\mathbf{z}]$, can be written as
\begin{small}
\begin{gather}
\label{Eq:QuarticSurfaces12}
    0 = \xi_0 \, (\mathbf{w}^4+\mathbf{x}^4+\mathbf{y}^4+\mathbf{z}^4)  + \xi_4 \, \mathbf{w} \mathbf{x} \mathbf{y} \mathbf{z} \qquad \\
\nonumber
    +\xi_1 \, \big(\mathbf{w}^2 \mathbf{z}^2+\mathbf{x}^2 \mathbf{y}^2\big)  +\xi_2 \, \big(\mathbf{w}^2 \mathbf{x}^2+\mathbf{y}^2 \mathbf{z}^2\big)  +\xi_3 \, \big(\mathbf{w}^2 \mathbf{y}^2+\mathbf{x}^2 \mathbf{z}^2\big) \,,
\end{gather}
\end{small}%
for some parameter set $[\xi_0:\xi_1:\xi_2:\xi_3:\xi_4] \in \mathbb{P}^4$. A general member of the family~(\ref{Eq:QuarticSurfaces12}) is smooth. As soon as the surface is singular at a general point, it must have sixteen singular nodal points because of its symmetry. The discriminant turns out to be a homogeneous polynomial of degree eighteen in the parameters $[\xi_0:\xi_1:\xi_2:\xi_3:\xi_4] \in \mathbb{P}^4$ and was determined in \cite{MR2062673}*{Sec.~7.7 (3)}.   Thus, the Kummer surfaces form an open set among these surfaces with parameters $[\xi_0:\xi_1:\xi_2:\xi_3:\xi_4] \in \mathbb{P}^4$, namely the ones that make the irreducible factor of degree three in the discriminant vanish, i.e.,
\beq
\label{eqn:threefold}
 \xi_0 \, \big( 16 \xi_0^2 - 4 \xi_1^2-4 \xi_2^2 - 4 \xi_3^3+ \xi_4^2\big) + 4 \, \xi_1 \xi_2 \xi_3 =0 \,.
\eeq
Setting $\xi_0=1$ and using the affine moduli $\xi_1=-A$, $\xi_2=-B$, $\xi_3=-C$, $\xi_4=2 D$, we obtain a normal form for a nodal quartic surface, known as \emph{G\"opel-Hudson quartic} (GH-quartic). The G\"opel-Hudson quartic is the projective surface in $\mathbb{P}^3=\mathbb{P}(\mathbf{w},\mathbf{x},\mathbf{y},\mathbf{z})$ given by
\begin{small}
\begin{gather}
\label{Goepel-Quartic}
  0 = \mathbf{w}^4+\mathbf{x}^4+\mathbf{y}^4+\mathbf{z}^4  + 2  D  \mathbf{w} \mathbf{x} \mathbf{y} \mathbf{z} \\
\nonumber
    - A  \big(\mathbf{w}^2 \mathbf{z}^2+\mathbf{x}^2 \mathbf{y}^2\big)  - B \big(\mathbf{w}^2 \mathbf{x}^2+\mathbf{y}^2 \mathbf{z}^2\big)  - C  \big(\mathbf{w}^2 \mathbf{y}^2+\mathbf{x}^2 \mathbf{z}^2\big)  \;, 
\end{gather}
\end{small}%
where $A, B, C, D \in \mathbb{C}$ satisfy
\beq
\label{paramGH}
 D^2 = A^2 + B^2 + C^2 + ABC - 4\;.
\eeq
By construction, we have the following:
\begin{lemma}
\label{lem:GH}
The minimal resolution of the quartic surface in Equation~(\ref{Goepel-Quartic}), for generic parameters $(A,B,C,D)$ satisfying~(\ref{paramGH}), is isomorphic to the Kummer surface $\operatorname{Kum}(\mathbf{A})$ associated with a principally polarized abelian surface $\mathbf{A}$.
\end{lemma}
\par The abelian surface in Lemma~\ref{lem:GH} for generic parameters $(A,B,C,D)$ does \emph{not} admit any additional automorphism and is in fact isomorphic to $\operatorname{Jac}{(\mathcal{C})}$ for a general genus-two curve $\mathcal{C}$. In \cite{ClingherMalmendier2019}*{Thm.~4.46} two of the authors determined explicitly the connection between the parameters $(A,B,C,D)$ and the moduli of the abelian surface $\mathbf{A} =\operatorname{Jac}{(\mathcal{C})}$ for a smooth genus-two curve $\mathcal{C}$. It was also proved in  \cite{ClingherMalmendier2019} that $(A,B,C,D)$ are modular functions relative to $\Gamma_2(2)$. 
We have the following:
\begin{proposition}
\label{thm:extra_auto}
If the minimal resolution of the Hudson quartic~(\ref{Goepel-Quartic}), for parameters $(A,B,C,D)$ satisfying Equation~(\ref{paramGH}),  is isomorphic to the Kummer surface $\operatorname{Kum}(\mathbf{A})$ associated with a principally polarized abelian surface $\mathbf{A}=\operatorname{Jac}{(\mathcal{C}_0)}$ for a smooth genus-two curve $\mathcal{C}_0$ admitting an elliptic involution, then one of the following additional relations holds:
\beq
\label{eqn:extra_relation}
\begin{array}{lcl}
 \text{case I}: &&  D = 0 \,,\\
 \text{case II}: &&   B = \pm C \,, \\
 \text{case III}: &&    A = \pm B \,, \quad \text{or} \quad  A = \pm C \, , \\
 \text{case IV}: && \left\lbrace \begin{array}{l}
 A+B+C \pm D + 6=0 \,, \quad \text{or} \quad  -A-B+C\pm D+6=0 \,, \quad \text{or}  \\
 A-B-C\pm D+6=0 \,,\quad \text{or} \quad  -A+B-C \pm D + 6 =0 \,. \end{array}\right.
\end{array} 
\eeq
\end{proposition}
\begin{remark}
Table~\ref{fig:Pringsheim_extended} shows the one-to-one correspondence between the 15 components in Equation~(\ref{eqn:extra_relation}) and the components in $\mathbb{H}_2/\Gamma_2(2)$ covering $\mathcal{H}_4$ in the Pringsheim decomposition in Proposition~\ref{prop:pringsheim}, using the same marking as before. 
\end{remark}
\begin{remark}
Equation~(\ref{eqn:threefold}) defines the Segre cubic threefold, i.e., the projective dual of the Igusa quartic in $\mathbb{P}^4$ that is the Satake compactification of $\mathcal{A}_2(2)$. Under this projective duality the decomposable abelian surfaces in $\mathcal{A}_2(2)$, that is $\mathcal{H}_1(2)$, are contracted to 10 nodes of the Segre cubic. The Segre cubic has the simpler equation $\sum x_i = \sum x_i^3=0$ in $\mathbb{P}^5$ such that  the permutation of the coordinates corresponds to the action of $S_6$ on $\mathcal{A}_2(2)$ and induces an action on the parameters $A, B, C, D$ \cite{MR1007155}*{p.~182-183 and p.~156-159} and \cite{MR669299}*{p.~317–350 and p.~348 for $\mathcal{H}_\delta(2)$ for $\delta= 1,4$}. One has to omit the 15 hyperplanes in the Segre cubic whose divisors correspond to the boundary components of the Satake compactification $\overline{\mathcal{A}_2(2)}$ \cite{MR1007155}*{Prop.~6}. The 10 nodes of the Segre cubic correspond to the double quadrics given by $(\mathbf{w}^2 \pm \mathbf{x}^2 \pm \mathbf{y}^2 \pm \mathbf{z}^2)^2$ with an even number of minus signs, and $(ab \pm cd)^2$ with $\{a , b, c, d\} = \{\mathbf{w}, \mathbf{x}, \mathbf{y}, \mathbf{z} \}$. They are attained in Equation~(\ref{Goepel-Quartic}) for images of $\mathcal{E}_1 \times \mathcal{E}_2$. Once all these boundary components are excluded, the statement in Proposition~\ref{thm:extra_auto} becomes `if and only if'. 
\end{remark}
\begin{proof}
A linear transformation of variables in Equation~(\ref{eqn:Baker_det}) over a suitable finite extension field yields a G\"opel quartic in Equation~(\ref{Goepel-Quartic}). In \cite{ClingherMalmendier2019}*{Thm.~4.46} a corresponding solution for the parameters $(A,B,C,D)$ in terms of the Rosenhain roots of a genus-two curve $\mathcal{C}_0$ in Equation~(\ref{Eq:Rosenhain}) was determined explicitly. It is given by
\begin{small}
\begin{gather}
\nonumber
 A =  2 \, \frac{\lambda_1+1}{\lambda_1-1}, \quad
 B =  2 \, \frac{\lambda_1\lambda_2+\lambda_1\lambda_3-2\lambda_2\lambda_3-2\lambda_1+\lambda_2+\lambda_3}{(\lambda_2-\lambda_3)(\lambda_1-1)},  \quad
 C  = 2 \, \frac{\lambda_3+\lambda_2}{\lambda_3-\lambda_2}, \\ 
 \label{KummerParameter4}
 D = 4 \, \frac{\lambda_1-\lambda_2 \lambda_3}{(\lambda_2 - \lambda_3) (\lambda_1-1)} \;,
\end{gather}
\end{small}%
such that Equation~(\ref{paramGH}) is satisfied. All other possible solutions are obtained from this one by the 15 transpositions acting on the roots $(\lambda_1, \lambda_2, \lambda_3, 0, 1,\infty)$. This is precisely the relation between the Pringsheim components for $\mathcal{H}_4$ given in Proposition~\ref{prop:pringsheim}. We can then rewrite any relation between $\lambda_1, \lambda_2,\lambda_3$ as an additional relation for $A, B, C, D$. In this way, we obtain a one-to-one correspondence between the 15 Pringsheim components in Proposition~\ref{prop:pringsheim} and the 15 additional relations for the G\"opel-Hudson quartic in Equation~(\ref{eqn:extra_relation}). 
\begin{table}[ht]
\scalemath{0.8}{
\begin{tabular}{|l|l|r|c|}
\hline
\multicolumn{4}{|c|}{{\color{black}I}} \\
\hdashline
$ 2 \tau_{12} + \tau_{11} \tau_{22} - \tau_{12}^2 =0$ & $\lambda_1=\lambda_2 \lambda_3$  & $D=0$ & $\{ \mathsf{p}_{15}, \mathsf{p}_{23}, \mathsf{p}_{46} \}$ \\
\hline
\multicolumn{4}{|c|}{{\color{magenta}II}} \\
\hdashline
$\tau_{11} + 2 \tau_{12}  =0$ 							& $\lambda_1- \lambda_2 = \lambda_3(1-\lambda_2)$					& $B+C=0$
& $\{ \mathsf{p}_{14}, \mathsf{p}_{23}, \mathsf{p}_{56} \}$ \\
$\tau_{11} + 2 \tau_{12} - (\tau_{11} \tau_{22} - \tau_{12}^2)=0$	& $\lambda_1(1-\lambda_3)=\lambda_2(\lambda_1-\lambda_3)$			& $B-C=0$
& $\{ \mathsf{p}_{16}, \mathsf{p}_{23}, \mathsf{p}_{45} \}$ \\
\hline
\multicolumn{4}{|c|}{{\color{red}III}} \\
\hdashline
$2 \tau_{12} - \tau_{22} =0 $ 							& $\lambda_3 =\lambda_1 \lambda_2$ 								& $A-C=0$
& $\{ \mathsf{p}_{12}, \mathsf{p}_{35}, \mathsf{p}_{46} \}$ \\
$2 \tau_{12} - \tau_{22} + (\tau_{11} \tau_{22} - \tau_{12}^2) =0$& $\lambda_2 =\lambda_1 \lambda_3$								& $A+C=0$
& $\{ \mathsf{p}_{13}, \mathsf{p}_{25}, \mathsf{p}_{46} \}$ \\
$\tau_{11} - \tau_{22}=0$								& $\lambda_1 - \lambda_2 = \lambda_2 (\lambda_1 - \lambda_3)$			& $A+B=0$
& $\{ \mathsf{p}_{15}, \mathsf{p}_{26}, \mathsf{p}_{34} \}$ \\
$\tau_{11} -\tau_{22} + (\tau_{11} \tau_{22} - \tau_{12}^2) =0$	& $\lambda_1 - \lambda_3 = \lambda_3 (\lambda_1 - \lambda_2)$			& $A-B=0$
& $\{ \mathsf{p}_{15}, \mathsf{p}_{24}, \mathsf{p}_{36} \}$  \\
\hline
\multicolumn{4}{|c|}{{\color{blue}IV}} \\
\hdashline
$ 2 \tau_{12}=1$									& $ \lambda_2-\lambda_3 = - \lambda_1(1-\lambda_2)$					& $ -A+B-C-D+6=0$
& $\{ \mathsf{p}_{12}, \mathsf{p}_{34}, \mathsf{p}_{56} \}$ \\
$ 2 \tau_{12} + \tau_{22} =1$							& $ \lambda_3 (1-\lambda_2) = -\lambda_1 (\lambda_2 -\lambda_3)$		& $ A+B+C+D+6=0$
& $\{ \mathsf{p}_{12}, \mathsf{p}_{36}, \mathsf{p}_{45} \}$ \\
$\tau_{11} + 2 \tau_{12}=1$							& $\lambda_2-\lambda_3 = \lambda_1 (1-\lambda_3)$					& $-A-B+C+D+6=0$
& $\{ \mathsf{p}_{13}, \mathsf{p}_{24}, \mathsf{p}_{56} \}$ \\
$\tau_{11} + 2 \tau_{12} + \tau_{22} - (\tau_{11} \tau_{22}-\tau_{12}^2)=1$ & $\lambda_1(\lambda_2-\lambda_3)=\lambda_2(1-\lambda_3)$	& $ A-B-C-D+6=0$
& $\{ \mathsf{p}_{13}, \mathsf{p}_{26}, \mathsf{p}_{45} \}$ \\
$4 \tau_{12} + 3 (\tau_{11} \tau_{22} - \tau_{12}^2)=1	$		& $\lambda_2 -\lambda_3=\lambda_2(\lambda_1-\lambda_3)$				& $-A+B-C+D+6=0$
& $\{ \mathsf{p}_{14}, \mathsf{p}_{35}, \mathsf{p}_{26} \}$ \\
$4\tau_{12} + \tau_{22} + 3 (\tau_{11} \tau_{22} - \tau_{12}^2)=1$& $\lambda_1 -\lambda_3=\lambda_1(\lambda_2-\lambda_3)$			& $ A+B+C-D+6=0$
& $\{ \mathsf{p}_{16}, \mathsf{p}_{24}, \mathsf{p}_{35} \}$ \\ 
$\tau_{11} + 4\tau_{12} + 3 (\tau_{11} \tau_{22} - \tau_{12}^2)=1$& $\lambda_2-\lambda_3=-\lambda_3(\lambda_1-\lambda_2)$			& $-A-B+C - D+6=0$
& $\{ \mathsf{p}_{14}, \mathsf{p}_{25}, \mathsf{p}_{36} \}$ \\
$\tau_{11} + 4\tau_{12} +  \tau_{22} + 2 (\tau_{11} \tau_{22} - \tau_{12}^2)=1$& $\lambda_1-\lambda_2=-\lambda_1(\lambda_2-\lambda_3)$	& $ A-B-C+D+6=0$
& $\{ \mathsf{p}_{16}, \mathsf{p}_{25}, \mathsf{p}_{34} \}$ \\
\hline
\end{tabular}}
\smallskip
\caption{Components of $\mathcal{H}_4$ with matching constraints for the GH-quartic}
\label{fig:Pringsheim_extended}
\end{table}
\end{proof}
Proposition~\ref{thm:extra_auto} provides a characterization of the 15 components in $\mathbb{H}_2/\Gamma_2(2)$ covering $\mathcal{H}_4$. Because of Remark~\ref{rem:config} each of the components also corresponds to a configuration in the Kummer plane where three nodes on $\mathcal{K}_\mathbf{A}$ in Table~\ref{tab:nodes} are collinear. On the other hand, for each genus-two curve $\mathcal{C}_0$ with an elliptic involution there is precisely one G\"opel group $G \leqslant \mathbf{A}[2]$ with $\mathbf{A}=\operatorname{Jac}{(\mathcal{C}_0)}$ such that $\mathbf{A}/G \cong \mathcal{E}_1 \times \mathcal{E}_2$; see Corollary~\ref{cor:special_G_group}. We then have the following:
\begin{corollary}
For a smooth genus-two curve $\mathcal{C}_0$ admitting an elliptic involution and $\mathbf{A}=\operatorname{Jac}{(\mathcal{C}_0)}$, the group $G= \{ P_0, P_{i j}, P_{k l}, P_{m n}\}  \leqslant \mathbf{A}[2]$  is the unique G\"opel group such that $\mathbf{A}/G \cong \mathcal{E}_1 \times \mathcal{E}_2$ if and only if for the corresponding Kummer configuration $(\mathbb{P}^2; \gerade_1, \dots, \gerade_6)$  on $\mathcal{K}_\mathbf{A}$ the three nodes $\{\mathsf{p}_{ij},  \mathsf{p}_{m n}, \mathsf{p}_{k l}\}$ are collinear. 
\end{corollary}
\begin{remark}
Table~\ref{fig:Pringsheim_extended} shows the 15 components in $\mathbb{H}_2/\Gamma_2(2)$ covering $\mathcal{H}_4$ and the corresponding collinear points in the Kummer plane, using the marking of Weierstrass points $(P_1, \dots, P_6) = (\lambda_1=\lambda_2\lambda_3, \lambda_2, \lambda_3, 0, 1,\infty)$ for the curve $\mathcal{C}_0$. 
\end{remark}
\begin{proof}
The one-to-one correspondence between the possible configurations of collinear points in the Kummer plane and the 15 components in $\mathbb{H}_2/\Gamma_2(2)$ covering $\mathcal{H}_4$ is computed explicitly using their respective characterizations in terms of Rosenhain roots in Table~\ref{tab:nodes} and Table~\ref{fig:Pringsheim}, respectively.
\end{proof}
\begin{remark}
For $A^2=B^2=C^2=1$ and $D=0$ in Equation~(\ref{Goepel-Quartic}) one finds the special solutions in Table~\ref{fig:solutions_20}. The minimal resolution of Equation~(\ref{Goepel-Quartic}) is then isomorphic to the Kummer surface $\operatorname{Kum}(\operatorname{Jac} \mathcal{C}_0)$ of Picard rank $\rho=20$ where $\mathcal{C}_0$ is given by Equation~(\ref{Eq:Rosenhain_special}) with $\lambda_2, \lambda_3$ as specified in the table. In this case, the two elliptic-curve quotients are isomorphic to the elliptic curve $\mathcal{E}$, given by
\beqn
 \mathcal{E}: \quad y^2 z = x^3 + z^3 \,,
\eeqn
admitting a $\mathbb{Z}_3$-symmetry $x \mapsto \omega_3 x$ for $\omega_3^3=1$.
\end{remark}
\begin{table}[ht]
\scalemath{1}{
 \begin{tabular}{r|rr|rrrr|cc|c|c}
 $x$ & $\lambda_2$ & $\lambda_3$ & $A$ & $B$ & $C$ & $D$ & $\Lambda_1 \Lambda_2$ & $\Lambda_1 + \Lambda_2$ & 
 $j(\mathcal{E}_1) = j(\mathcal{E}_2)$ & $\rho$\\
 \hline
   $ \pm1$ 	&$-1 $		& $3 	$			& $1$	& $1$	& $1$	& $0$	& $1$	& $1$ & $0$ & $20$\\
   $ \pm1$ 	&$3 	$		& $-1 $			& $1$	& $-1$ 	& $-1$  	& $0$ 	& $1$	& $1$ & $0$ & $20$\\
   $ \pm1$ 	&$-1 $		& $\frac{1}{3}$ 		& $-1$ 	& $1$	& $-1$ 	& $0$ 	& $1$	& $1$ & $0$ & $20$\\
   $\pm 1$ 	&$\frac{1}{3}$ 	&$-1 $			& $-1$ 	& $-1$ 	& $1$	& $0$	& $1$ 	& $1$ & $0$ & $20$
\end{tabular}}
\smallskip
\caption{Solutions for $A^2=B^2=C^2=1$ and $D=0$}
\label{fig:solutions_20} 
\end{table}
\subsubsection{Related normal form}
\label{ssec:quartic_surfaces}
The Shioda sextic $\mathcal{W}_0$ in Equation~(\ref{kummer_middle_twist}) has a presentation as a simple quartic hypersurface, isomorphic to $\mathcal{W}_0$ over $\mathbb{Q}(\lambda_2 + \lambda_3, \lambda_2 \lambda_3)$.  Notice that this is \emph{not} the case for the G\"opel-Hudson quartic in Equation~(\ref{Goepel-Quartic}) -- there, a finite field extension is needed to construct such an isomorphism. 
\par Each elliptic curve $\mathcal{E}_l$ for $l=1, 2$, as given in Legendre form in Equation~(\ref{eqn:EC}), can also be represented as a complete quadric intersection in $\mathbb{P}^3$. Let  $\mathcal{I}_1$ be the complete intersection of the two quadric surfaces in $\mathbb{P}^3 = \mathbb{P}(\mathbf{X}_{00}, \mathbf{X}_{01}, \mathbf{X}_{10}, \mathbf{X}_{11})$, given by
\beq
\label{eqn:intersections_n_1} \mathcal{I}_1 : \quad 
 \left\lbrace \begin{array}{lcl} 
  \mathbf{X}_{01}^2 & = & \mathbf{X}_{10}^2 + \mathbf{X}_{11}^2  \,,\\[4pt]
  \mathbf{X}_{00}^2 & = & \mathbf{X}_{10}^2    +  \big(1 - \Lambda_1\big ) \mathbf{X}_{11}^2  \,, \end{array} \right.
\eeq
and $\mathcal{I}_2$ be the complete intersection in $\mathbb{P}^3 = \mathbb{P}(\mathbf{Y}_{00}, \mathbf{Y}_{01}, \mathbf{Y}_{10}, \mathbf{Y}_{11})$, given by
\beq
\label{eqn:intersections_n_2} \mathcal{I}_2 : \quad 
 \left\lbrace \begin{array}{lcl} 
 \mathbf{Y}_{01}^2 & = & \mathbf{Y}_{10}^2 + \mathbf{Y}_{11}^2  \,,\\[4pt]
 \mathbf{Y}_{00}^2 & = & \mathbf{Y}_{10}^2    +  \big(1 - \Lambda_2\big ) \mathbf{Y}_{11}^2  \,. \end{array} \right.
\eeq
We have the following:
\begin{lemma}
\label{lem:EC_isomorphic2}
$\mathcal{E}_l$ is birational equivalent to $\mathcal{I}_l$ for $l= 1,2$ over $\mathbb{Q}(\Lambda_l)$.
\end{lemma}
\begin{proof}
A rational map $\mathcal{E}_1  \hookrightarrow \mathbb{P}^3, [ x_1: y_1: z_1 ]  \mapsto [ \mathbf{X}_{00}, \mathbf{X}_{01}, \mathbf{X}_{10}, \mathbf{X}_{11} ]$ is given by
\beq
\label{eqn:varphi}
  \Big[ \mathbf{X}_{00} :   \mathbf{X}_{01} :   \mathbf{X}_{10} :   \mathbf{X}_{11} \Big]  =  \Big[ x_1^2 - 2 \Lambda_1 x_1 z_1+\Lambda_1 z_1^2:  x_1^2 - \Lambda_1 z_1^2: x_1^2 - 2x_1z_1+\Lambda_1 z_1^2  : \pm 2 y_1z_1  \Big] \,.
\eeq
It has a rational inverse $\mathcal{I}_1 \dasharrow \mathcal{E}_1,  [ \mathbf{X}_{00}, \mathbf{X}_{01}, \mathbf{X}_{10}, \mathbf{X}_{11} ] \mapsto [ x_1: y_1: z_1 ]$, given by
\beq
  \Big[ x_1 : y_1 : z_1 \Big]  \ = \ \Big[ \Lambda_1 \big(\mathbf{X}_{1,0} - \mathbf{X}_{0,0}\big) : \ \pm \Lambda_1 \big(\Lambda_1-1\big) \mathbf{X}_{1,1} : \ \big(1-\Lambda_1\big) \mathbf{X}_{0,1} + \Lambda_1 \mathbf{X}_{1,0} - \mathbf{X}_{0,0} \Big] \,.
\eeq
Thus, we obtain a birational equivalence between $\mathcal{E}_1$ and $\mathcal{I}_1$.  An analogous argument holds for $\mathcal{E}_2$ and $\mathcal{I}_2$.
\end{proof}
\par There is a well defined map $\tilde{\pi}: \mathcal{I}_1 \times \mathcal{I}_2 \dasharrow  \mathbb{P}^3$ with $\mathbb{P}^3 = \mathbb{P}(\mathbf{Z}_{00}, \mathbf{Z}_{01}, \mathbf{Z}_{10}, \mathbf{Z}_{11})$ where one sets
\beq 
\label{eqn:projectionP3}
    \Big[ \mathbf{Z}_{00} :   \mathbf{Z}_{01} :   \mathbf{Z}_{10} :   \mathbf{Z}_{11} \Big] 
    \ = \  \Big[ \mathbf{X}_{00} \mathbf{Y}_{00} : \  \mathbf{X}_{01} \mathbf{Y}_{01} :  \ \mathbf{X}_{10} \mathbf{Y}_{10} : \  \mathbf{X}_{11}\mathbf{Y}_{11}  \Big]  \,,
\eeq
for
\beq
   \Big[ \mathbf{X}_{00} :   \mathbf{X}_{01} :   \mathbf{X}_{10} :   \mathbf{X}_{11} \Big]  \in \mathcal{I}_1 \,, \qquad
    \Big[ \mathbf{Y}_{00} :   \mathbf{Y}_{01} :   \mathbf{Y}_{10} :   \mathbf{Y}_{11} \Big]  \in \mathcal{I}_1 \,.
\eeq
We have the following:
\begin{theorem}
\label{prop:KUMC0}
Assume that $\Lambda_1, \Lambda_2 \in  \mathbb{P}^1 \backslash \{ 0, 1, \infty\}$ and $\Lambda_1 \not = \Lambda_2$. The image $\tilde{\pi}(\mathcal{I}_1 \times \mathcal{I}_2)$ in $\mathbb{P}^3 = \mathbb{P}(\mathbf{Z}_{00}, \mathbf{Z}_{01}, \mathbf{Z}_{10}, \mathbf{Z}_{11})$ is the quartic projective surface $\mathcal{V}_0$, given by
\beq
\label{eqn:K3_X}
\begin{split}
 \mathcal{V}_0: \quad \left\lbrace \begin{array}{c} \mathbf{Z}_{00}^4 + \big(1-\Lambda_1\big) \big(1-\Lambda_2\big)  \mathbf{Z}_{01}^4 +  \Lambda_1 \Lambda_2 \mathbf{Z}_{10}^4 
 + \Lambda_1 \Lambda_2  \big(1-\Lambda_1\big) \big(1-\Lambda_2\big)  \mathbf{Z}_{11}^4 \\[4pt]
- \big(2-\Lambda_1-\Lambda_2\big) \Big( \mathbf{Z}_{00}^2 \mathbf{Z}_{01}^2 +  \Lambda_1 \Lambda_2  \mathbf{Z}_{10}^2 \mathbf{Z}_{11}^2 \Big) 
- \big( 2 \Lambda_1 \Lambda_2 - \Lambda_1 - \Lambda_2  \big)  \Big( \mathbf{Z}_{00}^2 \mathbf{Z}_{11}^2 +   \mathbf{Z}_{01}^2 \mathbf{Z}_{10}^2 \Big) \\[4pt]
- \big(\Lambda_1 + \Lambda_2 \big)  \Big( \mathbf{Z}_{00}^2 \mathbf{Z}_{10}^2 +  \big(1-\Lambda_1\big) \big(1-\Lambda_2\big)  \mathbf{Z}_{01}^2 \mathbf{Z}_{11}^2 \Big)  \ = \ 0 \,. \end{array} \right.
\end{split} 
\eeq
The minimal resolution is isomorphic to a Kummer surface  of Picard rank 18.  In particular, the surfaces $\mathcal{V}_0$ and $\mathcal{W}_0$ in Equation~(\ref{kummer_middle_twist}) for parameters satisfying~(\ref{eqn:EC_12_j_invariants}) are birational equivalent over $\mathbb{Q}( \lambda_2 \lambda_3,  \lambda_2 + \lambda_3)$. 
\end{theorem}
\begin{proof}
We multiply (pairwise) Equations~(\ref{eqn:intersections_n_1}) and~(\ref{eqn:intersections_n_1}) and then use the variables in Equation~(\ref{eqn:projectionP3}). Upon eliminating the remaining variables we obtain Equation~(\ref{eqn:K3_X}).  For elliptic moduli $K_l$ and complementary elliptic moduli $K'_l$ with $\Lambda_l = K_l^2 = 1 - (K'_l)^2  \in \mathbb{P}^1 \backslash \lbrace 0, 1, \infty \rbrace$ for $l=1, 2$, the surface $\mathcal{V}_0$ is isomorphic over $\mathbb{Q}(\sqrt{K_1K_2}, \sqrt{K'_1 K'_2})$ to the equation in $\mathbb{P}^3 = \mathbb{P}(\mathbf{w}, \mathbf{x}, \mathbf{y}, \mathbf{z})$ given by
\beq
\begin{split}
\label{eqn:K3_XX}
0    \ = \ \mathbf{w}^4 +  \mathbf{x}^4 +  \mathbf{y}^4  + \mathbf{z}^4 
  - \frac{(K_1')^2+(K_2')^2}{K_1' K_2'} \Big( \mathbf{w}^2 \mathbf{z}^2 + \mathbf{x}^2 \mathbf{y}^2  \Big) \;\\
  - \frac{K_1^2+K_2^2}{K_1 K_2} \Big(  \mathbf{w}^2 \mathbf{y}^2 + \mathbf{x}^2 \mathbf{z}^2 \Big)
  + \frac{(K_1 K'_2)^2+(K'_1K_2)^2}{K_1 K_2 K_1'K_2'}  \Big( \mathbf{w}^2 \mathbf{x}^2 +  \mathbf{y}^2 \mathbf{z}^2 \Big) \,.
\end{split} 
\eeq
In fact, in Equation~(\ref{eqn:K3_X}) we can rescale
\beq
  \mathbf{Z}_{00} = \mathbf{x} \,, \quad
  \mathbf{Z}_{01} = \frac{ \mathbf{y} }{\sqrt{K'_1 K'_2}} \,, \quad
  \mathbf{Z}_{10} = \frac{ \mathbf{z} }{\sqrt{K_1K_2} } \,, \quad 
  \mathbf{Z}_{11} = \frac{ \mathbf{w} }{\sqrt{K_1K_2} \sqrt{K'_1 K'_2}} \,, 
\eeq 
and obtain Equation~(\ref{eqn:K3_XX}). We check that the latter is a G\"opel-Hudson quartic in Equation~(\ref{Goepel-Quartic}) with $D=0$.  According to Proposition~\ref{thm:extra_auto}, its minimal resolution is a Jacobian Kummer surface  of Picard rank $18$.
\par For parameters satisfying Equations~(\ref{eqn:EC_12_j_invariants}), a map $\mathcal{V}_0 \rightarrow  \mathcal{U}_0$, given by
\beq
\label{eqn:subi_identify}
\begin{array}{rl}
 \multicolumn{2}{c}{[ \mathbf{Z}_{00} :   \mathbf{Z}_{01} :   \mathbf{Z}_{10} :   \mathbf{Z}_{11}] \ \ \mapsto  \ \ [ \mathbf{W}: \mathbf{X}: \mathbf{Y}: \mathbf{Z}]\,, }\\[6pt]
\text{with}  \qquad  
\mathbf{W} 	 = \!& \mathbf{Z}_{0,1} -   \mathbf{Z}_{1,0} -   \mathbf{Z}_{1,1} \,, \\[4pt]
\mathbf{X} 	 = &- (1-\lambda_2)(1-\lambda_3)\mathbf{Z}_{0,0} + (1 + \lambda_2\lambda_3) \mathbf{Z}_{0,1}  - (\lambda_2 + \lambda_3) \mathbf{Z}_{1,0} \,,\\ [4pt]
\mathbf{Y} 	 = & \lambda_2 \lambda_3 \big(  \mathbf{Z}_{0,1} -   \mathbf{Z}_{1,0} +   \mathbf{Z}_{1,1} \big) \,,\\[4pt]
\mathbf{Z} 	 = & -\lambda_2\lambda_3 \big((1-\lambda_2)(1-\lambda_3)\mathbf{Z}_{0,0} +  (1 + \lambda_2\lambda_3) \mathbf{Z}_{0,1}  -  (\lambda_2 + \lambda_3) \mathbf{Z}_{1,0} \big)\,,
\end{array}
\eeq
is an isomorphism, defined over $\mathbb{Q}( \lambda_2 \lambda_3,  \lambda_2 + \lambda_3)$, between the quartic surface $\mathcal{V}_0$ in Equation~(\ref{eqn:K3_X}) and $\mathcal{U}_0$ in Equation~(\ref{eqn:Baker_det}) with $\lambda_1=\lambda_2\lambda_3$. Substituting Equations~(\ref{eqn:subi_identify}) into Equation~(\ref{eqn:Baker_det}), we obtain Equation~(\ref{eqn:K3_X}) up to a non-vanishing scale factor. Inverting Equations~(\ref{eqn:subi_identify}) yields
\beq
\begin{split}
\mathbf{Z}_{0,0} & = \lambda_2 \lambda_3 \mathbf{X} + \mathbf{Z} \,,\\
\mathbf{Z}_{0,1} & = \lambda_2 \lambda_3 (\lambda_2 + \lambda_3) \mathbf{W} -  \lambda_2 \lambda_3 \mathbf{X}  +  (\lambda_2 + \lambda_3) \mathbf{Y} + \mathbf{Z} \,,\\
\mathbf{Z}_{1,0} & =  \lambda_2 \lambda_3  (\lambda_2 \lambda_3+1)  \mathbf{W} -  \lambda_2 \lambda_3 \mathbf{X}  +  (\lambda_2 \lambda_3+1) \mathbf{Y} + \mathbf{Z} \,,\\
\mathbf{Z}_{1,1} & =- \lambda_2 \lambda_3  (1 -\lambda_2) (1- \lambda_3)  \mathbf{W} + (1 -\lambda_2) (1- \lambda_3)  \mathbf{Y}  \,.
\end{split}
\eeq
Moreover, Equations~(\ref{eqn:transfo_1}) and~(\ref{eqn:transfo_2}) provide a birational equivalence between $\mathcal{U}_0$ and $\mathcal{W}_0$. It is easy to see that for $\lambda_1=\lambda_2 \lambda_3$ in Equations~(\ref{eqn:Ls}) this equivalence is well defined over $\mathbb{Q}( \lambda_2 \lambda_3,  \lambda_2 + \lambda_3)$ 
\end{proof}
\begin{remark}
The surface $\mathcal{V}_0$ in  Equation~(\ref{eqn:K3_X}) is the Kummer-surface analogue of the genus-two curve in Equation~(\ref{eqn:LegendreSerreCurve}) obtained by Legendre's gluing method.
\end{remark}
\subsection{Geometric isogeny}
\label{ssec:geom2isog}
We proved in Proposition~\ref{prop:quotient_maps} that the smooth curve $\mathcal{C}_0$ in Equation~(\ref{Eq:Rosenhain_special}) admits an elliptic involution with the elliptic-curve quotients $\mathcal{E}_l$ for $l= 1,2$ and rational quotient maps $\pi_{\mathcal{E}_l}:  \mathcal{C}_0 \dasharrow \mathcal{E}_l$.  In this situation, we consider the rational map, given by
\beq
\label{eqn:rational_map}
 \psi_0 : \quad \mathcal{C}_0 \times \mathcal{C}_0 \  \longrightarrow  \ \mathcal{E}_1 \times \mathcal{E}_2 \,, \qquad  (P,Q) \  \mapsto \ \Big( \pi_{\mathcal{E}_1}(P) \oplus  \pi_{\mathcal{E}_1}(Q) , \;   \pi_{\mathcal{E}_2}(P) \oplus  \pi_{\mathcal{E}_2}(Q) \Big)\,,
\eeq 
where the symbol $\oplus$ refers to the addition of two points on the elliptic curve $\mathcal{E}_1$ and $\mathcal{E}_2$, respectively. One can easily show that the argument from \cite{MR2214473}*{Sec.~4.6} extends, and the map $\psi_0$ induces a geometric isogeny $\psi$ between the Shioda sextic $\mathcal{W}_0$ in Equation~(\ref{kummer_middle}) with $\lambda_1=\lambda_2\lambda_3$ and the double quadric surface $\mathcal{Z}$ in Equation~(\ref{eqn:Kummer44}). If we additionally use the fact that the Jacobian $\operatorname{Jac}{(\mathcal{C}_0)}$ is birational to the symmetric product $\operatorname{Sym}^2(\mathcal{C}_0)$,  it follows that the map $\psi$ is related to the $(2,2)$-isogeny $\Psi$ in Equation~(\ref{eqn:Psi}) by the following commutative diagram:
\beq
\label{eqn:defn_psi}
\begin{array}{rclcccl}
  \mathcal{C}_0 \times \mathcal{C}_0 & \longrightarrow  & \mathcal{W}_0 & \longleftarrow & \operatorname{Kum}(\operatorname{Jac}{\mathcal{C}_0} )  & \longleftarrow & \operatorname{Jac}{(\mathcal{C}_0 )}  \\[4pt]
  {\scriptstyle \psi_0 } \downarrow & & \downarrow {\scriptstyle \psi}  & &  \downarrow & &  \downarrow  {\scriptstyle \Psi }\\[4pt]
  \mathcal{E}_1 \times \mathcal{E}_2 & \longrightarrow &  \mathcal{Z} &  \longleftarrow  &\operatorname{Kum}(\mathcal{E}_1 \times \mathcal{E}_2)  & \longleftarrow & \mathcal{E}_1 \times \mathcal{E}_2
\end{array}  
\eeq
Here, the middle horizontal arrows represent minimal resolutions of singularities.  We will now show that, relative to the quartic surface $\mathcal{V}_0$ in Equation~(\ref{eqn:K3_X}), the geometric isogeny $\psi$ will take a simple form and can be constructed explicitly.
\par  By construction of the map $\psi_0$ and Proposition~\ref{prop:quotient_maps}, the isogeny $\psi$ is defined over the field $\mathbb{Q}(q, r^2) \cong \mathbb{Q}(\sqrt{\lambda_2 \lambda_3}, \lambda_1 + \lambda_2)$ for $q^2=\lambda_2\lambda_3$ and $r^2=(1-\lambda_2)(1-\lambda_3)$. We then have the following:
\begin{lemma}
\label{lem:map_psi}
For the holomorphic two-forms  $\omega_{\mathcal{W}_0} = dz_2 \wedge dz_3/\tilde{z}_4$ in the chart $z_1=1$ in Equation~(\ref{kummer_middle}) and $\omega_\mathcal{Z} = dx_1 \wedge dx_2/y_{1,2}$ in the chart $z_1=z_2=1$ in Equation~(\ref{eqn:Kummer44}), it follows
\beq
\label{eqn:mathc_two_forms}
 \psi^* \left(  \frac{dx_1 \wedge dx_2}{y_{12}} \right) = \delta \cdot \frac{dz_2 \wedge dz_3}{\tilde{z}_4} 
\eeq
with $\delta = 2\sqrt{\lambda_2 \lambda_3} (1-\lambda_2)(1- \lambda_3)$.
\end{lemma}
\begin{proof}
 The smooth genus-two curve $\mathcal{C}_0$ given by Equation~(\ref{Eq:Rosenhain_special}) is defined by a polynomial of degree five.  Thus, in the chart $Z=1$, the regular differentials from a vector space with a basis $dX/Y$ and $X dX/Y$. Similarly, in the charts $z_l=1$ on the elliptic curves $\mathcal{E}_l$ holomorphic one-forms are given by $dx_l/y_l$ for $l=1, 2$. One checks by a direct computation that their pullbacks are give by
\beq
 p_l^* \left( \frac{dx_l}{y_l} \right) = \Big( p_{l1} X + p_{l0} \Big) \, \frac{dX}{Y} =  \Big( r  X - (-1)^l r q \Big) \, \frac{dX}{Y}  \,.
\eeq
The independence of the pullbacks is equivalent to the non-vanishing of the quantity
\beq
 \delta = p_{21} p_{10} - p_{20} p_{11}  = 2 r^2 q \not = 0 \,.
\eeq 
The regular two-from $dx_1\wedge dx_2/(y_1y_2)$ in Equation~(\ref{eqn:two-form}) on $\mathcal{E}_1 \times \mathcal{E}_2$ is the pullback (along the horizontal arrow in Equation~(\ref{eqn:defn_psi})) of the holomorphic two-form $\omega_\mathcal{Z}$. It was shown in \cite{MR2214473}*{Sec.~4.6} that one has
\beqn
 \psi_0^* \left(  \frac{dx_1}{y_1} \wedge \frac{dx_2}{y_2} \right) = \delta  \big( X^{(2)} - X^{(1)} \big)  \frac{d X^{(1)} \wedge d X^{(2)}}{Y^{(1)}Y^{(2)}} 
 = \delta \frac{ d\left( X^{(1)}  + X^{(2)} \right) \wedge d\left( X^{(1)}  X^{(2)} \right) }{Y^{(1)}Y^{(2)}} \,,
\eeqn
where $[X^{(l)} : Y^{(l)}: Z^{(l)}]$ for $l= 1,2$ are the coordinates of the two copies of $\mathcal{C}_0$.
\end{proof}
One avoids the pre-factor $\delta$ in Equation~(\ref{eqn:mathc_two_forms}) by using quadratic twists:
\begin{corollary}
\label{cor:psi}
The map $\psi$ extends to a rational map between the surfaces $\mathcal{W}^{(\varepsilon)}_0$ in Equation~(\ref{kummer_middle_twist}) and $\mathcal{Z}^{(2^4)}$ in Equation~(\ref{eqn:Kummer44_b}), i.e., 
\beq
 \psi: \quad \mathcal{W}^{(\varepsilon)}_0 \ \dasharrow \ \mathcal{Z}^{(2^4)} \,, \qquad \text{with} \quad  \varepsilon =  \frac{4}{\lambda_2 \lambda_3 (1-\lambda_2)^2(1-\lambda_3)^2} \,,
\eeq
such that the holomorphic two-forms  $\omega_{\mathcal{W}^{(\varepsilon)}_0} = dz_2 \wedge dz_3/\hat{z}_4$ in the chart $z_1=1$ and $\omega_{\mathcal{Z}^{(2^4)}} = dx_1 \wedge dx_2/\hat{y}_{1,2}$ in the chart $z_1=z_2=1$ satisfy $\omega_{\mathcal{W}^{(\varepsilon)}_0} = \psi^* \omega_{\mathcal{Z}^{(2^4)}}$.
\end{corollary}
\begin{proof}
An isomorphism between $\mathcal{W}_0$ in Equation~(\ref{kummer_middle})  and $\mathcal{W}^{(\varepsilon)}_0$ in Equation~(\ref{kummer_middle_twist}) is given by $\tilde{z}_4 = \hat{z}_4/\sqrt{\epsilon}$. Similarly, an isomorphism between $\mathcal{Z}$ in Equation~(\ref{kummer_middle})  and $\mathcal{Z}^{(2^4)}$ in Equation~(\ref{kummer_middle_twist}) is given by $y_{1,2} = \hat{y}_{1,2}/4$. Equation~(\ref{eqn:mathc_two_forms}) then becomes $\omega_{\mathcal{W}^{(\varepsilon)}_0} = \psi^* \omega_{\mathcal{Z}^{(2^4)}}$.
\end{proof}
\par On the other hand, the double quadric $\mathcal{Z}$ in Equation~(\ref{eqn:Kummer44}) is related to the surface $\mathcal{V}_0$ in Equation~(\ref{eqn:K3_X}) as follows:
\begin{lemma}
\label{lem:phi}
In the commutative diagram
\beq
\begin{array}{rcl}
  \mathcal{E}_1 \times \mathcal{E}_2 & \overset{\cong}{\longrightarrow } &  \mathcal{I}_1 \times  \mathcal{I}_2 \\[4pt]
  {\pi} \downarrow  & & \downarrow {\tilde{\pi}} \\[4pt]
  \mathcal{Z} & \overset{\phi_\pm}{\longrightarrow} & \mathcal{V}_0 
\end{array}  
\eeq
the maps $\phi_\pm: \mathcal{Z} \dasharrow \mathcal{V}_0$ are rational maps of degree two given by
\beq
\label{eqn:subi_X}
\begin{array}{rl}
 \multicolumn{2}{l}{ \phi_\pm: \qquad (x_1, z_1, x_2, z_2, y_{1,2})  \ \ \mapsto \ \ [ \mathbf{Z}_{00} :   \mathbf{Z}_{01} :   \mathbf{Z}_{10} :   \mathbf{Z}_{11}]\,, }\\[6pt]
\text{with} \qquad   
  \mathbf{Z}_{00}   = & \Big(x_1^2 - 2 \Lambda_1 x_1 z_1+\Lambda_1 z_1^2\Big)  \Big(x_2^2 - 2 \Lambda_2 x_2 z_2+\Lambda_2 z_2^2\Big) \,,\\[4pt]
  \mathbf{Z}_{01}   = & \Big(x_1^2 - \Lambda_1 z_1^2 \Big) \Big(x_2^2 - \Lambda_2 z_2^2\Big)\,,\\[4pt]
  \mathbf{Z}_{10}   = &\Big(x_1^2 - 2 x_1 z_1+\Lambda_1 z_1^2\Big)  \Big(x_2^2 - 2 x_2 z_2+\Lambda_2 z_2^2\Big) \,,\\[6pt]
  \mathbf{Z}_{11}   = & \pm 4 y_{1,2}\,,
\end{array}
\eeq
such that $\phi_\pm \circ \pi \circ \imath_{\mathcal{E}_l} = \phi_\mp \circ \pi$ for $l= 1,2$ where $\imath_{\mathcal{E}_l}$ is the hyperelliptic involution on $\mathcal{E}_l$.
\end{lemma}
\begin{proof}
Equation~(\ref{eqn:subi_X}) follows immediately from the equivalence $\mathcal{E}_1 \times \mathcal{E}_2 \cong  \mathcal{I}_1 \times  \mathcal{I}_2 $ given by Lemma~\ref{lem:EC_isomorphic2} (with a choice $\pm 1$ for the relative sign of $y_1$ and $y_2$) and the fact that the projection map $\pi: \mathcal{E}_1 \times  \mathcal{E}_2 \to \mathcal{Z}$ is given by $y_{1,2} = z_1 z_2 y_1y_2$. The fact that the map is a rational map of degree two follows directly from the fact that Equation~(\ref{eqn:Kummer44}) is a double quadric surface.
\end{proof}
For the elliptic curve $\mathcal{E}_l$ with $\Lambda_l  \in \mathbb{P}^1 \backslash \lbrace 0, 1, \infty \rbrace$ for $l=1, 2$, multiplication by two is given by the composition of two-isogenies $\chi_{\mathcal{E}_l}^\prime \circ \chi_{\mathcal{E}_l}^{\phantom{vee}}\!\!$, or, equivalently, by the map
\beq
\label{eqn:multiplication2}
\begin{split}
 x_l \ \mapsto \ &2 \big(x_l^2-\Lambda_l z_l\big)^2 y_l z_l \,,\\
 y_l \ \mapsto \ & \big(x_l^2 - 2 \Lambda_l x_l z_l + \Lambda_l z_l^2\big) \big(x_l^2 -2 x_l z_l + \Lambda_l z_l^2\big) \big(x_l^2-\Lambda_l z_l^2\big) \,,\\
 z_l \ \mapsto \ & 8 y_l^3 z_l^3 \,.
\end{split}
\eeq
The two-isogenies $\chi_{\mathcal{E}_l}$ can then be expressed as rational maps originating from the complete intersection of quadrics $\mathcal{I}_1$ in Equation~(\ref{eqn:intersections_n_1}) (resp.~$\mathcal{I}_2$ in Equation~(\ref{eqn:intersections_n_2})). It follows that
\beq 
\chi^\prime_{\mathcal{E}_1} \circ \chi^{\phantom{\prime}}_{\mathcal{E}_1}: \quad   [\mathbf{X}_{00}, \mathbf{X}_{01}, \mathbf{X}_{10}, \mathbf{X}_{11} ] \mapsto  [x_1: y_1: z_1] 
 = [ \mathbf{X}_{0,1}^2  \mathbf{X}_{1,1}:   \mathbf{X}_{0,0}\mathbf{X}_{0,1}\mathbf{X}_{1,0}: \mathbf{X}_{1,1}^3 ] \,,
\eeq
and a similar formula holds for $\chi^\prime_{\mathcal{E}_2} \circ \chi_{\mathcal{E}_2}^{\phantom{\prime}}$. The composition of their Cartesian products, given by
\beq
\label{eqn:comp}
  \mathcal{I}_1 \times \mathcal{I}_2 \   \overset{ {\scriptscriptstyle\chi_{\mathcal{E}_1} \times \chi_{\mathcal{E}_2}}}{\longrightarrow } \  \mathcal{E}^\prime_1  \times \mathcal{E}^\prime_2   \ 
  \overset{ {\scriptscriptstyle\chi^\prime_{\mathcal{E}_1} \times \chi^\prime_{\mathcal{E}_2}}}{\longrightarrow } \  \mathcal{E}_1 \times  \mathcal{E}_2 \,,
\eeq
induces a rational map from $\mathcal{V}_0$ to $\mathcal{Z}$. We have the following:
\begin{proposition}
\label{lem:psi}
The family of rational maps $\psi_\pm: \mathcal{V}_0 \dasharrow \mathcal{Z}$ of degree two given by 
\beq
\label{eqn:subi_Y}
\begin{array}{rlcrl}
 \multicolumn{5}{l}{ \psi_\pm: \qquad [ \mathbf{Z}_{00} :   \mathbf{Z}_{01} :   \mathbf{Z}_{10} :   \mathbf{Z}_{11}] \ \ \mapsto \ \ (x_1, z_1, x_2, z_2, y_{1,2})   \,,}\\[6pt]
\text{with} \qquad   
 x_1 & = Q \,, &  & z_1 & =\, (\Lambda_1-\Lambda_2) \mathbf{Z}_{1,1}^2 \,, \\[4pt]
 x_2 & = (\Lambda_1-\Lambda_2) \mathbf{Z}_{0,1}^2 \,, &&  z_2 & = Q \,, \\[4pt]
 y_{1,2}& \multicolumn{4}{l}{= \pm (\Lambda_1-\Lambda_2)^2 Q^2 \mathbf{Z}_{0,0}  \mathbf{Z}_{0,1}  \mathbf{Z}_{1,0}  \mathbf{Z}_{1,1} \,,} 
\end{array}
\eeq
with
\beq
\label{eqn:PQ}
Q  =  \mathbf{Z}_{0,0}^2 - (1-\Lambda_1)  \mathbf{Z}_{0,1}^2 - \Lambda_1  \mathbf{Z}_{1,0}^2 +  \Lambda_1 (1-\Lambda_2)  \mathbf{Z}_{1,1}^2 \,,
\eeq
satisfies that the composition $\psi_\pm \circ \phi_\pm: \mathcal{Z} \dasharrow \mathcal{Z}$ is the diagonal action of multiplication by two on $\mathcal{E}_1 \times \mathcal{E}_2$ in Equation~(\ref{eqn:multiplication2}).
\end{proposition}
\begin{proof}
The composition of elliptic-curve isogenies in Equation~(\ref{eqn:comp}) is given by
\beq
  \left( \begin{array}{c} 
  	\lbrack \mathbf{X}_{00} :   \mathbf{X}_{01} :   \mathbf{X}_{10} :   \mathbf{X}_{11} \rbrack  \\ 
	\lbrack \mathbf{Y}_{00} :   \mathbf{Y}_{01} :   \mathbf{Y}_{10} :   \mathbf{Y}_{11} \rbrack  \end{array} \right)  \mapsto 
  \left( \begin{array}{l} 
  \lbrack x_1 : z_1 : y_1\rbrack = \lbrack \mathbf{X}_{0,1}^2 \mathbf{X}_{1,1} :  \mathbf{X}_{1,1}^3:   \mathbf{X}_{0,0} \mathbf{X}_{0,1} \mathbf{X}_{1,0}  \rbrack \\
  \lbrack x_2 : z_2 : y_2\rbrack =  \lbrack \mathbf{Y}_{0,1}^2 \mathbf{Y}_{1,1} :  \mathbf{Y}_{1,1}^3:  \mathbf{Y}_{0,0} \mathbf{Y}_{0,1} \mathbf{Y}_{1,0}  \rbrack 
  \end{array} \right).
\eeq
The relation, given by
\beq 
    \Big[ \mathbf{Z}_{00} :   \mathbf{Z}_{01} :   \mathbf{Z}_{10} :   \mathbf{Z}_{11} \Big] 
    \ = \  \Big[ \mathbf{X}_{00} \mathbf{Y}_{00} : \  \mathbf{X}_{01} \mathbf{Y}_{01} :  \ \mathbf{X}_{10} \mathbf{Y}_{10} : \  \mathbf{X}_{11}\mathbf{Y}_{11}  \Big]  \,,
\eeq
induces a correspondence between $\mathcal{Z}$ and $\mathcal{V}_0$, given by
\beq
\begin{split}
x_1 x_2 = \mathbf{Z}_{0,1}^2 \,, \qquad z_1 z_2 =  \mathbf{Z}_{1,1}^2 \,, \qquad y_{1,2} &= \mathbf{Z}_{0,0}\mathbf{Z}_{0,1}\mathbf{Z}_{1,0}\mathbf{Z}_{1,1} \,,\\
(\Lambda_1-\Lambda_2) x_1 z_2 = Q \,, \qquad (\Lambda_2-\Lambda_1)  x_2 z_1 = R \,, \qquad QR &= -(\Lambda_1-\Lambda_2)^2 \mathbf{Z}_{0,1}^2 \mathbf{Z}_{1,1}^2 \,,
\end{split}
\eeq
with $Q=(\Lambda_1-\Lambda_2) \mathbf{X}_{01}^2\mathbf{Y}_{11}^2$ and $R=(\Lambda_2-\Lambda_1)  \mathbf{X}_{11}^2\mathbf{Y}_{01}^2$.  One then obtains
\beq
\begin{split}
Q & =  \mathbf{Z}_{0,0}^2 - (1-\Lambda_1)  \mathbf{Z}_{0,1}^2 - \Lambda_1  \mathbf{Z}_{1,0}^2 +  \Lambda_1 (1-\Lambda_2)  \mathbf{Z}_{1,1}^2 \,,\\
R & =   \mathbf{Z}_{0,0}^2 - (1-\Lambda_2)  \mathbf{Z}_{0,1}^2 - \Lambda_2  \mathbf{Z}_{1,0}^2 + \Lambda_2 (1-\Lambda_1)  \mathbf{Z}_{1,1}^2 \,.
\end{split}
\eeq
When solving for an equivalence class, given by
\beq
 \Big( x_1, z_1, x_2, z_2, y_{1,2} \Big) \ \sim \  \Big( \mu x_1, \ \mu z_1, \ \nu x_2,\  \nu z_2, \ \mu^2\nu^2 y_{1,2} \Big)
 \; \text{with $\mu, \nu \in \mathbb{C}^\times$} \,,
\eeq 
we obtain Equations~(\ref{eqn:subi_Y}). Using the relation $QR = -(\Lambda_1-\Lambda_2)^2 \mathbf{Z}_{0,1}^2 \mathbf{Z}_{1,1}^2$ the solution can also be written as
\beq
\label{eqn:case2}
\begin{split}
 x_1 = (\Lambda_2-\Lambda_1) \mathbf{Z}_{0,1}^2  \,, &\quad 
 z_1 = \, R \,, \quad
 x_2 =R\,, \quad 
 z_2 =  (\Lambda_2-\Lambda_1) \mathbf{Z}_{1,1}^2  \,, \\
y_{1,2} & = (\Lambda_1-\Lambda_2)^2 R^2 \mathbf{Z}_{0,0}  \mathbf{Z}_{0,1}  \mathbf{Z}_{1,0}  \mathbf{Z}_{1,1} \,.
\end{split} 
\eeq
Equation~(\ref{eqn:multiplication2}) provides an explicit formula for the (diagonal) action of multiplication by two on $\mathcal{E}_1 \times \mathcal{E}_2$. We apply the natural projection map $\pi: \mathcal{E}_1 \times  \mathcal{E}_2 \to \mathcal{Z}$ given by $y_{1,2} = z_1 z_2 y_1y_2$ to obtain the induced action on the  the double quadric $\mathcal{Z}$ in Equation~(\ref{eqn:Kummer44}). We check that this agrees with the composition of maps $\psi_\pm \circ \phi_\pm: \mathcal{Z} \dasharrow \mathcal{Z}$, i.e., we have
\beq
 \psi_\pm \circ \phi_\pm\ = \ \pi\,\Big( \big(\chi^\prime_{\mathcal{E}_1} \times \chi^\prime_{\mathcal{E}_2}\big) \circ  \big(\chi_{\mathcal{E}_1} \times \chi_{\mathcal{E}_2}\big) \Big)\,. \qedhere
\eeq
\end{proof}
The map $\psi_-$ in Equation~(\ref{eqn:subi_Y}) can be related to the geometric isogeny $\psi$ in Equation~(\ref{eqn:defn_psi}). We have the following:
\begin{proposition}
\label{thm:coincidence}
The map $\psi$ in Equation~(\ref{eqn:defn_psi}) with the normalization determined by Equation~(\ref{eqn:mathc_two_forms}) coincides with $\psi_-$ in Proposition~\ref{lem:psi}.
\end{proposition}
\begin{proof}
Considering the finite field extension $\mathbb{Q}(k_1, k_2)$ with
\beq
\label{eqn:moduli_2}
 \Lambda_1 =  - \dfrac{(k_2 - k_3)^2}{(1- k_2^2) (1-k_3^2)} \,, \qquad
 \Lambda_2 =  - \dfrac{(k_2 + k_3)^2}{(1- k_2^2) (1-k_3^2)}  \,,
\eeq 
the maps $\psi_\pm: \mathcal{V}_0 \dasharrow \mathcal{Z}$ in Proposition~\ref{lem:psi} are equivalently given by 
\beq
\label{eqn:subi_Yb}
\psi_\pm: \qquad [ \mathbf{Z}_{00} :   \mathbf{Z}_{01} :   \mathbf{Z}_{10} :   \mathbf{Z}_{11}] \ \ \mapsto \ \ (x_1, z_1, x_2, z_2, y_{1,2})   \,,
\eeq
with
\beq
\label{eqn:case1_b}
\begin{split}
 x_1 & = S_+ \,, \qquad z_1 =4 k_2 k_3  (1-k_2^2) (1-k_3^2) \mathbf{Z}_{1,1}^2\,, \\
 x_2 & = S_- \,, \qquad z_2 = 4 k_2 k_3  (1-k_2^2) (1-k_3^2) \mathbf{Z}_{1,1}^2 \,, \\
 y_{1,2} & = \pm 256 k_2^4 k_3^4 (1-k_2^2)^2 (1-k_3^2)^2  \mathbf{Z}_{0,0}  \mathbf{Z}_{0,1}  \mathbf{Z}_{1,0}  \mathbf{Z}_{1,1}^5 \,,
\end{split}
\eeq
where $S_\pm$ are the quadrics given by
\beq
\label{eqn:ST}
\begin{split}
S_\pm 	& =  \pm (1-k_2^2)^2 (1-k_3^2) \mathbf{Z}_{0,0}^2  \mp  (1 \mp k_2k_3)^2 (1-k_2^2) (1-k_3^2)  \mathbf{Z}_{0,1}^2 \\
		& \pm (k_2 \mp k_3)^2  (1-k_2^2) (1-k_3^2)  \mathbf{Z}_{1,0}^2 \mp  (k_2 \mp k_3)^2 (1 \pm k_2k_3)^2 \mathbf{Z}_{1,1}^2\,,
\end{split}
\eeq
such that $S_+ S_- = 16  k_2^2 k_3^2 (1-k_2^2)^2 (1-k_3^2)^2\mathbf{Z}_{0,1}^2\mathbf{Z}_{1,1}^2$. The proof follows from a tedious computation using {\tt Maple}: one first constructs the rational map $\psi$ in Equation~(\ref{eqn:defn_psi}) as a map from $\mathcal{W}_0$ in the chart $z_1=1$ in Equation~(\ref{kummer_middle_twist}) to $\mathcal{Z}$ in the chart $z_1=z_2=1$ in Equation~(\ref{eqn:Kummer44}) over $\mathbb{Q}(k_2, k_3)$. One then uses the identification $\mathcal{W}_0 \cong \mathcal{V}_0$ from the proof of Proposition~\ref{prop:KUMC0}. The map $\psi$ then simplifies considerably and coincides with $\psi_-$ in Equation~(\ref{eqn:case1_b}).
\end{proof}
Legendre's gluing method in Section~\ref{sssec:Legendre} determines the smooth genus-two curve in Equation~(\ref{eqn:LegendreSerreCurve}), admitting an elliptic involution with elliptic-curve quotients $\mathcal{E}_1$ and $\mathcal{E}_2$. The quotient maps determine the $(2,2)$-isogeny $\Psi: \operatorname{Jac}{(\mathcal{C}_0)}  \rightarrow  \mathcal{E}_1 \times \mathcal{E}_2$ in Equation~(\ref{eqn:Psi}) and its dual isogeny $\Phi$. The construction descends and yields the following Kummer sandwich theorem:
\begin{theorem}
\label{thm:isogenies_explicit}
The surface $\mathcal{V}_0$ in  Equation~(\ref{eqn:K3_X}) and rational maps $\phi_\pm, \psi_\pm$ in Equations~(\ref{eqn:subi_X}) and~(\ref{eqn:subi_Y}) fit into the following Kummer sandwich:
\beqn
  \operatorname{Kum}\left(\mathcal{E}_1\times \mathcal{E}_2\right)
  \ \overset{\phi_\pm}{\longrightarrow} \
  \mathcal{V}_0
  \ \overset{\psi_\pm}{\longrightarrow} \
  \operatorname{Kum}\left(\mathcal{E}_1\times \mathcal{E}_2\right) \,.
\eeqn  
The maps $\phi_\pm, \psi_\pm$ are defined on families over $\mathbb{Q}(\Lambda_1, \Lambda_2)$ and induced (up to the action of the minus identity involution) by the $(2,2)$-isogeny in Equation~(\ref{eqn:Phi}) and its dual $(2,2)$-isogeny in Equation~(\ref{eqn:Psi}) of principally polarized abelian surfaces, respectively, i.e.,  the isogenies in the following diagram:
\beq
  \mathcal{E}_1 \times \mathcal{E}_2 \ \overset{\Phi}{\longrightarrow}  \  \operatorname{Jac}{(\mathcal{C}_0)}  \ \overset{\Psi}{\longrightarrow}  \  \mathcal{E}_1 \times \mathcal{E}_2  \,.
\eeq
$\mathcal{C}_0$ is  the smooth genus-two curve admitting an elliptic involution, and $\mathcal{E}_l$ are the elliptic-curve quotients determined in Proposition~\ref{prop:quotient_maps}.
\end{theorem}
\begin{proof}
General results in \cite{MR2062673}*{Sec.~1} and Theorem~\ref{prop:isogeny_Delta} show that there is a $(2,2)$-isogeny $\Phi$ in Equation~(\ref{eqn:Phi}), given by
\beq
 \Phi: \quad\mathcal{E}_1 \times \mathcal{E}_2   \ \longrightarrow \ \operatorname{Jac}{(\mathcal{C}_0)}  \,.
\eeq
The dual isogeny $\Psi: \operatorname{Jac}{(\mathcal{C}_0)} \rightarrow  \mathcal{E}_1 \times \mathcal{E}_2$ was constructed in Equation~(\ref{eqn:Psi}). Up to isomorphisms, the composition is given by multiplication by two on each factor of $\mathcal{E}_1 \times \mathcal{E}_2$.  Proposition~\ref{thm:coincidence} shows that $\Psi$ induces the map $\psi_\pm$. The minimal resolution of the double quadric surface $\mathcal{Z}$ is isomorphic to the Kummer surface $\operatorname{Kum}(\mathcal{E}_1\times \mathcal{E}_2)$.  Moreover, the induced action of $\Psi \circ \Phi$ on the associated Kummer surfaces coincides with the one of $\psi_\pm \circ \phi_\pm$ due to Proposition~\ref{lem:psi}. On the other hand, the map $\psi_\pm \circ \phi_\pm$ factors though $\mathcal{V}_0$ by construction. Proposition~\ref{prop:KUMC0} proves that $\mathcal{V}_0$ is isomorphic to  $\mathcal{W}_0$ whose minimal resolution is $\operatorname{Kum}(\mathcal{C}_0)$ due to Proposition~\ref{prop:Kummers}.
\end{proof}
\section{K3 surfaces related to the Legendre pencil}
\label{sec:LegendrePencil}
The family of projective surfaces, given by the equation
\beq
\label{eqn:Xt}
   y^2 = z_1z_2z_3 (z_1+z_2)(z_2+z_3)  (z_1 + t z_3)   
\eeq  
with $t \in \mathbb{Q}$ is a family of  double sextic surfaces considered by van~\!Geemen and Top in \cite{MR2214473}. The minimal resolutions yield K3 surfaces of Picard rank 19. A natural generalization, referred to as the \emph{two-parameter twisted Legendre pencil} in \cites{MR894512, MR1023921, MR1013162, MR0803354,  MR1877757}, is the family $\mathcal{X}$, given by
\beq
\label{eqn:X_general}
  \mathcal{X}: \quad y^2 = z_1 z_2(z_1-z_2) (z_1-z_3) \, (z_3^2 + 2 \rho_1 z_2z_3 + \rho^2_2 z_2^2) 
\eeq  
with $\rho_1, \rho_2 \in \mathbb{Q}$ and $\rho_1 \not = \pm \rho_2$. If fact, changing variables according to $z_1 \mapsto - z_1$, $z_3\mapsto 2 \rho_1 z_3$,  and $y \mapsto 2 i \rho_1 y$, the family $\mathcal{X}$ is isomorphic to the family in Equation~(\ref{eqn:Xt}) for $(\rho_1, \rho_2) = (t/2,0)$ over $\mathbb{Q}(i)$.  A natural holomorphic two-form $\omega_\mathcal{X}$ on $\mathcal{X}$ is given by  $dz_1\wedge dz_3/y$ in the chart $z_2=1$ in Equation~(\ref{eqn:X_general}). Notice that a general member of $\mathcal{X}$ is not the Kummer associated with an abelian surface.
\subsection{Configurations of six lines}
The general member of the family in Equation~(\ref{eqn:X_general}) yields a K3 surface of Picard rank 18, and thus its transcendental lattice $\operatorname{T}_\mathcal{X}$ has rank $4$. Hoyt showed in \cites{MR894512, MR1877757} that
\beq
\label{eqn:Tlattice}
 \operatorname{T}_\mathcal{X} = \langle 2 \rangle \oplus \langle 2 \rangle \oplus \langle -2 \rangle \oplus \langle - 2 \rangle \,.
\eeq 
There is also a further generalization, known as three-parameter twisted Legendre pencil,  given by
\beq 
\label{eqn:3param_pencil}
y^2 = z_1 (z_1-z_2) (z_1-z_3) (A z_2 - z_3) (B z_2 - z_3)  (C z_2 - z_3) \,.
\eeq
The branch locus consists of six lines, three of which are coincident in a point; see \cites{MR894512, MR2854198, Clingher:2017aa}. The minimal resolutions are K3 surfaces of Picard rank greater than or equal to 17. In \cite{MR1877757} it was shown that the transcendental lattice of the K3 surface associated with a general member in Equation~(\ref{eqn:3param_pencil}) is $\langle 2 \rangle^{\oplus 2} \oplus \langle -2 \rangle^{\oplus 3}$. In particular, a general member of $\mathcal{X}$ is not the Kummer surface.
\par The case $C=\infty$ in Equation~(\ref{eqn:3param_pencil}), after rescaling, is given by
\beq
\label{eqn:Xrestriction}
  \mathcal{X}: \quad y^2 = z_1 z_2(z_1-z_2) (z_1-z_3) (A z_2 - z_3) (B z_2 - z_3) \,.
\eeq  
This pencil is isomorphic to Equation~(\ref{eqn:X_general}) over $\mathbb{Q}(\rho_1, \rho_2, \sqrt{\rho_1^2-\rho_2^2})$ with $AB=\rho_2^2$ and $A+B=-2\rho_1$. We have the following:
\begin{proposition}
\label{lem:branch_locus}
For $A \not = B$ the general member in Equation~(\ref{eqn:Xrestriction}) is branched on the union of  six lines, one of which is coincident with two pairs of lines in two distinct points and generic otherwise. The configuration is shown in Figure~\ref{fig:6LinesSpecial}. The minimal resolution has Picard rank 18. 
For the following specializations the minimal resolution has Picard rank 19:
\beqn
 (1) \ AB=0, \qquad (2) \ (A+B-4)^2-4 AB=0,   \qquad (3) \ (A-B+1)(A-B-1)=0.
\eeqn
\end{proposition}
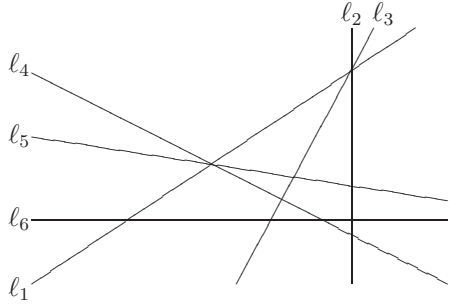
\begin{figure}[ht]
\scalemath{0.85}{
  $$
  \begin{xy}
    <0cm,0cm>;<1cm,0cm>:
    (-1,0.5);(5.5,0.5)**@{-}, 
    (-1,2.8);(5.5,-0.5)**@{-},   
    (-1,1.8);(5.5,0.8)**@{-},    
    (-1,-0.5);(5,3.5)**@{-},
    (4,3.5);(4,-0.5)**@{-},
    (4.35,3.5);(2.2,-0.5)**@{-},
    (-1.2,-1)*++!D\hbox{$\gerade_1$},
    (-1.2, 0.1)*++!D\hbox{$\gerade_6$},
    (-1.2, 1.4)*++!D\hbox{$\gerade_5$},
    (-1.2, 2.5)*++!D\hbox{$\gerade_4$},
    (4.0, 3.3)*++!D\hbox{$\gerade_2$},
    (4.5, 3.3)*++!D\hbox{$\gerade_3$},
  \end{xy}
  $$}
\caption{Line configuration for general member in Proposition~\ref{lem:branch_locus}}
\label{fig:6LinesSpecial}
\end{figure}
\begin{remark}
In terms of Equation~(\ref{eqn:X_general}) the subsets of the parameter space in Proposition~\ref{lem:branch_locus} are given by
\beqn
 (1) \ \rho_2 =0,  \qquad (2) \ (\rho_1 - \rho_2 +2)(\rho_1 + \rho_2 +2)= 0, \qquad (3) \  4 \rho_1^2 - 4 \rho_2^2 -1 =0.
\eeqn
\end{remark}
\begin{remark}
In case (1) of Proposition~\ref{lem:branch_locus} the branch locus is the configuration of six lines $\{ \gerade_1, \dots, \gerade_6\}$ in $\mathbb{P}^2$ such that, up to permutation, the triple intersections $\{ p_{123} \} = \gerade_1 \cap \gerade_2 \cap \gerade_3$, $\{ p_{145} \} =\gerade_1 \cap \gerade_4 \cap  \gerade_5$, and $\{ p_{246} \} =\gerade_2 \cap \gerade_4 \cap  \gerade_6$ consist of three pairwise distinct points $p_{123}, p_{145}, p_{246}$, and the configuration is generic otherwise. 
\end{remark}
\begin{proof}
The branch locus of Equation~(\ref{eqn:Xrestriction}) is given by the six lines
\beq
\label{lines}
 \gerade_1: \, z_2, \quad \gerade_2 : \, z_1,  \quad \gerade_3: \, z_1 - z_2,  \quad \gerade_4: \,A z_2 -  z_3,  \quad \gerade_5 :\, B z_2 - z_3, \quad \gerade_6: \,z_1 - z_3,  
\eeq
whose intersection pattern looks as follows:
\beq
\label{tab:6lines}
\scalemath{1}{
\begin{array}{c|c|c|c|c|c|c}
		& \gerade_1	& \gerade_2	& \gerade_3		& \gerade_4	& \gerade_5	& \gerade_6 \\ \hline &&&&&& \\[-0.9em]
\gerade_1	& -			& [0:0:1]		& [0:0:1]			& [1:0:0]		& [1:0:0]		& [1:0:1]		\\ &&&&&& \\[-0.9em]
\gerade_2	& [0:0:1]		& -			& [0:0:1]			& [0:1:A]		& [0:1:B]		& [0:1:0]		\\ &&&&&& \\[-0.9em]
\gerade_3	& [0:0:1]		& [0:0:1]		& -				& [1:1:A]		& [1:1:B]		& [1:1:1]		\\ &&&&&& \\[-0.9em]
\gerade_4	& [1:0:0]		& [0:1:A]		& [1:1:A]			& -			& [1:0:0]		& [A:1:A]		\\ &&&&&& \\[-0.9em]
\gerade_5	& [1:0:0]		& [0:1:B]		& [1:1:B]			& [1:0:0]		& -			& [B:1:B]		\\ &&&&&& \\[-0.9em]
\gerade_6	& [1:0:1]		& [0:1:0]		& [1:1:1]			& [A:1:A]		& [B:1:B]		& -	 		\\ &&&&&& \\[-0.9em]
\end{array}}
\eeq
Statement (1) is immediate. For (2) one can check that the double sextic in Equation~(\ref{eqn:Xrestriction}) admits a Jacobian elliptic fibration with a non-torsion section: the elliptic fibration is obtained by setting $z_1=x$, $z_2=t$, $z_3=1$. Then, for $x= - \frac{4(At-1)(Bt-1)}{(A-B)^2t}$ we have
\beqn
 y^2 = - \frac{4 (At-1)^2 (Bt-1)^2 ((A+B)t-2)^2 \big( 4ABt^2 + (A^2+B^2 -2 AB-4A-4B)t + 4 \big)}{(A-B)^6 t^2} \,.
\eeqn 
The discriminant for the quadratic term in the numerator of $y^2$ factors into the product of $(A-B)^2$ times $A^2+B^2-2AB-8A-8B-16$. Since we assumed $A \not = B$ we obtain a non-torsion section if $A^2+B^2-2AB-8A-8B-16=0$.  Thus, we set $A=\alpha^2$ and $B=(2 \pm \alpha)^2$ and obtain
\beqn
 y = \pm \frac{i (\alpha^2t-1)((\alpha-2)^2t-1)((\alpha^2-2\alpha+2)t-1)(\alpha(\alpha-2)t+1)}{8 (\alpha-1)^3 t} \,.
\eeqn
For (3) one checks that $x=\frac{At-1}{A-B}$ and $y^2=\frac{A(At-1)^2(Bt-1)^2t^2}{(A-B)^3}$ is a section if and only if $A-B+1=0$. Switching the roles of $A$ and $B$ the statement follows.
\end{proof}
\par In Equation~(\ref{eqn:Xrestriction}) the pencil of lines through the point $[1:0:0]$ (or, equivalently,  the point $[0:0:1]$) induces an elliptic fibration on $\mathcal{X}$. A simple coordinate transformation, defined over $\mathbb{Q}(\rho_1, \rho_2^2)$, yields a Weierstrass model with fibers over $\mathbb{P}^1=\mathbb{P}(v_1,v_2)$ embedded into $\mathbb{P}^2 = \mathbb{P}(x, y, z)$, given by
\beq
\label{eqn:X}
 \mathcal{X}: \quad y^2 z = x \Big( x -  v_1 v_2  \big( v_1^2 + 2 \rho_1 v_1v_2 + \rho^2_2 v_2^2 \big) z \Big)   
 \Big( x -  v_2^2  \big( v_1^2 +  2\rho_1 v_1v_2 + \rho^2_2 v_2^2 \big) z \Big)\,.
\eeq
This change of coordinates also maps via pullback the holomorphic two-form $\omega_\mathcal{X} = dv_1 \wedge dx/y$ in the chart $v_2=z=1$ in Equation~(\ref{eqn:X}) to the holomorphic two-form $dz_1\wedge dz_3/y$ in the chart $z_2=1$ in Equation~(\ref{eqn:X_general}).  The discriminant function of the fibration is given by $v_2^8 v_1^2 (v_1 - v_2)^2  ( v_1^2 +  2 \rho_1 v_1v_2 + \rho^2_2 v_2^2 )^6$.  We have the immediate:
\begin{lemma}
\label{lem:fib_X}
Equation~(\ref{eqn:X}) defines a Jacobian elliptic fibration with the singular fibers $I_2^* + 2 I_2 + 2 I_0^*$ and a Mordell-Weil group of sections $(\mathbb{Z}/2\mathbb{Z})^2$.
\end{lemma}
\begin{remark}
The pencil of lines through the point $[z_1:z_2:z_3]=[1:0:1]$ induces a second Jacobian elliptic fibration on $\mathcal{X}$ with the singular fibers $I_4^* + 4 I_2 +  I_0^*$ and a Mordell-Weil group of sections $(\mathbb{Z}/2\mathbb{Z})^2$.
\end{remark}
\begin{remark}
\label{rem:Hoyt}
Hoyt proved in \cite{MR1013162} that there are countably many algebraic curves $W_N \subset \mathbb{C}^2$ in the parameter space $\{ (A,B) \, | \, A \not = B \;  \text{and} \; A, B \not = 0, 1 \}$ for Equation~(\ref{eqn:Xrestriction}) such that each $W_N$ corresponds to a modular correspondence between two elliptic curves. Moreover, it was shown that the rank of the Mordell-Weil group in Lemma~\ref{lem:fib_X} is greater than zero if and only if $(A,B) \in \bigcup_N W_N$. We will provide the modular correspondences for the subsets of the parameter space in Proposition~\ref{lem:branch_locus} in Section~\ref{ssec:subfamilies}.
\end{remark}
\begin{remark}
The transcendental lattice in case (1) of Proposition~\ref{lem:branch_locus} is isomorphic to $\langle 2 \rangle \oplus \langle 2 \rangle \oplus \langle -2 \rangle$. For the remaining subfamilies in Proposition~\ref{lem:branch_locus} the transcendental lattices can be determined using the results of \cite{MR1877757}*{Sec.~3}.
\end{remark}
\subsection{Double coverings by Kummer surfaces}
We recall that an \emph{even eight} on a K3 surface $\mathcal{X}$ is a set of eight disjoint $(-2)$-rational smooth curves $E_1, \dots, E_8$ such that there is a divisor $D \in \operatorname{NS}(\mathcal{X})$ with
\beq  
 E_1 + \dots + E_8 \sim 2 D \,,
\eeq 
where the symbol $\sim$  denotes linear equivalence. Choosing an even eight as the branch locus on a K3 surface for a two-fold ramified covering, \emph{Nikulin's construction} yields a new K3 surface $\mathcal{Y}$ together with a double covering $f\colon \mathcal{Y} \to \mathcal{X}$; see \cite{MR0429917}. Equivalently, there is a Nikulin involution on $\mathcal{Y}$ such that the minimal resolution of the quotient surface of $\mathcal{Y}$ by the involution yields $\mathcal{X}$. A \emph{Nikulin involution} is an analytic involution of $\mathcal{Y}$ which acts trivially on $H^{2,0}(\mathcal{Y})$. It is well known that such an involution always has eight fixed points and that by blowing them up and taking the quotient we recover the K3 surface $\mathcal{X}$ \cite{MR728142}.
\par In our situation, reducible fibers of type $D_6$ or $D_4$ arise in Lemma~\ref{lem:fib_X} as minimal resolutions of the fibers  located over $v_2=0$ and $v_1^2 + 2 \rho_1 v_1v_2 + \rho^2_2 v_2^2=0$, respectively.  Recall that a $D_4$ fiber has one (central) component of multiplicity two and four reduced component, taking the sum of two such fibers one finds that the eight reduced components are twice the class of a fiber minus the two double components and this gives the divisibility by two. Thus, configurations of eight non-central components of these fibers then form even eights on the minimal resolution of $\mathcal{X}$.  Since the Mordell-Weil group of sections in  Lemma~\ref{lem:fib_X} is $(\mathbb{Z}/2\mathbb{Z})^2$, we can always choose the four marked components in the reducible fibers in such a way that each of them is met by the zero section or a two-torsion section. In our situation, there are then the following even eights to consider:
\begin{enumerate}
\item[(1)] non-central components of $D_6$ and $D_4$ over $v_2=0$ and $v_1 = (- \sqrt{\rho_1^2-\rho_2^2} + \rho_1)v_2$ meeting sections,
\item[(2)] non-central components of $D_6$ and $D_4$ over $v_2=0$ and $v_1 = (+ \sqrt{\rho_1^2-\rho_2^2} + \rho_1)v_2$ meeting sections,
\item[(3)] non-central components of $D_4$'s over $v_1^2 + 2 \rho_1 v_1v_2 + \rho^2_2 v_2^2=0$.
\end{enumerate}
Choosing any of these even eights as the branch locus of a double cover, Nikulin's construction yields a new K3 surface obtained as the minimal resolution of a Jacobian elliptic surface $\mathcal{Y}_n$, together with a rational double cover $f_n: \mathcal{Y}_n \dasharrow  \mathcal{X}$ for $n = 1, \dots,3$. In the case of two $D_4$ fibers,
the double cover is induced by the double cover of the base curve $\mathbb{P}^1$ branched over the two corresponding base points. The pull-back of the elliptic fibration, after a normalization, gives the elliptic K3 surface $\mathcal{Y}_n$ with smooth fibers over the ramification points. Moreover, the corresponding Nikulin involution on $\mathcal{Y}_n$ is the lift of the covering involution on the base composed with fiberwise multiplication by $-1$. This involution has eight fixed points, four in each fiber over the fixed points on the base, and with a quotient that is the K3 surface $\mathcal{X}$ one started with. Similar results hold in the other cases. Garbagnati considered Jacobian elliptic K3 surfaces with abelian and dihedral groups of symplectic automorphisms, which includes an isogeny $f_n: \mathcal{Y}_n \dasharrow  \mathcal{X}$ where $\mathcal{X}$ contains two reducible $D_4$ fibers as a special case \cite{MR3011784}*{Sec.~5-6}. Closely related is also work by Mehran \cite{MR2306633} who classified all the K3 surfaces that admit a rational map of degree two into a given Kummer surface.
\medskip
\subsubsection{Double coverings by Kummer surfaces associated with two elliptic curves}
Let us first construct the Jacobian elliptic surface $\mathcal{Y}_3$ and the double cover $f_3: \mathcal{Y}_3 \rightarrow \mathcal{X}$. A double cover between two rational curves, defined over  $\mathbb{Q}(\rho_1, \rho_2)$,  is given by
\beq
\label{eqn:double_cover}
 \mathbb{P}^1 \ \dasharrow \ \mathbb{P}^1 \,, \qquad [u_1 : u_2] \ \mapsto \  [v_1 : v_2] = [  (\rho_1 - \rho_2) u_1^2 - 2 \rho_1 u_1 u_2 +  (\rho_1 + \rho_2) u_2^2 : \ 2 u_1 u_2] \,.
\eeq
The pullback of Equation~(\ref{eqn:X}) along the double cover in Equation~(\ref{eqn:double_cover}) yields an elliptic fibration, with fibers over $\mathbb{P}^1=\mathbb{P}(u_1,u_2)$ embedded into $\mathbb{P}^2 = \mathbb{P}(X, Y, Z)$, given by the Weierstrass model
\beq
\label{eqn:Y}
\begin{split}
 \mathcal{Y}_3: \quad Y^2 Z &= X \, \left( X + \frac{1}{2} u_1 u_2 \big(u_1 - u_2 \big) \big( (\rho_1 -\rho_2) u_1 - (\rho_1 + \rho_2) u_2 \big) Z\right) \\
 & \times   \left( X + \frac{1}{2} u_1 u_2  \big( (\rho_1 - \rho_2) u_1^2 -2  (\rho_1+ 1) u_1 u_2 + (\rho_1+\rho_2) u_2^2 \big) Z \right)\,.
\end{split} 
\eeq
In fact, Equation~(\ref{eqn:Y}) is obtained by setting
\beq
\label{eqn:f_XYZ}
\begin{split}
 x = \big( (\rho_1-\rho_2) u_1^2 - (\rho_1+\rho_2) u_2^2 \big)^2 \Big(4 X  + 2 u_1 u_2 \big(u_1-u_2\big) \big( (\rho_1-\rho_2)u_1 - (\rho_1+\rho_2) u_2\big)Z\Big)\,, \\
 y = 8 \big( (\rho_1-\rho_2) u_1^2 - (\rho_1+\rho_2) u_2^2 \big)^3 Y \,, \qquad z=Z \,,
 \end{split}
\eeq
in Equation~(\ref{eqn:X}). Together, Equations~(\ref{eqn:f_XYZ}) and~(\ref{eqn:double_cover}) determine the double cover
\beq
 f_3: \quad \mathcal{Y}_3 \ \dasharrow \ \mathcal{X}\,, \qquad \Big( [u_1 : u_2], [X:Y:Z] \Big) \mapsto \Big( [v_1 : v_2], [x:y:z] \Big)  \,,
\eeq
which is branched along the even eight on $\mathcal{X}$ that consists of the non-central components of the two reducible fibers of type $D_4$ obtained as the minimal resolution of the fibers of Kodaira-type $I_0^*$ located over $ v_1^2 +  2 \rho_1 v_1v_2 + \rho^2_2 v_2^2=0$. Conversely, on $\mathcal{Y}_3$ in the chart $Z=1$ and $u_2=1$ a compatible Nikulin involution is given by
\beq
 \Big( u_1, \; X, \; Y \Big) \ \mapsto \ \left( u'_1= \frac{\rho_1 + \rho_2}{(\rho_1-\rho_2) u_1}, \; X' = \frac{(\rho_1+\rho_2)^2 X}{(\rho_1-\rho_2)^2 u_1^4}, \; Y'= - \frac{(\rho_1 + \rho_2)^3 Y}{(\rho_1-\rho_2)^3 u_1^6} \right) \,,
\eeq
and it leaves the holomorphic two-forms $\omega_{\mathcal{Y}_3} = du_1 \wedge dX/Y$ invariant. We have the following:
\begin{lemma}
\label{lem:Z_Y}
For the holomorphic two-form $\omega_\mathcal{X} = dw_1 \wedge dx/y$ in the chart $z=w_2=1$ on $\mathcal{X}$ in Equation~(\ref{eqn:X}) and the holomorphic two-form $\omega_{\mathcal{Y}_3} = du_1 \wedge dX/Y$ in the chart $Z=u_2=1$ on $\mathcal{Y}_3$ in Equation~(\ref{eqn:Y}), we have $\omega_{\mathcal{Y}_3} = f_3^* \omega_\mathcal{X}$.
\end{lemma}
\begin{proof}
The proof follows by an explicit computation using the transformation in Equations~(\ref{eqn:f_XYZ}) and~(\ref{eqn:double_cover}).
\end{proof}
It turns out that the minimal resolution of $\mathcal{Y}_3$ is isomorphic to the Kummer surface $\operatorname{Kum}(\mathcal{E}_1\times \mathcal{E}_2)$ associated with the product surface of two non-isogenous elliptic curves. We have the following:
\begin{proposition}
\label{prop:Z_Y}
For  $\Lambda_1, \Lambda_2  \in \mathbb{P}^1 \backslash\{ 0, 1, \infty\}$ let parameters $\rho_1, \rho_2$ be given by
\beq
\label{eqn:paramsY}
 \rho_1 = \Lambda_1 + \Lambda_2 - 2 \Lambda_1 \Lambda_2 -1\,, \qquad \rho_2 =  1 - \Lambda_1 - \Lambda_2\,.
\eeq
The surfaces $\mathcal{Y}_3$ in Equation~(\ref{eqn:Y}) and $\mathcal{Z}$ in Equation~(\ref{eqn:Kummer44}) are birational equivalent over $\mathbb{Q}(\Lambda_1, \Lambda_2)$. The birational equivalence identifies the holomorphic two-form $\omega_{\mathcal{Y}_3} = du_1 \wedge dX/Y$ in the chart $Z=u_2=1$ with $\omega_\mathcal{Z} = dx_1 \wedge dx_2/y_{1,2}$ in the chart $z_1=z_2=1$.
\end{proposition}
\begin{proof}
Equation~(\ref{eqn:Y}) is then precisely Equation~(\ref{eqn:Y0}), describing a Jacobian elliptic fibration on the double quadric surface $\mathcal{Z}$ whose minimal resolution is $\operatorname{Kum}(\mathcal{E}_1\times \mathcal{E}_2)$. The proof then follows from Proposition~\ref{prop:equivalence00} and Corollary~\ref{lem:2forms00}. 
\end{proof}
It turns out that the identification of the elliptic fibration in Equation~(\ref{eqn:X}) and the isogeny $f_3: \mathcal{Y}_3 \rightarrow \mathcal{X}$ are key in generalizing the results of \cite{MR2214473}; see Section~\ref{ssec:subfamilies}.
\subsubsection{Double coverings by Jacobian Kummer surfaces}
Next we construct the K3 surface  $\mathcal{Y}_1$ and the double cover $f_1: \mathcal{Y}_1 \rightarrow \mathcal{X}$. We introduce new parameters $\mu_2, \mu_3$ such that in Equation~(\ref{eqn:X}) we have
\beq
\label{eqn:data1}
\rho_1 = \frac{(\mu_3-1)^2(\mu_2^2+6 \mu_2+1)}{8 (\mu_2\mu_3-1)(\mu_2-\mu_3)} \,, \qquad
\rho_2^2 = \frac{(\mu_3-1)^4(\mu_2+1)^2 \mu_2}{4 (\mu_2\mu_3-1)^2(\mu_2-\mu_3)^2} \,.
\eeq
Equations~(\ref{eqn:data1}) are sufficient  to obtain a family of Weierstrass models from Equation~(\ref{eqn:X}) over $\mathbb{Q}(\mu_2, \mu_3)$ since the latter family only depends on $\rho_2^2$. 
\par A double cover $\mathbb{P}^1 \rightarrow  \mathbb{P}^1$ between two rational curves, defined over  $\mathbb{Q}(\mu_2, \mu_3)$,  is given by $[u_1 : u_2]  \mapsto \ [v_1 : v_2]$ with 
\beq
\label{eqn:double_cover1}
 [v_1 : v_2] = [  \mu_2 (\mu_3-1)^2 (u_1-\mu_3 u_2)(u_1-\mu_2^2\mu_3 u_2) : (\mu_2-\mu_3)(1-\mu_2\mu_3)(u_1-\mu_2 \mu_3 u_2)^2] \,.
\eeq
The pullback of Equation~(\ref{eqn:X}) along the double cover in Equation~(\ref{eqn:double_cover1}) then yields an elliptic fibration, with fibers over $\mathbb{P}^1=\mathbb{P}(u_1,u_2)$ embedded into $\mathbb{P}^2 = \mathbb{P}(X, Y, Z)$, given by the Weierstrass model
\beq
\label{eqn:Y1}
\begin{split}
 \mathcal{Y}_1: \quad Y^2 Z &= \epsilon X \, \left( X +  \mu_2 (\mu_3- 1)^2u_1 u_2 \big(u_1 - \mu_3 u_2 \big) \big( u_1 - \mu_2^2 \mu_3 u_2 \big) Z\right) \\
  & \times  \left( X +  \mu_3 (\mu_2- 1)^2u_1 u_2 \big(u_1 - \mu_2 u_2 \big) \big( u_1 - \mu_2 \mu_3^2 u_2 \big) Z\right) \,,
\end{split}  
\eeq
with $1/\epsilon=4\mu_2 \mu_3(1-\mu_2\mu_3)(\mu_2-\mu_3)$. In fact, Equation~(\ref{eqn:Y1}) is obtained by setting
\beq
\label{eqn:f_XYZ1}
\begin{split}
 x = \frac{1}{4} \, \mu_2 \mu_3 (1-\mu_2\mu_3)(\mu_2-\mu_3)  (\mu_2-1)^4 (\mu_3-1)^4 (u_1- \mu_2\mu_3 u_2)^2 (u_2 + \mu_2 \mu_3 u_2)^2 \\
 \times \, \Big( X + \mu_2 (\mu_3-1)^2 u_1 u_2 (u_1 - \mu_3 u_2) (u_1 - \mu_2^2 \mu_3 u_2) Z \Big) \,, \\
 y = \frac{1}{4} \,  \mu_2^2 \mu_3^2 (\mu_2-\mu_3)^2 (1-\mu_2\mu_3)^2 (\mu_2-1)^6 (\mu_3-1)^6 (u_1-\mu_2\mu_3u_2)^3 (u_1+\mu_2\mu_3u_2)^3 Y \,, 
 \end{split}
\eeq
and $z=Z$ in Equation~(\ref{eqn:X}). Together, Equations~(\ref{eqn:f_XYZ1}) and~(\ref{eqn:double_cover1}) then determine a double cover
\beq
\label{eqn:f1}
 f_1: \quad \mathcal{Y}_1 \ \dasharrow \ \mathcal{X}\,, \qquad \Big( [u_1 : u_2], [X:Y:Z] \Big) \mapsto \Big( [v_1 : v_2], [x:y:z] \Big)  \,,
\eeq
which is branched along the even eight on $\mathcal{X}$ that consists of non-central components of the reducible fibers of type $D_6$ and $D_4$ meeting the sections. In terms of the parameters $\mu_2, \mu_3$, the reducible fibers are the minimal resolution of the fibers of Kodaira-type $I_2^*$ and $I_0^*$ located over 
\beq
\label{eqn:pos_sing_fibs1}
 v_2=0 \,, \quad \text{and} \quad v_1 =  \frac{(\mu_2+1) (1-\mu_3)^2}{4(1-\mu_2\mu_3)(\mu_2-\mu_3)} v_2 \,,
\eeq 
whereas the remaining  fiber of Kodaira-type $I_0^*$ is located over
\beq
\label{eqn:pos_sing_fibs2}
 v_1 =   \frac{\mu_2 (1-\mu_3)^2}{(1-\mu_2\mu_3)(\mu_2-\mu_3)} v_2 \,.
\eeq 
A compatible Nikulin involution on $\mathcal{Y}_1$ can be constructed in the same way as already demonstrated for the case of $\mathcal{Y}_3$.  We have the following:
\begin{lemma}
\label{lem:Z_Y2}
For the holomorphic two-form $\omega_\mathcal{X} = dw_1 \wedge dx/y$ in the chart $z=w_2=1$ on $\mathcal{X}$ in Equation~(\ref{eqn:X}) and the holomorphic two-form $\omega_{\mathcal{Y}_1} = du_1 \wedge dX/Y$ in the chart $Z=u_2=1$ on $\mathcal{Y}_1$ in Equation~(\ref{eqn:Y1}), we have $\omega_{\mathcal{Y}_1} = f_1^* \omega_\mathcal{X}$.
\end{lemma}
\begin{proof}
The proof follows by an explicit computation using the transformation in Equations~(\ref{eqn:f_XYZ1}) and~(\ref{eqn:double_cover1}).
\end{proof}
\begin{remark}
For the construction of the double cover $f_2: \mathcal{Y}_2 \rightarrow \mathcal{X}$ the induced map on the base curves is branched over $v_2=0$ and $v_1 = (\sqrt{\rho_1^2-\rho_2^2} + \rho_1)v_2$. We introduce new Rosenhain roots $\widetilde{\mu}_2 =\widetilde{k}_2^2$ and $\widetilde{\mu}_3 =\widetilde{k}_3^2$ such that
\beq
\label{eqn:alt}
\begin{split}
  \frac{(\mu_2+1) (1-\mu_3)^2}{4(1-\mu_2\mu_3)(\mu_2-\mu_3)} 	& =  \frac{\widetilde{\mu}_2 (1-\widetilde{\mu}_3)^2}{(1-\widetilde{\mu}_2\widetilde{\mu}_3)(\widetilde{\mu}_2-\widetilde{\mu}_3)} \,,\\
  \frac{\mu_2 (1-\mu_3)^2}{(1-\mu_2\mu_3)(\mu_2-\mu_3)} 		& =  \frac{(\widetilde{\mu}_2+1) (1-\widetilde{\mu}_3)^2}{4(1-\widetilde{\mu}_2\widetilde{\mu}_3)(\widetilde{\mu}_2-\widetilde{\mu}_3)} 
  \,.\\
\end{split}  
\eeq
With respect to the new Rosenhain roots the positions of the two fibers of Kodaira-type $I_0^*$ in Equations~(\ref{eqn:pos_sing_fibs1}) and~(\ref{eqn:pos_sing_fibs2}) are interchanged, whereas the equations for the parameters $\rho_1, \rho_2^2$ remain unchanged, i.e.,
\beq
\rho_1 = \frac{(\widetilde{\mu}_3-1)^2(\widetilde{\mu}_2^2+6 \widetilde{\mu}_2+1)}{8 (\widetilde{\mu}_2\widetilde{\mu}_3-1)(\widetilde{\mu}_2-\widetilde{\mu}_3)} \,, \qquad
\rho_2^2 = \frac{(\widetilde{\mu}_3-1)^4(\widetilde{\mu}_2+1)^2 \widetilde{\mu}_2}{4 (\widetilde{\mu}_2\widetilde{\mu}_3-1)^2(\widetilde{\mu}_2-\widetilde{\mu}_3)^2} \,.
\eeq
The solutions of Equation~(\ref{eqn:alt}) are given by
\beq
\label{eqn:change}
 \mu_2 = - \frac{(i \pm \widetilde{k}_2)^2}{(i \mp \widetilde{k}_2)^2} \,, \qquad  \mu_3 = - \frac{(i \pm \widetilde{k}_3)^2}{(i \mp \widetilde{k}_3)^2}\,,
\eeq 
where the sign choices in the formulas for $\mu_2$ and $\mu_3$ are independent. Thus, all results related to $f_2: \mathcal{Y}_2 \rightarrow \mathcal{X}$ can be derived from those for $f_1$  by replacing the moduli according to Equation~(\ref{eqn:change}).
\end{remark}
\par Let us also consider the double sextic surface given by
\beq
\label{kummer_middle_twist2}
  \mathcal{W}^{\prime\,  (\varepsilon^\prime)}_0: \quad  \hat{z}_4^2 = \varepsilon^\prime  z_1 z_3 \big(  z_1  - z_2 +  z_3 \big) \big( (\mu_2\mu_3)^2 z_1  -  \mu_2\mu_3 \, z_2 +  z_3 \big)  
   \prod_{i=2}^3 \big( \mu_i^2 \, z_1  -  \mu_i \, z_2 +  z_3 \big) \;.
\eeq
It follows from Proposition~\ref{prop:Kummers} that $\mathcal{W}^{\prime\,  (\varepsilon^\prime)}_0$ is birational equivalent to the twisted Kummer surface $\operatorname{Kum}(\operatorname{Jac}{\mathcal{C}_0^\prime})^{(\varepsilon^\prime)}$ associated with the Jacobian of the smooth genus-two curve given by
\beq
\label{Eq:Rosenhain_special_dual_22}
  \mathcal{C}_0^\prime: \quad Y^2  =  X Z \big(X -Z\big) \big(X - \mu_2\mu_3 Z\big)  \big(X^2 - (\mu_2+\mu_3)X Z+  \mu_2\mu_3 Z^2\big)  \,.
\eeq
It turns out that the minimal resolution of $\mathcal{Y}_1$ is isomorphic to the Kummer surface $\operatorname{Kum}(\operatorname{Jac}{\mathcal{C}_0^\prime})$.  We have the following:
\begin{proposition}
\label{prop:Z_Y2}
Assume that $\mu_2, \mu_3 \in \mathbb{P}^1 \backslash \{ 0, 1, \infty\}$  satisfy $\mu_2 \not= \mu_3^{\pm1}$, and a quadric-twist factor is given by
\beq
 \varepsilon^\prime = \frac{1}{4 \mu_2 \mu_3 (1-\mu_2\mu_3) (\mu_2-\mu_3)} \,.
\eeq
The surfaces $\mathcal{Y}_1$ in Equation~(\ref{eqn:Y1}) and $\mathcal{W}^{\prime \, (\varepsilon^\prime)}_0$ in Equation~(\ref{kummer_middle_twist2}) are birational equivalent over $\mathbb{Q}(\mu_2 \mu_3, \mu_2+\mu_3)$. The birational equivalence identifies the holomorphic two-form $\omega_{\mathcal{Y}_1} = du_1 \wedge dX/Y$ in the chart $Z=u_2=1$ with $\omega_{\mathcal{W}^{\prime \, (\varepsilon^\prime)}_0} = dz_2 \wedge dz_3/\hat{z}_4$ in the chart $z_1=1$.
\end{proposition}
\begin{proof}
Equation~(\ref{eqn:Y1}) is precisely Equation~(\ref{eqn:W0twist}), with $\lambda_l$'s replaced by $\mu_l$'s for $l=2, 3$. The proof then follows from Corollary~\ref{cor:lem:2forms00_fib1}. The constraint that the birational equivalence identifies $\omega_{\mathcal{Y}_1}$ and $\omega_{\mathcal{W}^{\prime \, (\varepsilon^\prime)}_0}$ determines $\varepsilon^\prime$.
\end{proof}
\subsection{Kummer surfaces as quotients}
For the affine Weierstrass equation for $\mathcal{X}$ of the form $y^2= x^3 + p_1 x^2+p_2 x$, the translation by the two-torsion section $T: (x,y)=(0,0)$ in each fiber constitutes a Nikulin involution $\imath$, which is given by addition $\oplus$ of $T$ with respect to the group law in each smooth elliptic fiber, i.e.,
\beq
\label{VGS_involution_middle}
 \imath: \quad (x,y)  \mapsto  (x,y) \; \oplus \;   T \ =  \ \left( \frac{p_2}{x}, - \frac{p_2 y}{x^2} \right) \,.
\eeq
By resolving the eight nodes of  $\mathcal{X}/\langle \imath \rangle$ we obtain a new K3 surface as a Jacobian elliptic surface $\mathcal{X}^\vee$, given by the affine Weierstrass model 
\beq
\label{kummer_middle_ell_dual_W}
 Y^2 = X^3 - 2 \, p_1 X^2  +  \left( p_1^2 - 4 p_2\right)  X \,.
\eeq
One constructs a dual van Geemen-Sarti involution $\imath^\vee$ analogously. The explicit formulas for the isogeny and the dual isogeny are well known\footnote{They were given explicitly in \cite{MR3995925}} and given by
\beq
\label{Eq:isogeny}
\varphi: \quad \mathcal{X} \dashrightarrow \mathcal{X}^\vee\,, \qquad  (x,y) \  \mapsto \ \left( \frac{y^2}{x^2},  \frac{y \, \big(x^2-p_2 \big)}{x^2}\right)\,,
\eeq
and
\beq
\label{eqn:dual_isog}
\varphi^\vee:  \quad \mathcal{X}^\vee \dashrightarrow \mathcal{X}, \qquad  (X,Y) \ \mapsto \ \left( \frac{Y^2}{4X^2}, \frac{Y \, \big(X^2 - p_1^2 +4 \, p_2)  \big)}{8 \, X^2}\right)\,.
\eeq
We have the following:
\begin{proposition}
\label{prop2}
The K3 surfaces $\mathcal{X}$ and $\mathcal{X}^\vee$ admit dual van Geemen-Sarti involutions $\imath$ and $\imath^\vee$ associated with fiberwise translations by the two-torsion section $T: (x,y)=(0,0)$ and $T^\vee: (X,Y)=(0,0)$, respectively, and the following pair of dual geometric two-isogenies:
\begin{equation}
\label{isog_middle}
 \xymatrix 
{ \mathcal{X}^\vee \ar @(dl,ul) _{\imath^\vee}
\ar @/_0.5pc/ @{-->} _{\varphi^\vee} [rr] &
& \mathcal{X} \ar @(dr,ur) ^{\imath}
\ar @/_0.5pc/ @{-->} _{\varphi} [ll] \\
} 
\end{equation}
Moreover, it follows $\omega_\mathcal{X} = \varphi^* \omega_{\mathcal{X}^\vee}$ and $2\, \omega_{\mathcal{X}^\vee} =  (\varphi^\vee)^* \omega_\mathcal{X}$ for the holomorphic two-forms $\omega_\mathcal{X}=dv \wedge dx/y$ and $\omega_{\mathcal{X}^\vee} = dv \wedge dX/Y$  on $\mathcal{X}$ and $\mathcal{X}^\vee$, respectively.
\end{proposition}
\par Applying Proposition~\ref{prop2} to Equation~(\ref{eqn:X}) using the Nikulin involution $\imath_3$ associated with translations by the two-torsion section 
\beq
 T_3 :  \quad [X : Y : Z ]  \ = \ [v_1 v_2  \big( v_1^2 + 2 \rho_1 v_1v_2 + \rho^2_2 v_2^2 \big) :\ 0:\ -1] \,,
\eeq
we obtain a new K3 surface $\mathcal{X}^\vee_3$ as the minimal resolution of $\mathcal{X}/\langle \imath_3 \rangle$ together with a quotient map $\varphi_3: \mathcal{X} \dasharrow \mathcal{X}^\vee_3$ defined over $\mathbb{Q}(\rho_1, \rho_2^2)$. Here, a Weierstrass model for $\mathcal{X}^\vee_3$ is given by
\beq
\label{eqn:Yvee}
\begin{split}
 \mathcal{X}^\vee_3: \quad Y^2 Z &= X^3  - 2 v_2 \big(2 v_1 - v_2 \big) \Big( v_1^2 + 2 \rho_1 v_1 v_2 + \rho_2^2 v_2^2 \Big) X^2 Z \\
& + v_2^4   \Big( v_1^2 + 2 \rho_1 v_1 v_2 + \rho_2^2 v_2^2 \Big)^2 X Z^2 \,.
\end{split} 
\eeq
It turns out that the minimal resolution of $\mathcal{X}^\vee_3$ is isomorphic to the Kummer surface $\operatorname{Kum}(\mathcal{E}_1^\prime \times \mathcal{E}_2^\prime)$ associated with the product surface of the two-isogenous elliptic curves in Equation~(\ref{eqn:EC_dual}).  We have the following:
\begin{proposition}
\label{prop:Yv_Zvr}
For $\Lambda_1, \Lambda_2 \in \mathbb{P}^1 \backslash \{ 0, 1, \infty\}$ let parameters $\rho_1, \rho_2$ satisfy
\beq
\label{eqn:params_2}
 \rho_1 = \Lambda_1 + \Lambda_2 - 2\Lambda_1 \Lambda_2 -1\,, \qquad \rho_2^2 = (1-\Lambda_1 - \Lambda_2)^2\,.
\eeq
The surfaces $\mathcal{X}^\vee_3$ in Equation~(\ref{eqn:Yvee}) and $\mathcal{Z}^\prime$ in Equation~(\ref{eqn:Kummer44_dual}) are birationally equivalent  over $\mathbb{Q}(\kappa_1, \kappa_2)$ with $\Lambda_l=1/(1 -\kappa_l^2)$ for $l= 1,2$. The birational equivalence identifies the holomorphic two-form $\omega_{\mathcal{X}_3^\vee} = dv_1 \wedge dX/Y$ in the chart $Z=v_2=1$ with $\omega_{\mathcal{Z}^\prime} = dX_1 \wedge dX_2/Y_{1,2}$ in the chart $Z_1=Z_2=1$.
\end{proposition}
\begin{proof}
Equation~(\ref{eqn:Yvee}) is precisely Equation~(\ref{eqn:Yvee0}), describing a Jacobian elliptic fibration on the Kummer surface $\mathcal{Z}^\prime$ whose minimal resolution is $\operatorname{Kum}(\mathcal{E}_1^\prime \times \mathcal{E}_2^\prime)$. The proof then follows from Proposition~\ref{prop:equivalence00} and Corollary~\ref{lem:2forms00}.
\end{proof}
\par Applying Proposition~\ref{prop2} to Equation~(\ref{eqn:X}) using the Nikulin involution $\imath_1$ associated with translations by the two-torsion section $T_1: [X:Y:Z]=[0:0:1]$, we obtain a new K3 surface $\mathcal{X}^\vee_1$ as the minimal resolution of $\mathcal{X}/\langle \imath_1 \rangle$ together with a rational quotient map $\varphi_1: \mathcal{X} \dasharrow \mathcal{X}^\vee_1$ defined over $\mathbb{Q}(\rho_1, \rho_2^2)$.  Here, a Weierstrass model for $\mathcal{X}^\vee_1$, with fibers over $\mathbb{P}^1=\mathbb{P}(v_1,v_2)$ embedded into $\mathbb{P}^2 = \mathbb{P}(X,Y,Z)$, is 
\beq
\label{eqn:V}
\begin{split}
 \mathcal{X}^\vee_1: \quad Y^2 Z &= X^3  + 2 v_2 \big(v_1 + v_2 \big) \Big( v_1^2 + 2 \rho_1 v_1 v_2 + \rho_2^2 v_2^2 \Big) X^2 Z \\
&+  v_2^2 \big(v_1 - v_2 \big)^2  \Big( v_1^2 + 2 \rho_1 v_1 v_2 + \rho_2^2 v_2^3 \Big)^2 X Z^2 \,.
\end{split} 
\eeq
The discriminant function of the fibration is $16 v_1 v_2^7 (v_1 - v_2)^4  ( v_1^2 +  2 \rho_1 v_1v_2 + \rho^2_2 v_2^2 )^6$. 
\begin{remark}
Applying Proposition~\ref{prop2} using the (different) two-torsion section of the elliptic fibration~(\ref{eqn:X}) of $\mathcal{X}$, namely
\beq
 T_2 :  \quad  [X : Y : Z ]  \ = \ [v_2^2  \big( v_1^2 +  2\rho_1 v_1v_2 + \rho^2_2 v_2^2 \big) :\ 0:\ -1] \,,
\eeq
yields a K3 surface with the same singular fibers as $\mathcal{X}^\vee_2$, but for different moduli. In fact, the Weierstrass equation coincides with Equation~(\ref{eqn:V}) if we replace
\beq
\label{eqn:change_in_one}
 \Big( \rho_1 , \rho_2^2 \Big) \  \mapsto  \ \Big( -\rho_1 -1, \ \rho_2^2 + 2 \rho_1 +1 \Big) \,.
\eeq
In terms of Lemma~\ref{lem:fib_X} the change of parameters is equivalent to the interchange of the two fibers of Kodaira-type $I_2$ while keeping the fiber of Kodaira-type $I_2^*$ fixed.
 \end{remark}
 \par It turns out that the minimal resolution of $\mathcal{X}^\vee_1$ in Equation~(\ref{eqn:V}) is isomorphic to the Kummer surface $\operatorname{Kum}(\operatorname{Jac}{\mathcal{C}_0})$ associated with the Jacobian of a smooth genus-two curve admitting an elliptic involution.  We have the following:
\begin{proposition}
\label{prop:equivalence2}
Assume that $\lambda_2, \lambda_3 \in \mathbb{P}^1 \backslash  \{ 0, 1, \infty\}$  satisfy $\lambda_2 \not= \lambda_3^{\pm1}$. Further assume that parameters $\rho_1, \rho_2$ satisfy
\beq
\label{eqn:paramsZ1}
 \rho_1 = - \frac{\lambda_2^2\lambda_3^2+\lambda_2^2+\lambda_3^2-4\lambda_2\lambda_3 + 1}{(\lambda_2-1)^2(\lambda_3-2)^2} \,, \qquad \rho_2^2 = \frac{(\lambda_2+1)^2(\lambda_3+1)^2}{(\lambda_2-1)^2(\lambda_3-1)^2} \,,
\eeq
and a  quadric-twist factor is given by
\beq
 \varepsilon =  \frac{4}{\lambda_2 \lambda_3 (1-\lambda_2)^2(1-\lambda_3)^2} \,.
\eeq
The surfaces $\mathcal{X}^\vee_1$ in Equation~(\ref{eqn:V}) and $\mathcal{W}^{(\varepsilon)}_0$ in Equation~(\ref{kummer_middle_twist}) are birational equivalent over $\mathbb{Q}(\lambda_2 + \lambda_3, \lambda_2\lambda_3)$. The rational equivalence identifies the two-forms $\omega_{\mathcal{X}_1^\vee} = dv_1 \wedge dX/Y$ in the chart $Z=v_2=1$ and $\omega_{\mathcal{W}^{(\varepsilon)}_0} = dz_2 \wedge dz_3/\hat{z}_4$ in the chart $z_1=1$.
\end{proposition}
\begin{proof}
Since the genus-two curve $\mathcal{C}_0$ is smooth, we have $\lambda_2, \lambda_3 \in \mathbb{P}^1 \backslash  \{ 0, 1, \infty\}$ and $\lambda_2 \not= \lambda_3^{\pm1}$.  A birational equivalence  relating Equation~(\ref{eqn:W_twist}) and Equation~(\ref{eqn:V}) is given by
\beqn
\begin{split}
w_1 = - \lambda_2 \lambda_3 \big(v_1 + v_2\big) \,, \quad w_2 =  \frac{1}{2} \big(v_1 - v_2\big) \,,\\
x  = \frac{\varepsilon}{4} (\lambda_2\lambda_3)^3 \big( \lambda_2\lambda_3 - \lambda_2 - \lambda_3+1)^2 X \,, \; \;
y =  \frac{\varepsilon^\frac{3}{2}}{8} (\lambda_2\lambda_2)^\frac{9}{2} \big( \lambda_2\lambda_3 - \lambda_2 - \lambda_3+1)^3 Y \,, \; 
z = Z \,.
\end{split}
\eeqn
We observe that for $\varepsilon=\varepsilon_1^2/( \lambda_2\lambda_3)$ the above transformation is well defined over $\mathbb{Q}(\lambda_2 + \lambda_3, \lambda_2\lambda_3)$. The remainder of the proof follows from Corollary~\ref{cor:lem:2forms00_fib2} and by an explicit computation using the transformations provided above. The constraint that the birational equivalence identifies $\omega_{\mathcal{X}_1^\vee}$ and $\omega_{\mathcal{W}^{(\varepsilon)}_0}$ determines $\varepsilon$.
\end{proof}
\subsection{Kummer sandwich theorems}
In this section we combine the results of the previous sections to obtain several \emph{Kummer sandwich theorems}, a term introduced by Shioda \cite{MR2279280}, that relate the Legendre pencil to various Kummer surfaces. In \cites{MR3992148, MR4099481} it was demonstrated that the number of points rational over a finite field on the K3 surfaces in a sandwich are the same.
\par The first Kummer sandwich relates the twisted Legendre pencil $\mathcal{X}$ to two Kummer surfaces  associated with the product of two elliptic curves. The first one is the Kummer surface $\operatorname{Kum}\left(\mathcal{E}_1\times \mathcal{E}_2\right)$ associated with the product abelian surface  $\mathcal{E}_1\times \mathcal{E}_2$, where the elliptic curves $\mathcal{E}_l$ with modular parameter $\Lambda_l$ is given by Equations~(\ref{eqn:EC}).  The second  is the Kummer surface $\operatorname{Kum}\left(\mathcal{E}_1^\prime \times \mathcal{E}_2^\prime \right)$ where $\mathcal{E}^\prime_{l}$ for $l= 1,2$ are the two-isogenous elliptic curves in Equation~(\ref{eqn:EC_dual}). Models for the Kummer surfaces are the double quadric surfaces in Equation~(\ref{eqn:Kummer44}) and Equation~(\ref{eqn:Kummer44_dual}), respectively. We have the following:
\begin{proposition}
\label{thm:maps_EC}
For $\Lambda_1, \Lambda_2 \in \mathbb{P}^1 \backslash \{ 0, 1, \infty\}$ let parameters $\rho_1, \rho_2$ satisfy
\beq
\label{eqn:params_2_thm}
 \rho_1 = \Lambda_1 + \Lambda_2 - 2\Lambda_1 \Lambda_2 -1\,, \qquad \rho_2^2 = (1-\Lambda_1 - \Lambda_2)^2\,.
\eeq
The twisted Legendre pencil $\mathcal{X}$ in Equation~(\ref{eqn:X_general}) fits into the following Kummer sandwich of rational maps:
\beq
\label{eqn:diag_statement}
	 \operatorname{Kum}(\mathcal{E}_1\times \mathcal{E}_2)  
	 \ \longrightarrow \ 
	 \mathcal{X}  
	 \ \longrightarrow \
	 \operatorname{Kum}(\mathcal{E}^\prime_1\times \mathcal{E}^\prime_2) 
	\ \longrightarrow \
	 \operatorname{Kum}(\mathcal{E}_1\times \mathcal{E}_2)^{(2^4)}  .
\eeq
The  maps are defined on families over $\mathbb{Q}(\kappa_1, \kappa_2)$ with $\Lambda_l=1/(1 -\kappa_l^2)$ for $l= 1,2$. The holomorphic two-form $\omega_\mathcal{X}$ equals the pullback of $\omega_{\operatorname{Kum}(\mathcal{E}^\prime_1\times \mathcal{E}^\prime_2)}$, resp.~its pullback equals $\omega_{\operatorname{Kum}(\mathcal{E}_1\times \mathcal{E}_2)}$ and the pullback of $\omega_{\operatorname{Kum}(\mathcal{E}_1\times \mathcal{E}_2)^{2^4}}$ equals $\omega_{\operatorname{Kum}(\mathcal{E}^\prime_1\times \mathcal{E}^\prime_2)}$.
\end{proposition}
\begin{proof}
We relate the twisted Legendre pencil $\mathcal{X}$ to the double quadric surfaces $\mathcal{Z}$ and $\mathcal{Z}^\prime$ in Equation~(\ref{eqn:Kummer44}) and~(\ref{eqn:Kummer44_dual}), respectively, by relating it to the Jacobian elliptic K3 surfaces $\mathcal{Y}_3$ and $\mathcal{X}^\vee_3$ in Equations~(\ref{eqn:Y}) and~(\ref{eqn:Yvee}) first. As a reminder, the minimal resolution of $\mathcal{Z}$ is the Kummer surface $\operatorname{Kum}\left(\mathcal{E}_1\times \mathcal{E}_2\right)$. Similarly, the minimal resolution of $\mathcal{Z}^\prime$ is the Kummer surface $\operatorname{Kum}\left(\mathcal{E}_1^\prime \times \mathcal{E}_2^\prime \right)$.
\par First we assume $\rho_2 = 1-\Lambda_1 - \Lambda_2$.  We combine Lemma~\ref{lem:Z_Y}, Proposition~\ref{prop:Z_Y}, Proposition~\ref{prop2}, Proposition~\ref{prop:Yv_Zvr}, Lemma~\ref{lem:map_phi} to obtain the following diagram of rational maps:
\beq
\label{eqn:diag}
\begin{array}{lcccrcc}
	\mathcal{Z} \  \cong \   \mathcal{Y}_3  &  
	\overset{f_3}{\longrightarrow}  &
	\mathcal{X}  &
	\overset{\varphi_3}{\longrightarrow} &
	\mathcal{X}^\vee_3 \ \cong \ \mathcal{Z}^\prime  & 
	\overset{\chi}{\longrightarrow}  &
	 \mathcal{Z}^{(2^4)}  
	 \\
	  \ \rotatebox{90}{$\simeq$} &&&&  \rotatebox{90}{$\simeq$}\ \ &&  \rotatebox{90}{$\simeq$} \ \quad
	 \\
	 \operatorname{Kum}(\mathcal{E}_1\times \mathcal{E}_2)  &
	 \longrightarrow &
	 \mathcal{X}  &
	 \longrightarrow &
	 \operatorname{Kum}(\mathcal{E}^\prime_1\times \mathcal{E}^\prime_2) &
	 \longrightarrow &
	 \operatorname{Kum}(\mathcal{E}_1\times \mathcal{E}_2)^{(2^4)}  
 \end{array}   
\eeq
By construction, all  maps and equivalences are defined over $\mathbb{Q}(\Lambda_1, \Lambda_2)$, except the equivalence $\mathcal{X}^\vee_3 \cong \mathcal{Z}^\prime$ which is defined only over the field extension $\mathbb{Q}(\kappa_1, \kappa_2)$. The holomorphic two-forms are related by pullback as follows:
\beq
\omega_\mathcal{Z} = \omega_{\mathcal{Y}_3} \,, \quad
\omega_{\mathcal{Y}_3} = f_3^* \omega_\mathcal{X} \,, \quad
\omega_\mathcal{X} = \varphi_3^* \omega_{\mathcal{X}_3^\vee} \,, \quad
\omega_{\mathcal{X}_3^\vee}= \omega_{\mathcal{Z}^\prime} \,, \quad
\omega_{\mathcal{Z}^\vee} = \chi^* \omega_{\mathcal{Z}^{(2^4)}} \,.
\eeq 
Using Equation~(\ref{eqn:two-form_E1E2}) the statement regarding the two-forms follows. 
\par Finally, the modular transformation $(\Lambda_1, \Lambda_2) \mapsto (1-\Lambda_1, 1-\Lambda_2)$ leaves the j-functions of $\mathcal{E}_1$ and $\mathcal{E}_2$ invariant and induces a sign change $(\rho_1, \rho_2) \mapsto (\rho_1, -\rho_2)$. Thus, the same statements hold for $(\rho_1, -\rho_2)$.
\end{proof}
\par The second Kummer sandwich relates the twisted Legendre pencil $\mathcal{X}$ to certain Jacobian Kummer surfaces. The first is the Kummer surface $\operatorname{Kum}( \operatorname{Jac}{\mathcal{C}_0} )$ associated with the Jacobian of the smooth genus-two curve $\mathcal{C}_0$ admitting an elliptic involution in Equation~(\ref{Eq:Rosenhain_special}) (with Rosenhain roots $\lambda_2, \lambda_3$).  The second is the Kummer surface  $\operatorname{Kum}(\operatorname{Jac}{\mathcal{C}_0^\prime})$ where $\mathcal{C}_0^\prime$ is the $(2,2)$-isogenous Richelot-curve in Equation~(\ref{Eq:Rosenhain_special_dual}) (with Rosenhain roots $\mu_2, \mu_3$).  That is, we have $\operatorname{Jac}{(\mathcal{C}_0)} \cong \operatorname{Jac}{(\mathcal{C}_0^\prime)}/G$ for a certain G\"opel group determined in Proposition~\ref{prop:Richelot-genus-two}. Models for the Kummer surfaces are the double sextic surfaces in Equation~(\ref{kummer_middle_twist}) and~(\ref{kummer_middle_twist2}), respectively. We have the following:
\begin{proposition}
\label{thm:maps_Jac}
The twisted Legendre pencil $\mathcal{X}$ in Equation~(\ref{eqn:X_general}) fits into the following Kummer sandwich of rational maps:
\beq
\label{eqn:diag_statement2}
	\operatorname{Kum}(\operatorname{Jac}{\mathcal{C}_0^\prime})^{(\varepsilon^\prime)}
	 \ \longrightarrow \ 
	 \mathcal{X}  
	 \ \longrightarrow \
	 \operatorname{Kum}( \operatorname{Jac}{\mathcal{C}_0} )^{(\varepsilon)}
	\ \longrightarrow \
	 \operatorname{Kum}(\mathcal{E}_1\times \mathcal{E}_2 )^{(2^4)}  .
\eeq
The maps are defined on families over a finite field extension of $\mathbb{Q}(\lambda_2, \lambda_3)$. The holomorphic two-form $\omega_\mathcal{X}$ equals the pullback of $\omega_{\operatorname{Kum}( \operatorname{Jac}{\mathcal{C}_0} )^{(\varepsilon)}}$, resp.~its pullback equals $\omega_{\operatorname{Kum}( \operatorname{Jac}{\mathcal{C}^\prime_0} )^{(\varepsilon^\prime)}}$ and the pullback of $\omega_{\operatorname{Kum}(\mathcal{E}_1\times \mathcal{E}_2)^{(2^4)}}$ equals  $\omega_{\operatorname{Kum}( \operatorname{Jac}{\mathcal{C}_0} )^{(\varepsilon)}}$.
\par Here, $\lambda_2, \lambda_3 \in \mathbb{P}^1 \backslash  \{ 0, 1, \infty\}$  satisfy $\lambda_2 \not= \lambda_3^{\pm1}$ and $\lambda_2 \not = -1$. The parameters $\mu_2, \mu_3$ are given by
\beq
\label{eqn:relats_genus2}
\begin{split}
\mu_2 & = \frac{\lambda_2\lambda_3+\lambda_2-\lambda_3-1+ 2 d}{\lambda_2\lambda_3+\lambda_2-\lambda_3-1- 2 d} \,,\\
\mu_3 & = \frac{\big(\lambda_2\lambda_3+\lambda_2-\lambda_3-1+ 2 d\big) \big((1+\lambda_2)^2 k_3 + (1-\lambda_2)d  -2 (1+\lambda_3) \lambda_2\big)}
{\big(\lambda_2\lambda_3+\lambda_2-\lambda_3-1- 2 d\big) \big((1+\lambda_2)^2 k_3 - (1-\lambda_2)d  -2 (1+\lambda_3) \lambda_2\big)} \,,
\end{split}  
\eeq
with $k_3^2 = \lambda_3$ and $d^2=(\lambda_2-\lambda_3)(\lambda_2\lambda_3-1)$, and  parameters $(\rho_1, \rho_2)$ satisfy
\beq
\label{eqn:paramsZ1_2}
 \rho_1 = - \frac{\lambda_2^2\lambda_3^2+\lambda_2^2+\lambda_3^2-4\lambda_2\lambda_3 + 1}{(\lambda_2-1)^2(\lambda_3-2)^2} \,, \qquad \rho_2^2 = \frac{(\lambda_2+1)^2(\lambda_3+1)^2}{(\lambda_2-1)^2(\lambda_3-1)^2} \,,
\eeq
and quadratic-twist factors are given by
\beq
\label{eqn:paramsZ1_2b}
 \varepsilon =  \frac{4}{\lambda_2 \lambda_3 (1-\lambda_2)^2(1-\lambda_3)^2} \,, \qquad 
 \varepsilon^\prime =  \frac{1}{4 \mu_2 \mu_3 (1-\mu_2\mu_3) (\mu_2-\mu_3)} \,.
\eeq
\end{proposition}
\begin{proof}
We relate the twisted Legendre pencil $\mathcal{X}$ to the (twisted) double sextic surfaces $\mathcal{W}^{(\varepsilon)}_0$ and $\mathcal{W}^{\prime \, (\varepsilon^\prime)}_0$ in Equation~(\ref{kummer_middle_twist}) and~(\ref{kummer_middle_twist2}), respectively, by relating it to the Jacobian elliptic K3 surface $\mathcal{Y}_1$ and $\mathcal{X}^\vee_1$ in Equations~(\ref{eqn:Y1}) and~(\ref{eqn:V}) first. As a reminder, the minimal resolution of $\mathcal{W}^{(\varepsilon)}_0$ is isomorphic to the Kummer surface $\operatorname{Kum}( \operatorname{Jac}{\mathcal{C}_0})^{(\varepsilon)}$.  Similarly, the minimal resolution of $\mathcal{W}^{\prime \, (\varepsilon^\prime)}_0$ is isomorphic to the Kummer surface  $\operatorname{Kum}(\operatorname{Jac}{\mathcal{C}_0^\prime})^{ (\varepsilon^\prime)}$. \vspace*{0.2em} 
\par We combine Lemma~\ref{lem:Z_Y2}, Proposition~\ref{prop:Z_Y2}, Proposition~\ref{prop2}, Proposition~\ref{prop:equivalence2}, Corollary~\ref{cor:psi} to obtain the following diagram of rational maps:
\beq
\label{eqn:diag2}
\begin{array}{lcccrcc}
	\mathcal{W}^{\prime \, (\varepsilon^\prime)}_0\  \cong \ \mathcal{Y}_1  &  
	\overset{f_1}{\longrightarrow}  &
	\mathcal{X}  &
	\overset{\varphi_1}{\longrightarrow} &
	\mathcal{X}^\vee_1 \ \cong \ \mathcal{W}^{(\varepsilon)}_0  & 
	\overset{\psi}{\longrightarrow}  &
	 \mathcal{Z}^{(2^4)}  
	 \\
	 \ \rotatebox{90}{$\simeq$} &&&&  \rotatebox{90}{$\simeq$}\ \ \ &&  \rotatebox{90}{$\simeq$}\ \quad
	 \\
	 \operatorname{Kum}(\operatorname{Jac}{\mathcal{C}_0^\prime})^{(\varepsilon^\prime)}  &
	 \longrightarrow &
	 \mathcal{X}  &
	 \longrightarrow &
	 \operatorname{Kum}( \operatorname{Jac}{\mathcal{C}_0})^{(\varepsilon)}  &
	 \longrightarrow &
	 \operatorname{Kum}(\mathcal{E}_1\times \mathcal{E}_2)^{(2^4)}  
 \end{array}   
 \eeq
Comparing Equations~(\ref{eqn:data1}) and~(\ref{eqn:paramsZ1}), one obtains a relation between the Rosenhain roots $(\lambda_2, \lambda_3)$ and $(\mu_2, \mu_3)$. There are eight solutions for $(\mu_2, \mu_3)$; see Remark~\ref{rem:all_solutions}. Since we only assumed $\lambda_2 \not =-1$ the four solutions in Proposition~\ref{prop:Richelot-genus-two} -- with the sign for  $k_3$ and $d$ chosen independently -- determine a smooth genus-two curve $\mathcal{C}_0^\prime$.
\par By construction, the maps and equivalences in Equation~(\ref{eqn:diag2}) to the left of $\mathcal{X}$ are defined over $\mathbb{Q}(\mu_2, \mu_3)$. The map $\psi$ is defined over the field extension $\mathbb{Q}(\lambda_2 + \lambda_3, q)$ with $q^2=\lambda_2\lambda_3$.  The remaining maps and equivalences in Equation~(\ref{eqn:diag2}) are families defined over $\mathbb{Q}(\lambda_2, \lambda_3)$. The holomorphic two-forms are related by pullback as follows:
\beq
\omega_{\mathcal{W}^{\prime \, (\varepsilon^\prime)}_0}= \omega_{\mathcal{Y}_1} \,, \quad
\omega_{\mathcal{Y}_1} = f_1^* \omega_\mathcal{X} \,, \quad
\omega_\mathcal{X} = \varphi_1^* \omega_{\mathcal{X}_1^\vee} \,, \quad
\omega_{\mathcal{X}_1^\vee}= \omega_{\mathcal{W}^{(\varepsilon)}_0} \,, \quad
\omega_{\mathcal{W}^{(\varepsilon)}_0}  = \psi^* \omega_{\mathcal{Z}^{(2^4)}} \,.
\eeq
Using Equation~(\ref{eqn:two-form_JacC}) the statement regarding the two-forms follows. 
\end{proof}
We combine the results of Proposition~\ref{thm:maps_EC} and Proposition~\ref{thm:maps_Jac} to obtain one of the main results of this article:
\begin{theorem}
\label{thm:combined}
Assume that $\lambda_2, \lambda_3 \in \mathbb{P}^1 \backslash  \{ 0, 1, \infty\}$  satisfy $\lambda_2 \not= \lambda_3^{\pm1}$. Further assume that $\Lambda_1, \Lambda_2$ satisfy 
\beq
\label{eqn:moduli_thm}
 \Lambda_1 \Lambda_2 = \dfrac{(\lambda_2 +\lambda_3)^2 -4 \lambda_2\lambda_3}{(1-\lambda_2)^2 (1-\lambda_3)^2} \,, \qquad
 \Lambda_1 + \Lambda_2 = - \dfrac{2(\lambda_2 +\lambda_3)}{(1-\lambda_2) (1-\lambda_3)} \,.
\eeq 
and the quadratic-twist factor $\varepsilon$ is given by
\beq
\label{eqn:paramsZ1_2b_combined}
     \varepsilon =  \frac{4}{\lambda_2 \lambda_3 (1-\lambda_2)^2(1-\lambda_3)^2} \,.
\eeq
The twisted Legendre pencil, given by
\beq
\label{eqn:X_general_JacC_thm}
  \mathcal{X}: \quad y^2 = z_1 z_2(z_1-z_2) (z_1-z_3)  \left( z_3 - \left(\frac{1+\lambda_2}{1-\lambda_2}\right)^2 z_2\right)  \left( z_3 -  \left(\frac{1+\lambda_3}{1-\lambda_3}\right)^2 z_2\right)\,,
\eeq 
fits into the following Kummer sandwich of rational maps:
\beq
\label{eqn:diag_statement3}
	 \operatorname{Kum}(\mathcal{E}_1\times \mathcal{E}_2)  
	 \ \longrightarrow \ 
	 \mathcal{X}  
	 \ \longrightarrow \
	\operatorname{Kum}( \operatorname{Jac}{\mathcal{C}_0} )^{(\varepsilon)}
	\ \longrightarrow \
	 \operatorname{Kum}(\mathcal{E}_1\times \mathcal{E}_2)^{(2^4)} .
\eeq
The maps are defined on families over $\mathbb{Q}(k_2, k_3)$ with $k_l^2=\lambda_l$ for $l= 2,3$. The holomorphic two-form $\omega_\mathcal{X}$ equals the pullback of $\omega_{\operatorname{Kum}( \operatorname{Jac}{\mathcal{C}_0} )^{(\varepsilon)}}$, resp.~its pullback equals $\omega_{\operatorname{Kum}(\mathcal{E}_1\times \mathcal{E}_2)}$ and the pullback of $\omega_{\operatorname{Kum}(\mathcal{E}_1\times \mathcal{E}_2)^{(2^4)}}$ equals  $\omega_{\operatorname{Kum}( \operatorname{Jac}{\mathcal{C}_0} )^{(\varepsilon)}}$.
\end{theorem}
\begin{proof}
For the given relations in Equations~(\ref{eqn:paramsZ1_2}) the quantity $\rho_1^2-\rho_2^2$ is a perfect square of a rational function in $\mathbb{Q}(\lambda_2, \lambda_3)$. The twisted Legendre pencil $\mathcal{X}$ in  Equation~(\ref{eqn:X_general}) then is equivalent to Equation~(\ref{eqn:X_general_JacC_thm}). 
\par First assume that $\rho_2 = 1-\Lambda_1 - \Lambda_2$. Comparing Equations~(\ref{eqn:paramsY}) and~(\ref{eqn:paramsZ1}) one obtains a relation between the Rosenhain roots $(\lambda_2, \lambda_3)$  and the modular parameters $(\Lambda_1,\Lambda_2)$. This relation is precisely Equation~(\ref{eqn:EC_12_j_invariants}), relating a genus-two curve $\mathcal{C}_0$ with an elliptic involution and its elliptic-curve quotients $\mathcal{E}_l$ for $l= 1,2$ in Proposition~\ref{prop:quotient_maps}.  The second solution, obtained by changing $\rho_2 \mapsto - \rho_2$, is related to the first one by a modular transformation, e.g.,  by applying $(\Lambda_1, \Lambda_2) \mapsto (1-\Lambda_1, 1-\Lambda_2)$. Following Remark~\ref{rem:field_extensions} we set $\lambda_l= k_l^2$ for $l=2,3$ and obtain the modular parameters of the elliptic-curve quotients $\mathcal{E}_l$  in Lemma~\ref{lem:ECR} as the algebraic solutions in Equations~(\ref{eqn:moduli_thm}).  This also implies that $\Lambda_1, \Lambda_2 \in \mathbb{P}^1 \backslash \{0, 1, \infty \}$ so that the elliptic curves $\mathcal{E}_1$ and $\mathcal{E}_2$ are smooth. 
\par Combining Proposition~\ref{thm:maps_EC} and Proposition~\ref{thm:maps_Jac} we obtain the following diagram of rational maps:
\beq
\label{eqn:diag3}
\begin{array}{lcccrcc}
	\mathcal{Z} \  \cong \   \mathcal{Y}_3  &  
	\overset{f_3}{\longrightarrow}  &
	\mathcal{X}  &
	\overset{\varphi_1}{\longrightarrow} &
	\mathcal{X}^\vee_1 \ \cong \ \mathcal{W}^{(\varepsilon)}_0  & 
	\overset{\psi}{\longrightarrow}  &
	 \mathcal{Z}^{(2^4)}  
	 \\
	  \ \rotatebox{90}{$\simeq$} &&&&  \rotatebox{90}{$\simeq$}\ \ \ &&  \rotatebox{90}{$\simeq$}\ \quad
	 \\
	 \operatorname{Kum}\left(\mathcal{E}_1\times \mathcal{E}_2\right)  &
	 \longrightarrow &
	 \mathcal{X}  &
	 \longrightarrow &
	 \operatorname{Kum}( \operatorname{Jac}{\mathcal{C}_0} )^{(\varepsilon)}  &
	 \longrightarrow &
	 \operatorname{Kum}\left(\mathcal{E}_1\times \mathcal{E}_2\right)^{(2^4)}  
 \end{array}   
\eeq
By construction, the maps and equivalences in Equation~(\ref{eqn:diag3}) to the left of $\mathcal{X}$ are defined over $\mathbb{Q}(\Lambda_1, \Lambda_2)$. The map $\psi$ is defined over $\mathbb{Q}(\lambda_2 + \lambda_3, q)$ with $q^2=\lambda_2\lambda_3$. The map $\varphi_1$ is defined over $\mathbb{Q}(\rho_1, \rho_2)$, and the equivalence $\mathcal{X}^\vee_1 \cong \mathcal{W}^{(\varepsilon)}_0$ over   $\mathbb{Q}(\lambda_2 + \lambda_3, \lambda_2\lambda_3)$. Thus, the smallest common finite extension field is $\mathbb{Q}(k_1, k_2)$. Thus, all  maps and equivalences are families defined over $\mathbb{Q}(k_2, k_3)$, and the holomorphic two-forms are related by pullback as follows:
\beq
\omega_\mathcal{Z} = \omega_{\mathcal{Y}_3} \,, \quad
\omega_\mathcal{Y} = f_3^* \omega_\mathcal{X} \,, \quad
\omega_\mathcal{X} = \varphi_1^* \omega_{\mathcal{X}_1^\vee} \,, \quad
\omega_{\mathcal{X}_1^\vee}= \omega_{\mathcal{W}^{(\varepsilon)}_0} \,, \quad
\omega_{\mathcal{W}^{(\varepsilon)}_0}  = \psi^* \omega_{\mathcal{Z}^{(2^4)}} \,. 
\eeq 
The remainder of the statement follows easily.
\end{proof}
\par We make the following:
\begin{remark}
\label{rem:trans_lattices0}
For generic parameters $\lambda_2, \lambda_3$ in Theorem~\ref{thm:combined} the transcendental lattices of the corresponding K3 surfaces are given by
\beqn
 \operatorname{T}_{\operatorname{Kum}(\mathcal{E}_1\times \mathcal{E}_2) } = H(2) \oplus H(2), \quad
 \operatorname{T}_\mathcal{X} = \langle 2 \rangle^{\oplus2}  \oplus \langle -2 \rangle^{\oplus2}, \quad
 \operatorname{T}_{\operatorname{Kum}( \operatorname{Jac}{\mathcal{C}_0} )} = H(2) \oplus \langle 4 \rangle \oplus \langle -4 \rangle .
\eeqn
This follows from Remarks~\ref{rem:trans_KumE1E2}, \ref{rem:config}, and Equation~(\ref{eqn:Tlattice}). 
\end{remark}
\subsection{One-parameter subfamilies}
\label{ssec:subfamilies}
In this section we will construct five subfamilies associated with certain modular correspondences for the two elliptic curves in Theorem~\ref{thm:combined}. The first four families are associated with the specializations of the six-line configurations of Picard rank 19 found in Proposition~\ref{lem:branch_locus}. Since two algebraic K3 surfaces $\mathcal{S}_1$ and $\mathcal{S}_2$ with a rational map of finite degree $f: \mathcal{S}_1 \dasharrow \mathcal{S}_2$ have the same Picard number \cite{MR429915}*{Cor.~1.2}, all K3 surfaces occurring in the corresponding Kummer sandwiches have Picard rank 19. The last family is obtained by making one elliptic-curve factor constant with complex multiplication and j-invariant $12^3$.
\subsubsection{Family with isomorphic elliptic curves}
We will consider the three specializations of the families of K3 and abelian surfaces in Theorem~\ref{thm:combined} obtained by requiring the two elliptic curves to be isomorphic. The two elliptic curves $\mathcal{E}_1$ and $\mathcal{E}_2$ are isomorphic if and only if their j-invariants coincide. One easily checks that this is the case if and only if the following expression vanishes:
\beq
 \big(\Lambda_1-\Lambda_2\big) \big(1-\Lambda_1-\Lambda_2\big) \big( \Lambda_1 \Lambda_2 - \Lambda_1 - \Lambda_2\big) \big(1- \Lambda_1\Lambda_2\big)
 \big(\Lambda_1\Lambda_2-\Lambda_1+1\big)  \big(\Lambda_1\Lambda_2-\Lambda_2+1\big).
\eeq
However, it follows from Equation~(\ref{eqn:moduli_thm}) that
\beqn
  (\Lambda_1 - \Lambda_2)^2 = \frac{16 \lambda_2 \lambda_3}{(1-\lambda_2)^2(1-\lambda_3)^2} \not = 0 \,,
\eeqn
whence we must have $\Lambda_1 \not = \Lambda_2$ if $\mathcal{C}_0$ is a smooth genus-two curve. However, the two elliptic curves in Theorem~\ref{thm:combined} can still be isomorphic. We have the following:
\begin{lemma}
The elliptic curves $\mathcal{E}_1$ and $\mathcal{E}_2$ in Theorem~\ref{thm:combined} are isomorphic if the following polynomial vanishes:
\beq
\label{eqn:coincidence}
\begin{split}
0 = \ & \big(\lambda_2+1\big)  \big(\lambda_3+1\big)   \big(2 -\lambda_2-\lambda_3  \big)  \big(\lambda_2\lambda_3 - 2\lambda_2 +1 \big) 
 \big(2\lambda_2\lambda_3 - \lambda_2 - \lambda_3\big)  \big(\lambda_2\lambda_3 - 2\lambda_3 +1 \big) \\
 \times \Big( & \,  (\lambda_2+\lambda_3)^4 - 2 (\lambda_2 \lambda_3) (\lambda_2+\lambda_3)^3 + 3 (\lambda_2\lambda_3)^2  (\lambda_2+\lambda_3)^2 - 2   (\lambda_2\lambda_3)^3 (\lambda_2+\lambda_3) \\
 &  + (\lambda_2\lambda_3)^4 - 2 (\lambda_2+\lambda_3)^3  - 6 (\lambda_2 \lambda_3) (\lambda_2+\lambda_3)^2  +10 (\lambda_2 \lambda_3)^2 (\lambda_2+\lambda_3) - 8 (\lambda_2\lambda_3)^3 \\
 & +  3 (\lambda_2\lambda_3)^2 + 10 (\lambda_2 \lambda_3) (\lambda_2+\lambda_3)  - 2 (\lambda_2\lambda_3)^2 - 2 (\lambda_2+\lambda_3) -8 (\lambda_2\lambda_3) + 1 \Big) \,.
\end{split}
\eeq
\end{lemma} 
\begin{proof}
Equation~(\ref{eqn:coincidence}) is obtained by setting $j_1- j_2=0$. That is, we start with
\beq
 j_1 - j_2 =  j(\mathcal{E}_1) -  j(\mathcal{E}_2) = \frac{256 (\Lambda_1^2-\Lambda_1+1)^3}{\Lambda_1^2 (\Lambda_1-1)^2} -  \frac{256 (\Lambda_2^2-\Lambda_2+1)^3}{\Lambda_2^2 (\Lambda_2-1)^2} \,,
\eeq
use Equations~(\ref{eqn:moduli_thm}) to express the right hand side as a rational function in $\lambda_2, \lambda_3$, and then select the numerator.
\end{proof}
Thus, we consider the one-parameter subfamilies in Theorem~\ref{thm:combined} with a linear (more precisely, linear or fractional linear) relation between the Rosenhain roots:
\begin{enumerate}
\item $\lambda_3=-1$ (or $\lambda_2=-1$) such that $1-\Lambda_1-\Lambda_2=0$,
\item $\lambda_3=2-\lambda_2$ such that $\Lambda_1 \Lambda_2 - \Lambda_1 - \Lambda_2=0$, or,\\
$\lambda_3^{-1}=2-\lambda_2$ such that $1-\Lambda_1\Lambda_2=0$,
\item  $\lambda_3^{-1}=2-\lambda_2^{-1}$ such that $\Lambda_1 \Lambda_2 - \Lambda_1 - \Lambda_2=0$, or, \\
$\lambda_3=2-\lambda_2^{-1}$ such that $1-\Lambda_1\Lambda_2=0$.
\end{enumerate}
The case $\big(\Lambda_1\Lambda_2-\Lambda_1+1\big)  \big(\Lambda_1\Lambda_2-\Lambda_2+1\big) =0$ occurs exactly if the irreducible polynomial of degree $(4,4)$ in Equation~(\ref{eqn:coincidence}) vanishes. We will not consider this case further.
\par For each remaining case we introduce the following one-parameter families of elliptic curves, genus-two curves, double sextic and double quadric surfaces:
\begin{enumerate}
 \item for $\lambda_3=-1$  or  $\lambda_3=2-\lambda_2$ we set 
 \beq
 \label{eqn:family1}
 \begin{array}{rclcl}
	\widetilde{\mathcal{E}}_1:& \quad &y^2 z 	& = & \big(x-z\big) \left(x^2 - \frac{1}{t+1} z^2\right) \,,\\[0.4em]
 	\widetilde{\mathcal{C}}_1:& \quad &Y^2 	& = & XZ \big(X^2+ 4XZ + 4(t+1)Z^2\big) \big(X^2- 4XZ + 4(t+1)Z^2\big) \,,\\[0.4em]
 	\widetilde{\mathcal{X}}_1:& \quad &y^2 	& = & z_1 z_2 z_3 \big(z_1 - z_2\big)  \big(z_1 - z_3\big)  \big(t z_2 + z_3\big) \,, \\[0.4em]
  	\widetilde{\mathcal{Z}}_1^{(\varepsilon_1)}:& \quad & \hat{y}_{12}^2 & = & \varepsilon_1 {\displaystyle \prod_{k=1,2}} \big(x_k-z_k\big) \left(x_k^2 - \frac{1}{t+1} z_k^2\right)z_k \,,
 \end{array}
 \eeq 
 \item for  $\lambda_3=2-\lambda_2$  or  $\lambda_3^{-1}=2-\lambda_2$ we set 
 \beq
  \label{eqn:family2}
 \begin{array}{rclcl}
 	\widetilde{\mathcal{E}}_2:& \quad & y^2 z  &= & \big(x-z\big) \left(x^2 - \frac{(t-1)^2}{(t+1)(t-3)} z^2\right) \,,\\[0.4em]
 	\widetilde{\mathcal{C}}_2:& \quad & Y^2 	& = & XZ \big(X + (t^2-1)Z\big) \big(X + (t-1)^2Z\big) \\
					& & & & \times \big(X+(t+1)(t-3)Z\big) \big(X+(t-1)(t-3)Z\big)  \,,\\[0.4em]
  	\widetilde{\mathcal{X}}_2:& \quad & y^2	& = & z_1 z_2 \big(z_1 - z_2\big)  \big(z_1 - z_3\big)  \big(t^2 z_2 - z_3\big)  \big((t-2)^2 z_2 - z_3\big) \,,\\[0.4em]
    	\widetilde{\mathcal{Z}}_2^{(\varepsilon_2)}:& \quad & \hat{y}_{12}^2 & = & \varepsilon_2 {\displaystyle \prod_{k=1,2}} 
	\big(x_k-z_k\big) \left(x_k^2 - \frac{(t-1)^2}{(t+1)(t-3)} z_k^2\right) z_k \,,
 \end{array}
 \eeq 
 \item for $\lambda_3=2-\lambda_2^{-1}$ or $\lambda_3^{-1}=2-\lambda_2^{-1}$ we set
 \beq
  \label{eqn:family3}
   \begin{array}{rclcl}
	\widetilde{\mathcal{E}}_3: & \quad & y^2 z	&= &x \big(x-z\big) \left(x + \frac{(t+1)^2}{(t-1)^2} z\right) \,,\\[0.4em]
 	\widetilde{\mathcal{C}}_3: & \quad & Y^2 	& = & XZ \big(X + 4t^2 Z\big) \big(X + 2(t^2+1)Z\big)\\
 					& & & & \times  \big(X+2t^2(t^2+1)Z\big) \big(X+(t^2+1)^2Z\big) \,,\\[0.4em]
 	\widetilde{\mathcal{X}}_3: & \quad & y^2 	& = & z_1 z_2 \big(z_1 - z_2\big)  \big(z_1 - z_3\big)  \left(\frac{(t^2+3)^2}{(t^2-1)^2} z_2 - z_3\right)  
	\left(\frac{(3 t^2+1)^2}{(t^2-1)^2} z_2 - z_3\right) \,,\\[0.4em]
	\widetilde{\mathcal{Z}}_3^{(\varepsilon_3)}:& \quad & \hat{y}_{12}^2 & = & \varepsilon_3 {\displaystyle \prod_{k=1,2}} x_k z_k \big(x_k-z_k\big) \left(x_k + \frac{(t+1)^2}{(t-1)^2} z_k\right) \,.
 \end{array}
 \eeq 
\end{enumerate} 
By construction, the minimal resolution of $\widetilde{\mathcal{Z}}_n$ is $\operatorname{Kum}(\widetilde{\mathcal{E}}_n \times\widetilde{\mathcal{E}}_n )$ for $n=1, 2, 3$. We have the following:
\begin{corollary}
\label{thm:special}
For $t \in \mathbb{P}^1 \backslash S_n$ with
\beqn
\begin{array}{rlrlrlrl}
 S_1 &= \lbrace 0, -1 , \infty\rbrace\,,  		& \varepsilon_1&=  \frac{1}{t+1} \,,		& \delta_{1,1} & =   \frac{t+1}{2}  \,, 
 & \delta_{2,1} &=  \frac{1}{8} \,,\\
 S_2 &=\lbrace  \pm 1, 3, \infty \rbrace\,,		&  \varepsilon_2&= \frac{1}{(t+1)(t-3)} \,, \  & \delta_{1,2} & =   \frac{(t+1)(t-3)}{2} \,, \
 & \delta_{2,2} & = \frac{1}{2(t-1)} \,,\\
 S_3 &=\lbrace  0, \pm 1, \pm i, \infty \rbrace\,, \	&  \varepsilon_3 &= 1\,,   				& \delta_{1,3} &=  \frac{t-1}{t+1} \,, 
 & \delta_{2,3} &=  \frac{1}{t (t^2+1)(t^2-1)^2} \,,\\
\end{array}
\eeqn
and $\delta_{3,n} = 4 \, \delta_{1,n}$, the twisted Legendre pencil $ \widetilde{\mathcal{X}}_n$ for $n=1, 2, 3$ fits into the following smooth Kummer sandwich of rational maps:
\beq
\label{eqn:diag_statement3b}
	 \operatorname{Kum}(\widetilde{\mathcal{E}}_n \times\widetilde{\mathcal{E}}_n )^{(-\varepsilon_n \delta_{1,n}^2)}
	 \longrightarrow 
	 \ \widetilde{\mathcal{X}}_n \,
	 \longrightarrow 
	\operatorname{Kum}( \operatorname{Jac}{\widetilde{\mathcal{C}}_n} )^{(\varepsilon_n  \delta_{2,n}^2)}
	 \longrightarrow 
	\operatorname{Kum}(\widetilde{\mathcal{E}}_n \times\widetilde{\mathcal{E}}_n )^{(-\varepsilon_n \delta_{3,n}^{2})}.
\eeq
The maps are defined on families over a finite field extension of $\mathbb{Q}(t)$, and the holomorphic two-form $\omega_{\widetilde{\mathcal{X}}_n}$ equals the pullback of $\omega_{\operatorname{Kum}( \operatorname{Jac}{\widetilde{\mathcal{C}}_n} )^{(\varepsilon_n  \delta_{2,n}^2)}}$, resp.~its pullback equals $\omega_{\operatorname{Kum}(\widetilde{\mathcal{E}}_n \times\widetilde{\mathcal{E}}_n )^{(-\varepsilon_n \delta_{1,n}^2)}}$ and the pullback of $\omega_{\operatorname{Kum}(\widetilde{\mathcal{E}}_n \times\widetilde{\mathcal{E}}_n )^{(-\varepsilon_n \delta_{3,n}^{2})}}$ equals $\omega_{\operatorname{Kum}( \operatorname{Jac}{\widetilde{\mathcal{C}}_n} )^{(\varepsilon_n  \delta_{2,n}^2)}}$.
\end{corollary}
\begin{proof}
In the three cases we introduce a parameter $t$ in the following way:
\begin{enumerate}
 \item for $\lambda_3=-1$ set $t = - \frac{(1+\lambda_2)^2}{(1-\lambda_2)^2} = 4 \Lambda_1 (\Lambda_1-1)$,
 \item for $\lambda_3=2-\lambda_2$ set $\lambda_2 = \frac{t+1}{t-1}$ and $\frac{\Lambda_1^2}{(\Lambda_1-2)^2} =\frac{(t-1)^2}{(t+1)(t+3)}$, or,\\
 for $\lambda_3^{-1}=2-\lambda_2$ set $\lambda_2 =   \frac{t+1}{t-1}$ and $\frac{(\Lambda_1-1)^2}{(\Lambda_1+1)^2} =\frac{(t-1)^2}{(t+1)(t+3)}$,
  \item for  $\lambda_3=2-\lambda_2^{-1}$ set $\lambda_2=\frac{1+t^2}{2}$ and $\Lambda_1 = - \frac{(t+1)^2}{(t-1)^2}$, or, \\
for $\lambda_3^{-1}=2-\lambda_2^{-1}$ set $\lambda_2=\frac{1+t^2}{2}$ and $\Lambda_1 = \frac{2(t^2+1)}{(t+1)^2}$.
 \end{enumerate}
The field of definition $\mathbb{Q}(\lambda_2, \lambda_3)$ is related to $\mathbb{Q}(t)$ at worst by a quadratic field extension. In fact, in the three cases above, we have $\mathbb{Q}(t) \subseteq \mathbb{Q}(\lambda_2, \lambda_3)$, $\mathbb{Q}(\lambda_2, \lambda_3) \cong \mathbb{Q}(t)$, and $\mathbb{Q}(\lambda_2, \lambda_3)  \subseteq  \mathbb{Q}(t)$, respectively. 
\par One checks by an explicit computation that the double quadric surface $\widetilde{\mathcal{Z}}_n^{(-\epsilon_n \delta^2_{1,n})}$ for $n= 1, 2, 3$ is isomorphic to $\mathcal{Z}$ in Equation~(\ref{eqn:Kummer44}) over $\mathbb{Q}(\Lambda_1, \Lambda_2)$ which is (at worst) a finite field extension of $\mathbb{Q}(t)$. Moreover, the pullback by this isomorphism leaves the holomorphic two-form invariant. Taking the minimal resolution we obtain the desired isomorphism between  $\operatorname{Kum}(\widetilde{\mathcal{E}}_n \times\widetilde{\mathcal{E}}_n )^{(-\varepsilon_n \delta_{1,n}^2)}$ and $\operatorname{Kum}(\mathcal{E}_1\times \mathcal{E}_2)$. One also checks that the curve $\widetilde{\mathcal{C}}_n$ for $n=1, 2, 3$ is isomorphic to $\mathcal{C}_0$ for the given (fractional) linear relation between the Rosenhain roots. We use Equation~(\ref{eqn:Ctilde}) and Proposition~\ref{prop:twistedShiodaSextic} with a suitable parameter $\lambda_0$, namely,
\beq
 n=1: \, \lambda_0=\frac{4}{\lambda_2-1} \,, \quad n=2: \, \lambda_0 = - \frac{4}{(\lambda_2-1)^2} \,, \quad n=3: \, \lambda_0=-4(2\lambda_2-1) \,.
\eeq
It then follows that the (twisted) double sextic surfaces $\mathcal{W}^{(\varepsilon)}_0$ in Equation~(\ref{kummer_middle_twist}) is isomorphic to the quadratic twist of the double sextic defined by  $\widetilde{\mathcal{C}}_n$ with twist factor $\varepsilon_n  \delta_{2,n}^2$. By construction, the isomorphism is defined over a finite field extension of $\mathbb{Q}(t)$. Proposition~\ref{prop:twistedShiodaSextic} guarantees that the pullback by the isomorphism leaves the holomorphic two-form invariant.  The proof then follows from applying the constructed isomorphisms to the situation in Theorem~\ref{thm:combined}. The set $S_n$ is obtained from the vanishing locus of the discriminant for $\widetilde{\mathcal{C}}_n$.
\end{proof}
\par We make the following two remarks:
\begin{remark}
\label{rem:trans_lattices}
For the surfaces in Corollary~\ref{thm:special} the transcendental lattices in Remark~\ref{rem:trans_lattices0} specialize to
\beqn
 \operatorname{T}_{\operatorname{Kum}( \widetilde{\mathcal{E}}_n\times \widetilde{\mathcal{E}}_n) } = H(2) \oplus \langle 4 \rangle, \quad
 \operatorname{T}_{\widetilde{\mathcal{X}}_n} = \langle 2 \rangle^{\oplus2}  \oplus \langle -2 \rangle, \quad
 \operatorname{T}_{\operatorname{Kum}( \operatorname{Jac}{\widetilde{\mathcal{C}}_n} )} = \langle 4 \rangle^{\oplus2} \oplus \langle -4 \rangle .
\eeqn
This follows from Corollary~\ref{thm:special} and results in \cites{MR1023921, MR2214473}.
\end{remark}
\begin{remark}
\label{rem:locus}
In the two-dimensional parameter space $\{ (A,B) \, | \, A \not = B \}$ of Equation~(\ref{eqn:Xrestriction}) we considered the conic $(A+B-4)^2-4 AB=0$ in Proposition~\ref{lem:branch_locus}. One checks that for 
\beq
\label{eqns:paramsAB}
 A = \left(\frac{1+\lambda_2}{1-\lambda_2}\right)^2 \,, \qquad  B = \left(\frac{1+\lambda_3}{1-\lambda_3}\right)^2 \,,
\eeq 
the vanishing locus of the conic becomes
\beq
 \prod_{\sigma_2, \sigma_3  \in \{ \pm 1\}}  \Big( 2 - \lambda_2^{\sigma_2} - \lambda_3^{\sigma_3} \Big) = 0 \,.
\eeq
Thus, the double sextic surfaces associated with these four components realize Equations~(\ref{eqn:family2}) and~(\ref{eqn:family3}).
 \end{remark}
 \subsubsection{Family with isogenous elliptic curves}
As another example, we consider the classical modular curve $X_0(2)$, i.e., the irreducible plane algebraic curve of genus zero, given by
\beq
\label{eqn:modular_X_0(2)}
\begin{split}
 0 = & - j_1^2 j_2^2 + j_1^3 + j_2^3 + 1488 j_1 j_2 (j_1 + j_2) - 2^4 3^4 5^3 (j_1^2+j_2^2)  \\
 &+ 3^4 5^3 4027 j_1 j_2+ 2^7 3^7 5^5 (j_1 + j_2) - 2^{12} 3^9 5^9 \,.
\end{split} 
\eeq
This is the relation between the j-invariants $j_1=j(\tau_1) =j (\mathcal{E}_1)$ and $j_2=j(\tau_2) =j (\mathcal{E}_2)$ of two elliptic curves $\mathcal{E}_1$ and $\mathcal{E}_2$, respectively, such that $\tau_2=2\cdot \tau_1$; see \cite{MR1512996}. We have the following:
\begin{lemma}
\label{lem:component_X02}
Assume $t \in \mathbb{P}^1 \backslash \{0, \pm 1, \infty \}$. For the two smooth elliptic curves
\beq
 \mathcal{E}_1: \quad y_1^2 z_1 = x_1 \big(x_1 - z_1 \big) \big( x_1 - t^2 z_1\big) \,, \qquad  \mathcal{E}_2: \quad y_2^2 z_2 = \big(x_2 - (t^2+1) z_2 \big) \big( x_2^2 - 4 t^2 z^2_2 \big)\,,
\eeq
with j-invariants $j_1=j(\tau_1) =j ( \mathcal{E}_1)$ and $j_2=j(\tau_2) =j ( \mathcal{E}_2)$, respectively, one has $\tau_2=2\cdot \tau_1$. Moreover, for the two-isogenous elliptic curves 
-- as defined in Equation~(\ref{eqn:EC_dual}) -- one has $\mathcal{E}^\prime_1 \cong  \mathcal{E}_2$ and $\mathcal{E}^\prime_2 \cong  \mathcal{E}_1$.
\end{lemma}
\begin{proof}
For two elliptic curves in Legendre normal form as in Equation~(\ref{eqn:EC}) one irreducible component of Equation~(\ref{eqn:modular_X_0(2)}) is given by
\beq
\label{eqn:component_X02}
 0 =  16 \Lambda_1^2 \Lambda_2^2 - 16 \Lambda_1 \Lambda_2 (\Lambda_1+\Lambda_2) + 16 \Lambda_1 \Lambda_2-1 \,.
\eeq
$X_0(2)$ has a rational parametrization given by $j_1 = (h +256)^3/h^2$ and $j_2=(h+16)^3/h$ for $h \in \mathbb{C}^\times$; see \cite{MR1512996}. We may take a suitable covering by setting $h = 16(t^2-1)^2/t^2$ to obtain rational modular parameters for elliptic curves in Legendre form; see \cite{MR3068406}. Then, the elliptic curves in Legendre form with $\Lambda_1= -1/(t^2-1)$ and $\Lambda_2=(t+1)^2/(4t)$, i.e., the curves
\beq
 \mathcal{E}_1: \quad y^2 z = x \Big(x - z \Big) \left( x + \frac{1}{t^2-1} z \right) \,, \qquad  \mathcal{E}_2: \quad y^2 z = x \Big(x - z \Big) \left( x - \frac{(t+1)^2}{4t} z \right) \,,
\eeq
satisfy $j_1 =j (\mathcal{E}_1)$ and $j_2 =j (\mathcal{E}_2)$. In particular, the given parameters $\Lambda_1, \Lambda_2$ satisfy Equation~(\ref{eqn:component_X02}). The two-isogenous elliptic curves 
in Equation~(\ref{eqn:EC_dual}) are then given by
\beqn
 \mathcal{E}^\prime_1 : \ Y_1^2  Z_1  = X_1  \left(X_1^2+ \frac{2(t^2+1)}{(t^2-1)}  X_1 Z_1 + Z_1^2 \right) \,, 
 \quad
  \mathcal{E}^\prime_2 : \ Y_2^2  Z_2  = X_2  \left(X_2^2   - \frac{t^2+1}{t}  X_2 Z_2 + Z_2^2 \right) \,.
\eeqn
The remainder of the statement is checked by an explicit computation. 
\end{proof}
We make the following:
\begin{remark}
\label{rem:locus2}
The proof of Lemma~\ref{lem:component_X02} shows that the double sextic surfaces for case (4) in Proposition~\ref{lem:branch_locus} are covered by Kummer surfaces associated with two elliptic curves $\mathcal{E}_1$ and $\mathcal{E}_2$ such that $\tau_2=2\cdot \tau_1$. In particular, the conic $(A-B+1)(A-B-1)=0$ is equivalent to the relation
\beq
 0 =  16 \Lambda_1^2 \Lambda_2^2 - 16 \Lambda_1 \Lambda_2 (\Lambda_1+\Lambda_2) + 16 \Lambda_1 \Lambda_2-1 
\eeq
between the elliptic modular parameters of the elliptic curves. For the two-isogenous elliptic curves -- as defined by Equation~(\ref{eqn:EC_dual}) -- we then have $j(\mathcal{E}_1) =  j(\mathcal{E}^\prime_2) $ and  $j(\mathcal{E}_2) =  j(\mathcal{E}^\prime_1)$.  Using Equations~(\ref{eqn:moduli_thm}) this is also equivalent to the relation 
\beq
\label{eqn:locus_isog_Rosenhain}
\begin{split}
0 = & \big( \lambda_2^2\lambda_3^2 - 6 \lambda_2^2 \lambda_3 + 2 \lambda_2 \lambda_3^2+ \lambda_2^2 + \lambda_3^2 + 4 \lambda_2 \lambda_3 + 2 \lambda_2 - 6 \lambda_3 +1 \big) \\
\times & \big( \lambda_2^2\lambda_3^2 +2 \lambda_2^2 \lambda_3 - 6 \lambda_2 \lambda_3^2+ \lambda_2^2 + \lambda_3^2 + 4 \lambda_2 \lambda_3 -6 \lambda_2 +2 \lambda_3 +1 \big) \,,
\end{split}
\eeq
between the Rosenhain roots of a genus-two curve $\mathcal{C}_0$ with an elliptic involution.
 \end{remark}
 We define the following one-parameter families of elliptic curves, genus-two curves, double sextic and double quadric surfaces:  
  \beq
  \label{eqn:family4}
   \begin{array}{rclcl}
   \widetilde{\mathcal{E}}_4: & \quad & y^2 z	&= & x \big(x - z \big) \big( x - t^2 z\big) \,,\\[0.4em]
   \widetilde{\mathcal{E}}^\prime_4: & \quad & Y^2 Z	&= &  \big(X - (t^2+1) Z \big) \big( X^2 - 4 t^2 Z^2 \big) \,,\\[0.1em]
   \widetilde{\mathcal{C}}_4: & \quad & Y^2	&= & XZ \Big( X^4 - 4 X^3 Z + \frac{2(3t^4+8t^3+6t^2-8t+3)}{(t^2+1)^2} X^2Z^2\\
   & & & & - \frac{4(t^2+2t-1)^2}{(t^2-2t-1)^2} XZ^3 + \frac{(t^2+2t-1)^4}{(t^2-2t-1)^4} Z^4 \Big) \,,\\[0.1em]
   \widetilde{\mathcal{X}}_4: & \quad & y^2 	& = & z_1 z_2 \big(z_1 - z_2\big)  \big(z_1 - z_3\big) \left( z_3 + \frac{(t^2+1)^2}{4 t(t^2-1)} z_2 \right)  
   \left( z_3 + \frac{(t^2-2t-1)^2}{4 t(t^2-1)} z_2 \right)\,,\\[0.1em]
   \widetilde{\mathcal{Z}}_4^{(\varepsilon_4)}:& \quad & \hat{y}_{12}^2 & = & \varepsilon_4  x_1 z_1 \big(x_1 - z_1 \big) \big( x_1 - t^2 z_1\big)
    \big(x_2 - (t^2+1) z_2 \big) \big( x_2^2 - 4 t^2 z_2^2 \big) z_2\,.
\end{array}
\eeq 
By construction, the minimal resolution of $\widetilde{\mathcal{Z}}_4^{(\varepsilon_4)}$ is $\operatorname{Kum}(\widetilde{\mathcal{E}}_4 \times\widetilde{\mathcal{E}}_4^\prime )^{(\varepsilon_4)}$. Moreover, the j-invariants of $\widetilde{\mathcal{E}}_4$ and $\widetilde{\mathcal{E}}_4^\prime$ satisfy Equation~(\ref{eqn:modular_X_0(2)}). We have the following:
\begin{corollary}
\label{thm:special4}
For $t \in \mathbb{P}^1 \backslash S_4$ with $S_4 = \{ 0, \pm (\sqrt{2} \pm 1), \pm 1, \pm i,  \infty\}$ and
\beq
 \varepsilon_4 = \frac{1}{t(t^2-1)} \,, \qquad \delta_{1,4} =  \frac{1}{2}\,, \qquad \delta_{2,4} = \frac{(t^2-2t-1)^2(t^2+1)}{8 t (t^2-1) (t^2+2t-1)} \,, \qquad \delta_{3,4} = 2 \,,
\eeq
the twisted Legendre pencil $\widetilde{\mathcal{X}}_4$ fits into the following smooth Kummer sandwich of rational maps:
\beq
\label{eqn:diag_statement4}
	 \operatorname{Kum}(\widetilde{\mathcal{E}}_4 \times\widetilde{\mathcal{E}}^\prime_4 )^{(\varepsilon_4 \delta_{1,4}^2)}
	 \longrightarrow 
	 \ \widetilde{\mathcal{X}}_4 \,
	 \longrightarrow 
	\operatorname{Kum}( \operatorname{Jac}{\widetilde{\mathcal{C}}_4} )^{(\delta_{2,4}^2)}
	 \longrightarrow 
	\operatorname{Kum}(\widetilde{\mathcal{E}}_4 \times\widetilde{\mathcal{E}}^\prime_4 )^{(\varepsilon_4 \delta_{3,4}^2)}.
\eeq
The maps are defined on families over a finite field extension of $\mathbb{Q}(t)$, and the holomorphic two-form $\omega_{\widetilde{\mathcal{X}}_4}$ equals the pullback of $\omega_{\operatorname{Kum}( \operatorname{Jac}{\widetilde{\mathcal{C}}_4} )^{(\delta_{2,4}^2)}}$, resp.~its pullback equals $\omega_{\operatorname{Kum}(\widetilde{\mathcal{E}}_4 \times\widetilde{\mathcal{E}}_4^\prime )^{(\varepsilon_4 \delta_{1,4}^2)}}$ and the pullback of $\omega_{\operatorname{Kum}(\widetilde{\mathcal{E}}_4 \times\widetilde{\mathcal{E}}_4^\prime )^{(\varepsilon_4 \delta_{3,4}^2)}}$ equals $\omega_{\operatorname{Kum}( \operatorname{Jac}{\widetilde{\mathcal{C}}_4} )^{(\delta_{2,4}^2)}}$.
\end{corollary}
\begin{proof}
We assume that  two elliptic curves $\mathcal{E}_1$ and $\mathcal{E}_2$ in Theorem~\ref{thm:combined} are given by $\Lambda_1= -1/(t^2-1)$ and $\Lambda_2=(t+1)^2/(4t)$. It follows from Lemma~\ref{lem:component_X02} that $\tau_1 = 2 \tau_2$, and for the two-isogenous elliptic curves in Equation~(\ref{eqn:EC_dual}) we have $j(\mathcal{E}_1) =  j(\mathcal{E}^\prime_2) = j(\widetilde{\mathcal{E}}_4)$ and  $j(\mathcal{E}_2) =  j(\mathcal{E}^\prime_1)=  j(\widetilde{\mathcal{E}}^\prime_4)$. By an explicit computation one checks that the isomorphisms
\beq
 \mathcal{E}_1 \cong \left(\widetilde{\mathcal{E}}_4\right)^{(\frac{1}{t^2-1})} \,, \qquad \mathcal{E}_2 \cong \left(\widetilde{\mathcal{E}}_4^\prime\right)^{(\frac{1}{4t})}\,, \qquad  
  \mathcal{E}_1^\prime \cong \left(\widetilde{\mathcal{E}}_4^\prime\right)^{(\frac{1}{t})}  \,, \qquad \mathcal{E}_2^\prime \cong \left(\widetilde{\mathcal{E}}_4\right)^{(\frac{1}{t^2-1})} \,,
\eeq
are defined over $\mathbb{Q}(t)$ and identify the corresponding holomorphic one-forms via pullback. Furthermore, we use Equation~(\ref{eqn:Ctilde}) and Proposition~\ref{prop:twistedShiodaSextic} with the parameter
\beq
\lambda_0  =  \frac{(T^4+1)(T^4+2T^2-1)^2}{(T^4-2T^2-1)(T^8+4DT^5-6T^6-4DT^3+2T^2-1)} \,,
\eeq
where $t=T^2$ and $D^2=1-T^4$. The remainder of the proof is analogous to the proof of Corollary~\ref{thm:special}.
\end{proof}
\subsubsection{Family with trivial factor}
The last family is not obtained by a fixed modular correspondence between two elliptic curves. Instead we make one elliptic-curve factor constant with complex multiplication and j-invariant $12^3$. We introduce the following one-parameter families of elliptic curves, genus-two curves, double sextic and double quadric surfaces:  
   \beq
 \label{eqn:family5}
 \begin{array}{rclcl}
	\widetilde{\mathcal{E}}_5:& \quad &y^2 z 	& = & \big(x-(t^2+1)z\big) \left(x^2 -4 t^2 z^2\right) \,,\\[0.4em]
	\widetilde{\mathcal{F}}_5:& \quad &Y^2 Z 	& = & X \big(X^2-Z^2) \,,\\[0.4em]
 	\widetilde{\mathcal{C}}_5:& \quad &Y^2 	& = & XZ \big(X^2- 2 (t^2-4t+1)XZ + (t^2+1)^2Z^2\big) \\
						&	&	&  & \times \big(X^2- 2 (t^2+4t+1)XZ + (t^2+1)^2Z^2\big) \,,\\[0.2em]
 	\widetilde{\mathcal{X}}_5:& \quad &y^2 	& = & z_1 z_2  \big(z_1 - z_2\big) \big(z_1 - z_3\big)  \big(z_3^2 - z_2 z_3 + \frac{(t^2+1)^2}{16t^2} z_2^2\big) \,, \\[0.4em]
  	\widetilde{\mathcal{Z}}_5^{(\varepsilon_5)}:& \quad & \hat{y}_{12}^2 & = & \varepsilon_5 \big(x_1-(t^2+1)z_1\big) \left(x_1^2 -4 t^2 z_1^2\right) z _1 x_2 z_2 \big(x_2^2-z_2^2)  \,.
 \end{array}
 \eeq
By construction, the minimal resolution of $\widetilde{\mathcal{Z}}_5^{(\varepsilon_5)}$ is $\operatorname{Kum}(\widetilde{\mathcal{E}}_5 \times\widetilde{\mathcal{F}}_5 )^{(\varepsilon_5)}$. Moreover, one checks that $\widetilde{\mathcal{F}}_5$ is constant with complex multiplication and $j(\widetilde{\mathcal{F}}_5)=12^3$. We have the following:
 \begin{corollary}
\label{thm:special5}
For $t \in \mathbb{P}^1 \backslash S_5$ with $S_5 = \{ 0, \pm 1, \pm i, \infty \}$ and
\beq
 \varepsilon_5 = \frac{1}{4t} \,, \qquad  \delta_{1,5} = \frac{1}{2} \,, \qquad \delta_{2,5} = \frac{1}{8t (t^2+1)}, \qquad \delta_{3,5} =2\,,
\eeq
 the twisted Legendre pencil $ \widetilde{\mathcal{X}}_5$ fits into the following smooth Kummer sandwich of rational maps:
\beq
\label{eqn:diag_statement3c}
	 \operatorname{Kum}(\widetilde{\mathcal{E}}_5 \times\widetilde{\mathcal{F}}_5 )^{(\varepsilon_5 \delta_{1,5}^2)}
	 \longrightarrow 
	 \ \widetilde{\mathcal{X}}_5 \,
	 \longrightarrow 
	\operatorname{Kum}( \operatorname{Jac}{\widetilde{\mathcal{C}}_5} )^{(\delta_{2,5}^2)}
	 \longrightarrow 
	\operatorname{Kum}(\widetilde{\mathcal{E}}_5 \times\widetilde{\mathcal{F}}_5 )^{(\varepsilon_4 \delta_{3,5}^{2})}.
\eeq
The maps are defined on families over a finite field extension of $\mathbb{Q}(t)$, and the holomorphic two-form $\omega_{\widetilde{\mathcal{X}}_5}$ equals the pullback of $\omega_{\operatorname{Kum}( \operatorname{Jac}{\widetilde{\mathcal{C}}_5} )^{(\delta_{2,5}^2)}}$, resp.~its pullback equals $\omega_{\operatorname{Kum}(\widetilde{\mathcal{E}}_5 \times\widetilde{\mathcal{F}}_5 )^{(\varepsilon_5 \delta_{1,5}^2)}}$ and the pullback of $\omega_{\operatorname{Kum}(\widetilde{\mathcal{E}}_5 \times\widetilde{\mathcal{F}}_5 )^{(\varepsilon_4 \delta_{3,5}^{2})}}$ equals $\omega_{\operatorname{Kum}( \operatorname{Jac}{\widetilde{\mathcal{C}}_5} )^{(\delta_{2,5}^2)}}$.
\end{corollary}
\begin{proof}
We assume that the two elliptic curves $\mathcal{E}_1$ and $\mathcal{E}_2$ in Theorem~\ref{thm:combined} are given by $\Lambda_1= (t+1)^2/(4t)$ and $\Lambda_2=1/2$. Furthermore, we use Equation~(\ref{eqn:Ctilde}) and Proposition~\ref{prop:twistedShiodaSextic} with the parameter
\beq
\lambda_0  =  (u+ i \omega)^2 (u - \omega)^2\,,
\eeq
where $t=u^2$ and $\omega=\exp{(-\pi i/4)}$. The remainder of the proof is analogous to the proof of Corollary~\ref{thm:special}.
\end{proof}
 \subsection{Construction of correspondences}
In Theorem~\ref{thm:combined} we constructed a Kummer sandwich for the Legendre pencil in Equation~\ref{eqn:X_general_JacC_thm}. The surfaces involved are defined over $\mathbb{Q}(\lambda_2, \lambda_3)$, the maps are defined on families over the finite field extension $\mathbb{Q}(k_2, k_3)$ with $k_l^2=\lambda_l$ for $l= 2,3$.  
\par More generally, let us consider two families of Jacobian elliptic K3 surfaces $V_{\lambda}$ and $W_{\lambda}$ defined over $\mathbb{Q}(\lambda_2, \lambda_3)$ whose generic members have Picard rank 18. Here, we denote the tuple $(\lambda_2, \lambda_3)$ by $\lambda$.  Since we assume that the K3 surfaces are Jacobian elliptic K3 surfaces, they have Weierstrass models over $\mathbb{Q}(\lambda_2, \lambda_3)$. For example, for $V_{\lambda}$ we assume this Weierstrass model to be of the form
\beq
\label{eqn:WEQ2end}
 V_{\lambda} : \quad y^2 z = 4 x^3 -g_2(v) \, x z^2 - g_3(v) z^3 \,,
\eeq 
where $v$ is a suitable coordinate on the base curve $\mathbb{P}^1$, and $g_2$ and $g_3$ are polynomials of degree $8$ and $12$, respectively, with coefficients in $\mathbb{Q}(\lambda_2, \lambda_3)$. There is a natural notion of twisted K3 surfaces $V^{(\varepsilon)}_{\lambda}$ for $\varepsilon \in \mathbb{Q}(\lambda_2, \lambda_3)$ by taking the minimal resolution of the twisted Weierstrass model
\beq
\label{eqn:WEQ2endtwisted}
 V^{(\varepsilon)}_{\lambda} : \quad \hat{y}^2 z = \varepsilon \, \Big(4 x^3 -g_2(v) \, x z^2 - g_3(v) z^3 \Big) \,.
\eeq 
Let  $F: V_{\lambda} \dasharrow W_{\lambda}$ be a rational map  of finite degree. An \emph{algebraic correspondence} from $V_{\lambda}$ to $W_{\lambda}$ is a subset $\Gamma \subset V_{\lambda} \times W_{\lambda}$ such that $\Gamma$ is finite and surjective over each component of $V_{\lambda}$. Then, the graph of $F:  V_{\lambda} \dasharrow W_{\lambda}$ is clearly such a correspondence $\Gamma = \Gamma_F$. We have the following:
\begin{lemma}
\label{lem:twist_away}
Assume $\varepsilon = \delta^2 \not = 0$ with $\delta \in \mathbb{Q}(\lambda_2, \lambda_3)$. Then, there is an algebraic correspondence from $V_{\lambda}$ to $V^{(\varepsilon)}_{\lambda}$ over $\mathbb{Q}(\lambda_2, \lambda_3)$ that induces a nonzero map $H^{2,0}(V^{(\varepsilon)}_{\lambda}) \to H^{2,0}(V_{\lambda})$. 
\end{lemma}
\begin{proof}
Since $\varepsilon = \delta^2$ the map from Equation~(\ref{eqn:WEQ2end}) to~Equation~(\ref{eqn:WEQ2endtwisted}) is given by $F: y \mapsto \hat{y} = \delta y$. Generators of $H^{2,0}(V_{\lambda})$ and $H^{2,0}(V^{(\varepsilon)}_{\lambda})$ are given in the charts $z=1$ by the regular two-forms $\omega=dv \wedge dx/y$  and  $\omega'=dv \wedge dx/\hat{y}$, respectively. We have
$F^* \omega' = \omega/\delta \not = 0$.
\end{proof}
\par For a map $F: V_{\lambda} \dasharrow W_{\lambda}$ which is defined over $\mathbb{Q}(k_2, k_3)$, the graph $\Gamma_F$ determines a correspondence, but only over $\mathbb{Q}(k_2, k_3)$. However, in our situation $F$ has an additional property:  for generators of $H^{2,0}(V_{\lambda})$ and $H^{2,0}(W_{\lambda})$ given in local coordinates by the regular two-forms $\omega_{V_{\lambda}}$  and  $\omega_{W_{\lambda}}$, respectively, we have $F^*\omega_{W_{\lambda}} = \omega_{V_{\lambda}}$.  We consider the $\operatorname{Gal}{(\mathbb{Q}(k_2, k_3)/\mathbb{Q}(\lambda_2, \lambda_3))}$-conjugate maps $F_{(\pm, \pm)}: V_{\lambda} \dasharrow W_{\lambda}$ obtained by replacing $(k_1, k_2) \mapsto (\pm k_1, \pm k_2)$ such that $F_{(+,+)}=F$. Since the surfaces and the two-forms are defined over $\mathbb{Q}(\lambda_2, \lambda_3)$, we obtain
\beq
 \sum_{\sigma_2, \sigma_3 \in \{ \pm \}}  F_{(\sigma_2, \sigma_3)}^*\omega_{W_{\lambda}} = 4 \, \omega_{V_{\lambda}}  \not = 0\,. 
\eeq 
The sum of the graphs
\beq
 \Gamma \ = \ \Gamma_{F_{(+,+)}} + \Gamma_{F_{(+,-)}} + \Gamma_{F_{(-,+)}} + \Gamma_{F_{(-,-)}} 
\eeq  
then defines a correspondence on $V_{\lambda} \times W_{\lambda}$ over $\mathbb{Q}(\lambda_2, \lambda_3)$ that induces a nonzero map $H^{2,0}(W_{\lambda}) \to H^{2,0}(V_{\lambda})$. Thus, it must induce an isomorphism on the transcendental lattices of $W_{\lambda}$ and $V_{\lambda}$ when tensored with $\mathbb{Q}$. The Artin comparison theorem for complex and $\ell$-adic cohomology implies that the same is true for the corresponding Galois representations \cite{MR0232775}. Therefore, the correspondence $\Gamma$ induces an isomorphism of $\mathrm{G}_{\mathbb{Q}(\lambda_2,\lambda_3)}$-representations, denoted by $[\Gamma]$, between the part of $H^2_{\text{\'et}}(V_{\lambda, \overline{\mathbb{Q}}}, \mathbb{Q}_\ell)$ that is orthogonal to the eighteen algebraically independent \'etale cycle classes and the corresponding part of $H^2_{\text{\'et}}(W_{\lambda, \overline{\mathbb{Q}}}, \mathbb{Q}_\ell)$. This isomorphism produces, via the K\"unneth formula and Poincar\'e duality, a Galois invariant class in $H^4_{\text{\'et}}(V_{\lambda, \overline{\mathbb{Q}}} \times W_{\lambda, \overline{\mathbb{Q}}},  \mathbb{Q}_\ell)$. Of course, the existence of such a class is already expected due to the Tate conjecture.
\par In the generic situation of Theorem~\ref{thm:combined} the K3 surfaces have Picard rank 18. For the continuous Galois representations of $\mathrm{G}_\mathbb{Q} = \operatorname{Gal}{(\overline{\mathbb{Q}}/\mathbb{Q})}$ on the transcendental lattices, there are isomorphisms of $\mathrm{G}_\mathbb{Q}$-representations, namely
\beq
\label{eqn:Tlattices}
\begin{split}
 H^2_{\text{\'et}}\Big(\mathcal{X}_{\overline{\mathbb{Q}}}, \mathbb{Q}_\ell \Big)  & \ \cong \ \mathrm{T}^{(1)}_{\ell} \oplus \mathbb{Q}_\ell (-1)^{\oplus 18} , \\
 H^2_{\text{\'et}}\Big( \operatorname{Kum}( \mathcal{E}_1 \times \mathcal{E}_2 )_{\overline{\mathbb{Q}}}, \mathbb{Q}_\ell \Big) & \ \cong \ \mathrm{T}^{(2)}_{\ell} \oplus \mathbb{Q}_\ell (-1)^{\oplus 18} ,\\
 H^2_{\text{\'et}}\Big( \operatorname{Kum}( \operatorname{Jac}{\mathcal{C}_0 })^{(\varepsilon)}_{\overline{\mathbb{Q}}}, \mathbb{Q}_\ell \Big) 
 & \ \cong \ \mathrm{T}^{(3)}_{\ell} \oplus \mathbb{Q}_\ell (-1)^{\oplus 18} ,
\end{split} 
\eeq 
where $\mathrm{T}^{(m)}_{\ell}$ are the four-dimensional $\ell$-adic representations obtained from the corresponding transcendental lattices using the comparison theorem for $m= 1, 2, 3$.  Here, we used the fact that $\mathrm{NS}_\mathbb{Q}$ is generated by curves defined over $\mathbb{Q}$. This explains the summand $\mathbb{Q}_\ell (-1)^{\oplus 18}$ in Equation~(\ref{eqn:Tlattices}). We have the following:
\begin{theorem}
\label{thm:ladic_1}
Assume that $\lambda_2, \lambda_3 \in \mathbb{P}^1 \backslash  \{ 0, 1, \infty\}$ satisfy $\lambda_2 \not= \lambda_3^{\pm1}$ and are generic otherwise. Further assume that $\Lambda_1, \Lambda_2$ satisfy 
\beq
\label{eqn:moduli_intro_thm_end}
 \Lambda_1 \Lambda_2 = \dfrac{(\lambda_2 +\lambda_3)^2 -4 \lambda_2\lambda_3}{(1-\lambda_2)^2 (1-\lambda_3)^2} \,, \qquad
 \Lambda_1 + \Lambda_2 = - \dfrac{2(\lambda_2 +\lambda_3)}{(1-\lambda_2) (1-\lambda_3)} \,,
\eeq 
and $\varepsilon =\lambda_2\lambda_3$. Then, there are explicit algebraic correspondences
\beq
\begin{split}
 \Gamma^{(1,2)} &\ \subset \ \operatorname{Kum}( \mathcal{E}_1 \times \mathcal{E}_2) \times \mathcal{X} \,, \\
 \Gamma^{(2,3)} &\ \subset \ \mathcal{X} \times \operatorname{Kum}( \operatorname{Jac}{\mathcal{C}_0 })^{(\varepsilon)} \,, \\
 \Gamma^{(3,1)} &\ \subset \ \operatorname{Kum}( \operatorname{Jac}{\mathcal{C}_0 })^{(\varepsilon)} \times  \operatorname{Kum}( \mathcal{E}_1 \times \mathcal{E}_2)  \,, \\
\end{split} 
\eeq
defined over $\mathbb{Q}(\lambda_2, \lambda_3)$, which induce isomorphisms of the $\ell$-adic representations
\beq
 \left\lbrack \Gamma^{(i,j)} \right\rbrack: \quad  \mathrm{T}^{(i)}_{\ell}  \ \overset{\cong}{\longrightarrow} \ \mathrm{T}^{(j)}_{\ell} \,,
\eeq
for $i, j \in \{1, 2, 3 \}$ with $i \not = j$.
\end{theorem}
\begin{proof}
It follows from Remark~\ref{rem:trans_lattices} that for generic parameters $\lambda_2, \lambda_3 \in \mathbb{Q}$, the K3 surfaces in Theorem~\ref{thm:combined} have Picard rank $18$. Using the families of maps in Equation~(\ref{eqn:diag_statement3}) defined over $\mathbb{Q}(k_2, k_3)$, we construct correspondences between the different Kummer surfaces as explained above. Lemma~\ref{lem:twist_away} allows us to replace the twist factor $\varepsilon =  \frac{4}{\lambda_2 \lambda_3 (1-\lambda_2)^2(1-\lambda_3)^2}$ by $\varepsilon' = \lambda_2 \lambda_3$ as $\varepsilon'/\varepsilon = (\lambda_2 \lambda_3  (1-\lambda_2)(1-\lambda_3)/2)^2$. Similarly for $\Gamma^{(3,1)}$ we eliminate the twist factor $2^4$.
\end{proof}
\par In the situation of Corollary~\ref{thm:special} the K3 surfaces have Picard rank 19. Then, there are isomorphisms of $\mathrm{G}_\mathbb{Q}$-representations, namely
\beqn
\begin{split}
 H^2_{\text{\'et}}\Big(\widetilde{\mathcal{X}}_{n, \overline{\mathbb{Q}}}, \mathbb{Q}_\ell \Big) & \ \cong \ \widetilde{\mathrm{T}}^{(1)}_{n, \ell} \oplus \mathbb{Q}_\ell (-1)^{\oplus 19} \,, \\
 H^2_{\text{\'et}}\Big( \operatorname{Kum}( \widetilde{\mathcal{E}}_n \times \widetilde{\mathcal{E}}_n)^{(-\varepsilon_n)}_{\overline{\mathbb{Q}}}, \mathbb{Q}_\ell \Big) & \ \cong \ \widetilde{\mathrm{T}}^{(2)}_{n, \ell} \oplus \mathbb{Q}_\ell (-1)^{\oplus 19} \,,\\
 H^2_{\text{\'et}}\Big( \operatorname{Kum}( \operatorname{Jac}{\widetilde{\mathcal{C}}_n })^{(\varepsilon_n)}_{\overline{\mathbb{Q}}}, \mathbb{Q}_\ell \Big) 
 & \ \cong \ \widetilde{\mathrm{T}}^{(3)}_{n, \ell} \oplus \mathbb{Q}_\ell (-1)^{\oplus 19} \,,
\end{split} 
\eeqn
where $\widetilde{\mathrm{T}}^{(m)}_{n, \ell}$ are the three-dimensional $\ell$-adic representations obtained from the corresponding transcendental lattices using the comparison theorem for $m, n= 1, 2, 3$.  In the cases above, the representation $\widetilde{\mathrm{T}}^{(2)}_{n, \ell}$ are naturally isomorphic to a symmetric square, namely $\operatorname{Sym}^2( H^1_{\text{\'et}}(\widetilde{\mathcal{E}}_{n,\overline{\mathbb{Q}}}, \mathbb{Q}_\ell))(\chi)$ where $\chi$ is the Dirichlet character of the quadratic twist given by $\varepsilon_n$; see \cites{MR2214473}. We have the following:
\begin{theorem}
\label{thm:ladic_2}
In Equations~(\ref{eqn:family1}), (\ref{eqn:family2}), (\ref{eqn:family3}) let $t \in \mathbb{P}^1 \backslash S_n$  with
\beqn
 S_1 = \lbrace 0, -1, \infty \rbrace\,,  	 \qquad S_2 =\lbrace  \pm 1, 3 , \infty  \rbrace\,  \qquad S_3 =\lbrace 0, \pm 1, \pm i, \infty \rbrace\,,
\eeqn
and $\varepsilon_1=  \frac{1}{t+1}$, $\varepsilon_2= \frac{1}{(t+1)(t-3)}$, $\varepsilon_3=1$. Then, there are explicit algebraic correspondences
\beq
\begin{split}
	\widetilde{\Gamma}_n^{(1,2)} &\ \subset  \ \operatorname{Kum}( \widetilde{\mathcal{E}}_n \times \widetilde{\mathcal{E}}_n)^{(-\varepsilon_n)} \times \widetilde{\mathcal{X}}_n \,, \\
	\widetilde{\Gamma}_n^{(2,3)} &\ \subset  \ \widetilde{\mathcal{X}}_{n} \times  \operatorname{Kum}( \operatorname{Jac}{\widetilde{\mathcal{C}}_n })^{(\varepsilon_n)}\,, \\
	\widetilde{\Gamma}_n^{(3,1)} &\ \subset  \ \operatorname{Kum}( \operatorname{Jac}{\widetilde{\mathcal{C}}_n })^{(\varepsilon_n)} \times  \operatorname{Kum}( \widetilde{\mathcal{E}}_n \times \widetilde{\mathcal{E}}_n)^{(-\varepsilon_n)} \,, \\
\end{split} 
\eeq
defined over $\mathbb{Q}(t)$, which induce isomorphisms of the $\ell$-adic representations
\beq
 \left\lbrack \widetilde{\Gamma}_n^{(i,j)} \right\rbrack: \quad  \widetilde{\mathrm{T}}^{(i)}_{n, \ell}  \ \overset{\cong}{\longrightarrow} \ \widetilde{\mathrm{T}}^{(j)}_{n, \ell} \,,
\eeq
for $i, j,  n \in \{1, 2, 3 \}$ with $i \not = j$.
\end{theorem}
\begin{proof}
It follows from Remark~\ref{rem:trans_lattices} that the K3 surfaces in Corollary~\ref{thm:special} have Picard rank $19$. Using the families of maps in Equation~(\ref{eqn:diag_statement3b}) defined over a finite field extension of $\mathbb{Q}(t)$, we construct correspondences between the different Kummer surfaces as explained above. Lemma~\ref{lem:twist_away} then allows us to replace the quadratic-twist factors $\pm \varepsilon_n \delta^2_{m, n}$ by $\pm \varepsilon_n$.
\end{proof}
\begin{remark}
For $n=1$ the statement of Theorem~\ref{thm:ladic_2} is precisely the main result of \cite{MR2214473}. In fact, Equations~(\ref{eqn:family1}) recover the families that appeared in \cite{MR2214473}\footnote{\emph{All} surfaces in Theorem~\ref{thm:ladic_2} have been twisted by $(-1)$ compared to \cite{MR2214473}. This is due to a different sign choice in the definition of the twisted Legendre pencil in Equation~(\ref{eqn:X_general}).}. To our knowledge, the other two cases have not appeared in the literature.
\end{remark}
\par In the situation of Corollary~\ref{thm:special4} and~\ref{thm:special5} the K3 surfaces have Picard rank 19 and 18, respectively, and there are isomorphisms of $\mathrm{G}_\mathbb{Q}$-representations, namely
\beqn
 H^2_{\text{\'et}}\Big(\widetilde{\mathcal{X}}_{n, \overline{\mathbb{Q}}}, \mathbb{Q}_\ell \Big) \  \cong \ \widetilde{\mathrm{T}}^{(1)}_{n, \ell} \oplus \mathbb{Q}_\ell (-1)^{\oplus \mu_n} , \qquad
 H^2_{\text{\'et}}\Big( \operatorname{Kum}( \operatorname{Jac}{\widetilde{\mathcal{C}}_n })_{\overline{\mathbb{Q}}}, \mathbb{Q}_\ell \Big) 
  \ \cong \ \widetilde{\mathrm{T}}^{(3)}_{n, \ell} \oplus \mathbb{Q}_\ell (-1)^{\oplus \mu_n} ,
\eeqn
for $n= 4, 5$ and $\mu_4=19$, $\mu_5=18$. Moreover, we have the isomorphisms of $\mathrm{G}_\mathbb{Q}$-representations
\beqn
 H^2_{\text{\'et}}\Big( \operatorname{Kum}( \widetilde{\mathcal{E}}_n \times \widetilde{\mathcal{F}}_n)^{(\varepsilon_n)}_{\overline{\mathbb{Q}}}, \mathbb{Q}_\ell \Big)  \ \cong 
 \  \widetilde{\mathrm{T}}^{(2)}_{n, \ell} \oplus \mathbb{Q}_\ell (-1)^{\oplus \mu_n}.
\eeqn  
Here, $\widetilde{\mathcal{F}}_4 = \widetilde{\mathcal{E}}^\prime_4$ is the two-isogenous elliptic curve such that the j-invariants of $\widetilde{\mathcal{E}}_4$ and $\widetilde{\mathcal{E}}_4^\prime$ satisfy Equation~(\ref{eqn:modular_X_0(2)}), i.e., the equation of the classical modular curve $X_0(2)$. And  $\widetilde{\mathcal{F}}_5$ is the constant elliptic curve with complex multiplication and  $j(\widetilde{\mathcal{F}}_5)=12^3$.  We have the following:
\begin{theorem}
\label{thm:ladic_45}
In Equations~(\ref{eqn:family4}), (\ref{eqn:family5}) let $t \in \mathbb{P}^1 \backslash S_n$  with
\beqn
  S_4 = \{ 0, \pm (\sqrt{2} \pm 1), \pm 1, \pm i,  \infty\} \,, \qquad S_5 = \{ 0, \pm 1, \pm i, \infty \} \,,
\eeqn
and $\varepsilon_4=\frac{1}{t(t^2-1)}$, $\varepsilon_5=  \frac{1}{4t}$. Then, there are explicit algebraic correspondences
\beq
\begin{split}
	\widetilde{\Gamma}_n^{(1,2)} & \ \subset\   \operatorname{Kum}( \widetilde{\mathcal{E}}_n \times \widetilde{\mathcal{F}}_n)^{(\varepsilon_n)} \times \widetilde{\mathcal{X}}_n \,, \\
	\widetilde{\Gamma}_n^{(2,3)} & \ \subset\  \widetilde{\mathcal{X}}_n \times  \operatorname{Kum}( \operatorname{Jac}{\widetilde{\mathcal{C}}_n }) \,, \\
	\widetilde{\Gamma}_n^{(3,1)} & \ \subset\   \operatorname{Kum}( \operatorname{Jac}{\widetilde{\mathcal{C}}_n }) \times  \operatorname{Kum}( \widetilde{\mathcal{E}}_n \times \widetilde{\mathcal{F}}_n)^{(\varepsilon_n)} \,, \\
\end{split} 
\eeq
defined over $\mathbb{Q}(t)$, which induce isomorphisms of the $\ell$-adic representations
\beq
 \left\lbrack \widetilde{\Gamma}_n^{(i,j)} \right\rbrack: \quad  \widetilde{\mathrm{T}}^{(i)}_{n, \ell}  \ \overset{\cong}{\longrightarrow} \ \widetilde{\mathrm{T}}^{(j)}_{n, \ell} \,,
\eeq
for $i, j \in \{1, 2, 3 \}$ with $i \not = j$ and $n = 4, 5$.
\end{theorem}
\begin{proof}
The proof is analogous to the proof of Theorem~\ref{thm:ladic_2}.
\end{proof}
\small \bibliographystyle{amsplain}
\bibliography{references}{}

\begin{bibdiv}
\begin{biblist}

\bib{MR1890996}{article}{
      author={Ahlgren, Scott},
      author={Ono, Ken},
      author={Penniston, David},
       title={Zeta functions of an infinite family of {$K3$} surfaces},
        date={2002},
        ISSN={0002-9327},
     journal={Amer. J. Math.},
      volume={124},
      number={2},
       pages={353\ndash 368},
         url={https://mathscinet.ams.org/mathscinet-getitem?mr=1890996},
      review={\MR{1890996}},
}

\bib{MR884891}{book}{
      author={Andrianov, Anatolij~N.},
       title={Quadratic forms and {H}ecke operators},
      series={Grundlehren der Mathematischen Wissenschaften [Fundamental
  Principles of Mathematical Sciences]},
   publisher={Springer-Verlag, Berlin},
        date={1987},
      volume={286},
        ISBN={3-540-15294-6},
         url={https://mathscinet.ams.org/mathscinet-getitem?mr=884891},
      review={\MR{884891}},
}

\bib{MR0232775}{inproceedings}{
      author={Artin, M.},
       title={The etale topology of schemes},
        date={1968},
   booktitle={Proc. {I}nternat. {C}ongr. {M}ath. ({M}oscow, 1966)},
   publisher={Izdat. ``Mir'', Moscow},
       pages={44\ndash 56},
         url={https://mathscinet.ams.org/mathscinet-getitem?mr=0232775},
      review={\MR{0232775}},
}

\bib{MR1554977}{article}{
      author={Baker, H.~F.},
       title={On a system of differential equations leading to periodic
  functions},
        date={1903},
        ISSN={0001-5962},
     journal={Acta Math.},
      volume={27},
      number={1},
       pages={135\ndash 156},
         url={https://doi.org/10.1007/BF02421301},
      review={\MR{1554977}},
}

\bib{MR1922094}{incollection}{
      author={Barth, Wolf},
       title={Even sets of eight rational curves on a {$K3$}-surface},
        date={2002},
   booktitle={Complex geometry ({G}\"{o}ttingen, 2000)},
   publisher={Springer, Berlin},
       pages={1\ndash 25},
      review={\MR{1922094}},
}

\bib{MR2062673}{book}{
      author={Birkenhake, Christina},
      author={Lange, Herbert},
       title={Complex abelian varieties},
     edition={Second},
      series={Grundlehren der Mathematischen Wissenschaften [Fundamental
  Principles of Mathematical Sciences]},
   publisher={Springer-Verlag, Berlin},
        date={2004},
      volume={302},
        ISBN={3-540-20488-1},
         url={https://doi.org/10.1007/978-3-662-06307-1},
      review={\MR{2062673}},
}

\bib{MR1953527}{article}{
      author={Birkenhake, Christina},
      author={Wilhelm, Hannes},
       title={Humbert surfaces and the {K}ummer plane},
        date={2003},
        ISSN={0002-9947},
     journal={Trans. Amer. Math. Soc.},
      volume={355},
      number={5},
       pages={1819\ndash 1841},
         url={https://mathscinet.ams.org/mathscinet-getitem?mr=1953527},
      review={\MR{1953527}},
}

\bib{MR1505464}{article}{
      author={Bolza, Oskar},
       title={On binary sextics with linear transformations into themselves},
        date={1887},
        ISSN={0002-9327},
     journal={Amer. J. Math.},
      volume={10},
      number={1},
       pages={47\ndash 70},
         url={https://doi.org/10.2307/2369402},
      review={\MR{1505464}},
}

\bib{MR970659}{article}{
      author={Bost, Jean-Beno\^{\i}t},
      author={Mestre, Jean-Fran\c{c}ois},
       title={Moyenne arithm\'{e}tico-g\'{e}om\'{e}trique et p\'{e}riodes des
  courbes de genre {$1$} et {$2$}},
        date={1988},
        ISSN={0224-8999},
     journal={Gaz. Math.},
      number={38},
       pages={36\ndash 64},
         url={https://mathscinet.ams.org/mathscinet-getitem?mr=970659},
      review={\MR{970659}},
}

\bib{MR4099481}{article}{
      author={Braeger, Noah},
      author={Malmendier, Andreas},
      author={Sung, Yih},
       title={Kummer sandwiches and {G}reene-{P}lesser construction},
        date={2020},
        ISSN={0393-0440},
     journal={J. Geom. Phys.},
      volume={154},
       pages={103718, 18},
         url={https://mathscinet.ams.org/mathscinet-getitem?mr=4099481},
      review={\MR{4099481}},
}

\bib{MR1406090}{book}{
      author={Cassels, J. W.~S.},
      author={Flynn, E.~V.},
       title={Prolegomena to a middlebrow arithmetic of curves of genus {$2$}},
      series={London Mathematical Society Lecture Note Series},
   publisher={Cambridge University Press, Cambridge},
        date={1996},
      volume={230},
        ISBN={0-521-48370-0},
         url={https://doi.org/10.1017/CBO9780511526084},
      review={\MR{1406090}},
}

\bib{MR2369941}{article}{
      author={Clingher, Adrian},
      author={Doran, Charles~F.},
       title={Modular invariants for lattice polarized {$K3$} surfaces},
        date={2007},
        ISSN={0026-2285},
     journal={Michigan Math. J.},
      volume={55},
      number={2},
       pages={355\ndash 393},
         url={https://doi.org/10.1307/mmj/1187646999},
      review={\MR{2369941}},
}

\bib{Clingher:2017aa}{article}{
      author={Clingher, Adrian},
      author={Malmendier, Andreas},
       title={On the geometry of (1,2)-polarized {K}ummer surfaces},
        date={201704},
     journal={arXiv:1704.04884 [math.AG]},
      eprint={1704.04884},
         url={https://arxiv.org/pdf/1704.04884},
}

\bib{MR3995925}{article}{
      author={Clingher, Adrian},
      author={Malmendier, Andreas},
       title={Nikulin involutions and the {CHL} string},
        date={2019},
        ISSN={0010-3616},
     journal={Comm. Math. Phys.},
      volume={370},
      number={3},
       pages={959\ndash 994},
         url={https://doi.org/10.1007/s00220-019-03296-9},
      review={\MR{3995925}},
}

\bib{ClingherMalmendier2019}{article}{
      author={Clingher, Adrian},
      author={Malmendier, Andreas},
       title={Normal forms for {K}ummer surfaces},
        date={2019},
     journal={London Mathematical Society Lecture Note Series},
      volume={2},
      number={459},
       pages={107\ndash 162},
         url={https://arxiv.org/pdf/1809.02003},
}

\bib{clingher2020isogenies}{article}{
      author={Clingher, Adrian},
      author={Malmendier, Andreas},
      author={Shaska, Tony},
       title={Geometry of {P}rym varieties for special bielliptic curves of
  genus three and five},
        date={2019, arXiv:1901.09846},
}

\bib{clingher2019isogenies}{article}{
      author={Clingher, Adrian},
      author={Malmendier, Andreas},
      author={Shaska, Tony},
       title={On isogenies among certain abelian surfaces},
        date={2019, arXiv:1901.09846},
     journal={Michigan Math. J.},
}

\bib{MR0296076}{article}{
      author={Deligne, Pierre},
       title={La conjecture de {W}eil pour les surfaces {$K3$}},
        date={1972},
        ISSN={0020-9910},
     journal={Invent. Math.},
      volume={15},
       pages={206\ndash 226},
         url={https://doi.org/10.1007/BF01404126},
      review={\MR{0296076}},
}

\bib{MR1007155}{article}{
      author={Dolgachev, Igor},
      author={Ortland, David},
       title={Point sets in projective spaces and theta functions},
        date={1988},
        ISSN={0303-1179},
     journal={Ast\'{e}risque},
      number={165},
       pages={210 pp. (1989)},
         url={https://mathscinet.ams.org/mathscinet-getitem?mr=1007155},
      review={\MR{1007155}},
}

\bib{MR871067}{book}{
      author={Freitag, E.},
       title={Siegelsche {M}odulfunktionen},
      series={Grundlehren der Mathematischen Wissenschaften [Fundamental
  Principles of Mathematical Sciences]},
   publisher={Springer-Verlag, Berlin},
        date={1983},
      volume={254},
        ISBN={3-540-11661-3},
         url={https://doi.org/10.1007/978-3-642-68649-8},
      review={\MR{871067}},
}

\bib{MR3011784}{article}{
      author={Garbagnati, Alice},
       title={Elliptic {K}3 surfaces with abelian and dihedral groups of
  symplectic automorphisms},
        date={2013},
        ISSN={0092-7872},
     journal={Comm. Algebra},
      volume={41},
      number={2},
       pages={583\ndash 616},
         url={https://mathscinet.ams.org/mathscinet-getitem?mr=3011784},
      review={\MR{3011784}},
}

\bib{MR1182682}{article}{
      author={Gonzalez-Dorrego, Maria~R.},
       title={{$(16,6)$} configurations and geometry of {K}ummer surfaces in
  {${\bf P}^3$}},
        date={1994},
        ISSN={0065-9266},
     journal={Mem. Amer. Math. Soc.},
      volume={107},
      number={512},
       pages={vi+101},
         url={https://doi.org/10.1090/memo/0512},
      review={\MR{1182682}},
}

\bib{MR1438983}{article}{
      author={Gritsenko, V.~A.},
      author={Nikulin, V.~V.},
       title={Igusa modular forms and ``the simplest'' {L}orentzian
  {K}ac-{M}oody algebras},
        date={1996},
        ISSN={0368-8666},
     journal={Mat. Sb.},
      volume={187},
      number={11},
       pages={27\ndash 66},
         url={https://doi.org/10.1070/SM1996v187n11ABEH000171},
      review={\MR{1438983}},
}

\bib{MR1512996}{article}{
      author={Hecke, E.},
       title={Die eindeutige {B}estimmung der {M}odulfunktionen {$q$}-ter
  {S}tufe durch algebraische {E}igenschaften},
        date={1935},
        ISSN={0025-5831},
     journal={Math. Ann.},
      volume={111},
      number={1},
       pages={293\ndash 301},
         url={https://mathscinet.ams.org/mathscinet-getitem?mr=1512996},
      review={\MR{1512996}},
}

\bib{MR0803354}{incollection}{
      author={Hoyt, William~L.},
       title={On surfaces associated with an indefinite ternary lattice},
        date={1985},
   booktitle={Number theory ({N}ew {Y}ork, 1983--84)},
      series={Lecture Notes in Math.},
      volume={1135},
   publisher={Springer, Berlin},
       pages={197\ndash 210},
         url={https://mathscinet.ams.org/mathscinet-getitem?mr=803354},
      review={\MR{803354}},
}

\bib{MR894512}{incollection}{
      author={Hoyt, William~L.},
       title={Notes on elliptic {$K3$} surfaces},
        date={1987},
   booktitle={Number theory ({N}ew {Y}ork, 1984--1985)},
      series={Lecture Notes in Math.},
      volume={1240},
   publisher={Springer, Berlin},
       pages={196\ndash 213},
         url={https://mathscinet.ams.org/mathscinet-getitem?mr=894512},
      review={\MR{894512}},
}

\bib{MR1023921}{incollection}{
      author={Hoyt, William~L.},
       title={Elliptic fiberings of {K}ummer surfaces},
        date={1989},
   booktitle={Number theory ({N}ew {Y}ork, 1985/1988)},
      series={Lecture Notes in Math.},
      volume={1383},
   publisher={Springer, Berlin},
       pages={89\ndash 110},
         url={https://mathscinet.ams.org/mathscinet-getitem?mr=1023921},
      review={\MR{1023921}},
}

\bib{MR1013162}{incollection}{
      author={Hoyt, William~L.},
       title={On twisted {L}egendre equations and {K}ummer surfaces},
        date={1989},
   booktitle={Theta functions---{B}owdoin 1987, {P}art 1 ({B}runswick, {ME},
  1987)},
      series={Proc. Sympos. Pure Math.},
      volume={49},
   publisher={Amer. Math. Soc., Providence, RI},
       pages={695\ndash 707},
         url={https://mathscinet.ams.org/mathscinet-getitem?mr=1013162},
      review={\MR{1013162}},
}

\bib{MR1877757}{inproceedings}{
      author={Hoyt, William~L.},
      author={Schwartz, Charles~F.},
       title={Yoshida surfaces with {P}icard number {$\rho\geq 17$}},
        date={2001},
   booktitle={Proceedings on {M}oonshine and related topics ({M}ontr\'{e}al,
  {QC}, 1999)},
      series={CRM Proc. Lecture Notes},
      volume={30},
   publisher={Amer. Math. Soc., Providence, RI},
       pages={71\ndash 78},
         url={https://mathscinet.ams.org/mathscinet-getitem?mr=1877757},
      review={\MR{1877757}},
}

\bib{MR1097176}{book}{
      author={Hudson, R. W. H.~T.},
       title={Kummer's quartic surface},
      series={Cambridge Mathematical Library},
   publisher={Cambridge University Press, Cambridge},
        date={1990},
        ISBN={0-521-39790-1},
        note={With a foreword by W. Barth, Revised reprint of the 1905
  original},
      review={\MR{1097176}},
}

\bib{Humbert1901}{article}{
      author={Humbert, Georges},
       title={Sur les fonctionnes ab\'eliennes singuli\`eres. {I}, {II},
  {III}},
     journal={J. Math. Pures Appl. serie 5},
       pages={V, 233\ndash 350 (1899); VI, 279\ndash 386 (1900); VII, 97\ndash
  123 (1901)},
}

\bib{MR0141643}{article}{
      author={Igusa, Jun-ichi},
       title={On {S}iegel modular forms of genus two},
        date={1962},
        ISSN={0002-9327},
     journal={Amer. J. Math.},
      volume={84},
       pages={175\ndash 200},
         url={https://doi.org/10.2307/2372812},
      review={\MR{0141643}},
}

\bib{MR0229643}{article}{
      author={Igusa, Jun-ichi},
       title={Modular forms and projective invariants},
        date={1967},
        ISSN={0002-9327},
     journal={Amer. J. Math.},
      volume={89},
       pages={817\ndash 855},
         url={https://doi.org/10.2307/2373243},
      review={\MR{0229643}},
}

\bib{MR527830}{article}{
      author={Igusa, Jun-ichi},
       title={On the ring of modular forms of degree two over {${\bf Z}$}},
        date={1979},
        ISSN={0002-9327},
     journal={Amer. J. Math.},
      volume={101},
      number={1},
       pages={149\ndash 183},
         url={https://doi.org/10.2307/2373943},
      review={\MR{527830}},
}

\bib{MR429915}{article}{
      author={Inose, Hiroshi},
       title={On certain {K}ummer surfaces which can be realized as
  non-singular quartic surfaces in {$P^{3}$}},
        date={1976},
        ISSN={0040-8980},
     journal={J. Fac. Sci. Univ. Tokyo Sect. IA Math.},
      volume={23},
      number={3},
       pages={545\ndash 560},
         url={https://mathscinet.ams.org/mathscinet-getitem?mr=429915},
      review={\MR{429915}},
}

\bib{MR0165541}{article}{
      author={Kodaira, K.},
       title={On the structure of compact complex analytic surfaces. {I},
  {II}},
        date={1963},
        ISSN={0027-8424},
     journal={Proc. Nat. Acad. Sci. U.S.A.},
      volume={50},
       pages={218\ndash 221; ibid. 51 (1963), 1100\ndash 1104},
         url={https://mathscinet.ams.org/mathscinet-getitem?mr=165541},
      review={\MR{165541}},
}

\bib{MR3263663}{article}{
      author={Kumar, Abhinav},
       title={Elliptic fibrations on a generic {J}acobian {K}ummer surface},
        date={2014},
        ISSN={1056-3911},
     journal={J. Algebraic Geom.},
      volume={23},
      number={4},
       pages={599\ndash 667},
         url={https://mathscinet.ams.org/mathscinet-getitem?mr=3263663},
      review={\MR{3263663}},
}

\bib{MR2409557}{incollection}{
      author={Kuwata, Masato},
      author={Shioda, Tetsuji},
       title={Elliptic parameters and defining equations for elliptic
  fibrations on a {K}ummer surface},
        date={2008},
   booktitle={Algebraic geometry in {E}ast {A}sia---{H}anoi 2005},
      series={Adv. Stud. Pure Math.},
      volume={50},
   publisher={Math. Soc. Japan, Tokyo},
       pages={177\ndash 215},
         url={https://doi.org/10.2969/aspm/05010177},
      review={\MR{2409557}},
}

\bib{MR3712162}{article}{
      author={Malmendier, A.},
      author={Shaska, T.},
       title={The {S}atake sextic in {F}-theory},
        date={2017},
        ISSN={0393-0440},
     journal={J. Geom. Phys.},
      volume={120},
       pages={290\ndash 305},
         url={https://doi.org/10.1016/j.geomphys.2017.06.010},
      review={\MR{3712162}},
}

\bib{MR2854198}{article}{
      author={Malmendier, Andreas},
       title={Kummer surfaces associated with {S}eiberg-{W}itten curves},
        date={2012},
        ISSN={0393-0440},
     journal={J. Geom. Phys.},
      volume={62},
      number={1},
       pages={107\ndash 123},
         url={https://mathscinet.ams.org/mathscinet-getitem?mr=2854198},
      review={\MR{2854198}},
}

\bib{MR3366121}{article}{
      author={Malmendier, Andreas},
      author={Morrison, David~R.},
       title={K3 surfaces, modular forms, and non-geometric heterotic
  compactifications},
        date={2015},
        ISSN={0377-9017},
     journal={Lett. Math. Phys.},
      volume={105},
      number={8},
       pages={1085\ndash 1118},
         url={https://doi.org/10.1007/s11005-015-0773-y},
      review={\MR{3366121}},
}

\bib{MR3068406}{article}{
      author={Malmendier, Andreas},
      author={Ono, Ken},
       title={Moonshine for {$M_{24}$} and {D}onaldson invariants of
  {$\mathbb{CP}^2$}},
        date={2012},
        ISSN={1931-4523},
     journal={Commun. Number Theory Phys.},
      volume={6},
      number={4},
       pages={759\ndash 770},
         url={https://mathscinet.ams.org/mathscinet-getitem?mr=3068406},
      review={\MR{3068406}},
}

\bib{MR3731039}{article}{
      author={Malmendier, Andreas},
      author={Shaska, Tony},
       title={A universal genus-two curve from {S}iegel modular forms},
        date={2017},
     journal={SIGMA Symmetry Integrability Geom. Methods Appl.},
      volume={13},
       pages={Paper No. 089, 17},
         url={https://doi.org/10.3842/SIGMA.2017.089},
      review={\MR{3731039}},
}

\bib{MR3992148}{article}{
      author={Malmendier, Andreas},
      author={Sung, Yih},
       title={Counting rational points on {K}ummer surfaces},
        date={2019},
        ISSN={2522-0160},
     journal={Res. Number Theory},
      volume={5},
      number={3},
       pages={Paper No. 27, 23},
         url={https://mathscinet.ams.org/mathscinet-getitem?mr=3992148},
      review={\MR{3992148}},
}

\bib{MR1500768}{article}{
      author={McDonald, John~Hector},
       title={A problem in the reduction of hyperelliptic integrals},
        date={1906},
        ISSN={0002-9947},
     journal={Trans. Amer. Math. Soc.},
      volume={7},
      number={4},
       pages={578\ndash 587},
         url={https://doi.org/10.2307/1986247},
      review={\MR{1500768}},
}

\bib{MR2306633}{article}{
      author={Mehran, Afsaneh},
       title={Double covers of {K}ummer surfaces},
        date={2007},
        ISSN={0025-2611},
     journal={Manuscripta Math.},
      volume={123},
      number={2},
       pages={205\ndash 235},
         url={https://mathscinet.ams.org/mathscinet-getitem?mr=2306633},
      review={\MR{2306633}},
}

\bib{MR1106431}{incollection}{
      author={Mestre, Jean-Fran\c{c}ois},
       title={Construction de courbes de genre {$2$} \`a partir de leurs
  modules},
        date={1991},
   booktitle={Effective methods in algebraic geometry ({C}astiglioncello,
  1990)},
      series={Progr. Math.},
      volume={94},
   publisher={Birkh\"{a}user Boston, Boston, MA},
       pages={313\ndash 334},
      review={\MR{1106431}},
}

\bib{MR728142}{article}{
      author={Morrison, D.~R.},
       title={On {$K3$} surfaces with large {P}icard number},
        date={1984},
        ISSN={0020-9910},
     journal={Invent. Math.},
      volume={75},
      number={1},
       pages={105\ndash 121},
         url={https://mathscinet.ams.org/mathscinet-getitem?mr=728142},
      review={\MR{728142}},
}

\bib{MR2514037}{book}{
      author={Mumford, David},
       title={Abelian varieties},
      series={Tata Institute of Fundamental Research Studies in Mathematics},
   publisher={Published for the Tata Institute of Fundamental Research, Bombay;
  by Hindustan Book Agency, New Delhi},
        date={2008},
      volume={5},
        ISBN={978-81-85931-86-9; 81-85931-86-0},
         url={https://mathscinet.ams.org/mathscinet-getitem?mr=2514037},
        note={With appendices by C. P. Ramanujam and Yuri Manin, Corrected
  reprint of the second (1974) edition},
      review={\MR{2514037}},
}

\bib{MR0429917}{article}{
      author={Nikulin, V.~V.},
       title={Kummer surfaces},
        date={1975},
        ISSN={0373-2436},
     journal={Izv. Akad. Nauk SSSR Ser. Mat.},
      volume={39},
      number={2},
       pages={278\ndash 293, 471},
         url={https://mathscinet.ams.org/mathscinet-getitem?mr=0429917},
      review={\MR{0429917}},
}

\bib{MR1013073}{article}{
      author={Oguiso, Keiji},
       title={On {J}acobian fibrations on the {K}ummer surfaces of the product
  of nonisogenous elliptic curves},
        date={1989},
        ISSN={0025-5645},
     journal={J. Math. Soc. Japan},
      volume={41},
      number={4},
       pages={651\ndash 680},
         url={https://doi.org/10.2969/jmsj/04140651},
      review={\MR{1013073}},
}

\bib{MR1509868}{article}{
      author={Pringsheim, Alfred},
       title={Zur {T}ransformation zweiten {G}rades der hyperelliptischen
  {F}unctionen erster {O}rdnung},
        date={1876},
        ISSN={0025-5831},
     journal={Math. Ann.},
      volume={9},
      number={4},
       pages={445\ndash 475},
         url={https://doi.org/10.1007/BF01442473},
      review={\MR{1509868}},
}

\bib{MR1578135}{article}{
      author={Richelot, Fried.~Jul.},
       title={De transformatione integralium {A}belianorum primi ordinis
  commentatio. {C}aput secundum. {D}e computatione integralium {A}belianorum
  primi ordinis},
        date={1837},
        ISSN={0075-4102},
     journal={J. Reine Angew. Math.},
      volume={16},
       pages={285\ndash 341},
         url={https://mathscinet.ams.org/mathscinet-getitem?mr=1578135},
      review={\MR{1578135}},
}

\bib{Serre85}{article}{
      author={Serre, Jean-Pierre},
       title={Rational points on curves over finite fields},
        date={1985},
     journal={notes by F. Gouvea of lectures at Harvard University},
}

\bib{MR0419459}{incollection}{
      author={Shafarevitch, E.~R.},
       title={Le th\'{e}or\`eme de {T}orelli pour les surfaces alg\'{e}briques
  de type {$K3$}},
        date={1971},
   booktitle={Actes du {C}ongr\`es {I}nternational des {M}ath\'{e}maticiens
  ({N}ice, 1970), {T}ome 1},
       pages={413\ndash 417},
      review={\MR{0419459}},
}

\bib{MR2279280}{article}{
      author={Shioda, Tetsuji},
       title={Kummer sandwich theorem of certain elliptic {$K3$} surfaces},
        date={2006},
        ISSN={0386-2194},
     journal={Proc. Japan Acad. Ser. A Math. Sci.},
      volume={82},
      number={8},
       pages={137\ndash 140},
         url={https://mathscinet.ams.org/mathscinet-getitem?mr=2279280},
      review={\MR{2279280}},
}

\bib{MR2296439}{incollection}{
      author={Shioda, Tetsuji},
       title={Classical {K}ummer surfaces and {M}ordell-{W}eil lattices},
        date={2007},
   booktitle={Algebraic geometry},
      series={Contemp. Math.},
      volume={422},
   publisher={Amer. Math. Soc., Providence, RI},
       pages={213\ndash 221},
         url={https://doi.org/10.1090/conm/422/08062},
      review={\MR{2296439}},
}

\bib{MR3683625}{article}{
      author={Totaro, Burt},
       title={Recent progress on the {T}ate conjecture},
        date={2017},
        ISSN={0273-0979},
     journal={Bull. Amer. Math. Soc. (N.S.)},
      volume={54},
      number={4},
       pages={575\ndash 590},
         url={https://mathscinet.ams.org/mathscinet-getitem?mr=3683625},
      review={\MR{3683625}},
}

\bib{MR669299}{article}{
      author={van~der Geer, G.},
       title={On the geometry of a {S}iegel modular threefold},
        date={1982},
        ISSN={0025-5831},
     journal={Math. Ann.},
      volume={260},
      number={3},
       pages={317\ndash 350},
         url={https://mathscinet.ams.org/mathscinet-getitem?mr=669299},
      review={\MR{669299}},
}

\bib{MR2214473}{article}{
      author={van Geemen, Bert},
      author={Top, Jaap},
       title={An isogeny of {$K3$} surfaces},
        date={2006},
        ISSN={0024-6093},
     journal={Bull. London Math. Soc.},
      volume={38},
      number={2},
       pages={209\ndash 223},
         url={https://doi.org/10.1112/S0024609306018170},
      review={\MR{2214473}},
}

\end{biblist}
\end{bibdiv}
\end{document}